 \documentclass[12pt]{article}
\usepackage{amsmath,amssymb,latexsym, amsfonts, amscd, amsthm}
\usepackage{graphicx,epsfig}

\theoremstyle{plain}
\newtheorem{theorem}{Theorem}[section]

\newtheorem{proposition}[theorem]{Proposition}
\newtheorem{corollary}[theorem]{Corollary}

\newtheorem{lemma}[theorem]{Lemma}

\theoremstyle{definition}
\newtheorem{definition}[theorem]{Definition}
\newtheorem{remark}[theorem]{Remark}

\newcommand{\lra}{\longrightarrow}

\newcommand{\sym}{\mbox{Sym}}

\newcommand{\Spf}{\mbox{Spf}}

\newcommand{\Fil}{\mbox{\rm Fil}}

\newcommand{\GL}{{\rm \mathbf{GL}}}

\newcommand{\Z}{{\mathbb Z}}
\newcommand{\Q}{{\mathbb Q}}
\newcommand{\C}{{\mathbb C}}
\newcommand{\R}{{\mathbb R}}
\newcommand{\F}{{\mathbb F}}
\newcommand{\N}{{\mathbb N}}
\newcommand{\V}{{\mathbb V}}

\newcommand{\cH}{{\cal H}}
\newcommand{\cI}{{\cal I}}
\newcommand{\cJ}{{\cal J}}
\newcommand{\cU}{{\cal U}}
\newcommand{\cP}{{\cal P}}

\newcommand{\cF}{{\cal F}}

\newcommand{\cE}{{\cal E}}

\newcommand{\PP}{{\mathbb P}}

\newcommand{\bV}{{\mathbb V}}

\newcommand{\bW}{{\mathbb W}}
\newcommand{\bE}{{\mathbb E}}

\newcommand{\fw}{{\mathfrak w}}
\newcommand{\bA}{{\mathbb A}}

\newcommand{\cS}{{\cal S}}
\newcommand{\cD}{{\cal D}}
\newcommand{\cO}{{\cal O}}

\newcommand{\rH}{\mathrm{H}}

\newcommand{\fX}{{\mathfrak{X}}}
\newcommand{\fP}{{\mathfrak{P}}}
\newcommand{\cX}{{\mathcal{X}}}
\newcommand{\fIG}{{\mathfrak{IG}}}

\newcommand{\fN}{{\mathfrak{N}}}
\newcommand{\fS}{{\mathfrak{S}}}

\newcommand{\fT}{{\mathfrak{T}}}

\newcommand{\fF}{{\mathfrak{F}}}

\newcommand{\fa}{{\mathfrak{a}}}
\newcommand{\fb}{{\mathfrak{b}}}
\newcommand{\fp}{{\mathfrak{p}}}

\newcommand{\Pic}{{\mathrm{Pic}}}

\newcommand{\mat}[4]{\left( \begin{array}{cc} {\sharp1} & {\sharp2} \\ {\sharp3} & {\sharp4}
\end{array} \right)}

\begin{document}

\bigskip

 \title{Katz type $p$-adic $L$-functions for primes $p$ non-split in the {\rm CM} field. }
\author{by Fabrizio Andreatta,  Adrian Iovita}\maketitle

\begin{abstract}
 For every triple $F$, $K$, $p$ where $F$ is a classical elliptic eigenform, $K$ is a quadratic imaginary field
and $p$ is an odd prime integer which is not split in $K$, we attach a $p$-adic $L$-function which interpolates the algebraic parts of the special values of the complex $L$-functions of 
$F$ twisted by algebraic Hecke characters of $K$ such that the $p$-part of their conductor is $p^n$, with $n$ large enough (for $p\geq 5$ it suffices $n\ge 2$). This construction extends a classical construction of N. Katz, for $F$ an Eisenstein series and of Bertolini-Darmon-Prasanna, for
$F$ a cuspform, when $p$ is split in $K$. Moreover we prove a Kronecker limit formula, respectively $p$-adic Gross-Zagier formulae for our newly    
defined $p$-adic $L$-functions.

\end{abstract}

\tableofcontents

\section{Introduction.}
\label{sec:intro}

 The celebrated article of N.~Katz ``p-Adic Interpolation of Real Analytic Eisenstein Series", 1976, (see \cite{{katzEisenstein}}) attaches a two variables $p$-adic $L$-function to a pair  $(K, p)$ consisting of  a quadratic imaginary  field $K$  and a prime integer $p$, satisfying a number of assumptions of which the most important is that the prime $p$ is split in $K$.   Katz' $p$-adic $L$-function associated to a pair $(K, p)$ as above can be seen to interpolate $p$-adically the central critical values of the complex Rankin $L$-functions of a $p$-adic family of Eisenstein series twisted by a family of certain algebraic Hecke characters.  
This article was very influential and produced, during the last $45$ years, a large number of papers extending these ideas to other similar situations and/or trying to prove properties of these new ``Katz type" $p$-adic $L$-functions.

An example is the work of  Bertolini-Darmon-Prasanna (see \cite{bertolini_darmon_prasana}) who studied Katz type $p$-adic $L$-functions in which the Eisenstein family was replaced by a cuspform. More precisely they defined a one variable anticyclotomic $p$-adic $L$-function attached to an elliptic  cuspidal eigenform $F$ and an imaginary quadratic field, which interpolates  the central critical values of the Rankin $L$-functions of $F$ twisted by anticyclotomic characters of higher infinity type and proved $p$-adic Gross-Zagier formulae for these $p$-adic $L$-functions. They assumed that $p$ was split in the quadratic imaginary field. 

The main result of this article is the construction of one (respectively two)-variable $p$-adic $L$-functions \`a la Bertolini-Darmon-Prasanna, i.e., associated to elliptic cuspidal eigenforms (respectively $p$-adic families of such), or  \`a la Katz, i.e. associated to elliptic Eisenstein series (respectively to $p$-adic families of Eisenstein series), and a pair $(K, p)$ consisting of a quadratic imaginary field and prime integer 
$p>2$ which is not split in $K$. We also prove special values formulae for these new one variable $p$-adic $L$-functions.

To simplify the discussion, in this introduction we focus on the case of the one variable $p$-adic $L$-functions for $p$ inert in $K$ and $p\geq 5$.  For the cases $p=3$ inert in $K$ or $p$ odd and ramified in $K$ and for the two variable $p$-adic $L$-functions  the reader can consult the sections \S \ref{sec:padicLinert} (the inert case) and \S \ref{sec:padicLramified} (the ramified case) respectively.

\medskip
\noindent
{\it Classical $L$-values.}

\medskip
In what follows we fix a classical, elliptic eigenform $F$ of weight $k\ge 1$, level $\Gamma_1(N)$ for $N\ge 5$ and character $\epsilon$ and a quadratic imaginary field $K$. We assume  that there is an ideal $\fN$ of $\cO_K$  such that  we have a ring isomorphism $\cO_K/\fN\cong \Z/N\Z$ (this is the so called Heegner hypothesis); we choose and fix such an ideal.  We also fix a positive, odd integer $c$ prime to $N$.
We consider the $L$-functions $L(F,\chi, s)$  of the complex variable $s$, where $\chi$ varies in the set  $\Sigma_{cc}(k, c,\fN, \epsilon)$ of algebraic Hecke characters  of $K$, of infinity type $(k_1,k_2)$ with $k_1\geq 1$ and $k_1+k_2=k$, of conductor dividing $c\fN$ and character $\epsilon$ (in the sense of \S \ref{sec:grossencharacter}) and such that we are in one of the following two cases:

\smallskip

a) $F=E_{k,\epsilon}$ is an Eisenstein series, $\epsilon$ has parity $k$ and is non-trivial if $k=2$;

\smallskip 

b) $F$ is a cuspform of weight $k\geq 2$ and $\chi$ is such that $s=0$ is the central critical value of  $L(F,\chi, s)$.

\smallskip

Looking at the infinity types of the characters in $\Sigma_{cc}(k, c, \fN, \epsilon)$  we have
a natural decomposition $\Sigma_{cc}(k, c, \fN, \epsilon)=\Sigma_{cc}^{(1)}(k, c, \fN, \epsilon)\amalg \Sigma_{cc}^{(2)}(k, c, \fN, \epsilon)$, where 
 $\Sigma_{cc}^{(1)}(k, c, \fN, \epsilon)$ is the subset of Hecke characters  having infinity type 
$(k-1-j, 1+j)$ with $0\le j\le k-2$ and 
$\Sigma_{cc}^{(2)}(k, c, \fN, \epsilon)$ is the subset of Hecke characters  having infinity type 
$(k+j, -j)$ for $j\ge 0$.

\bigskip

For $F$ as above and $\chi\in \Sigma_{cc}^{(2)}(k, c, \fN, \epsilon)$ we have an explicit description of the algebraic part of the special value of $L(F,\chi,0)$ as follows, thanks to Damerell's theorem in the Eisenstein case and thanks to a result of Waldspurger's in the cuspidal case (assuming that $K$ has odd discriminant):

$$
(\ast)\quad L_{\rm alg}(F,\chi):=\sum_{\fa\in {\rm Pic}(\cO_c)}\chi_j^{-1}(\fa)\delta_k^j(F)\Bigl(\fa\ast (A_0, t_0, \omega_0)  \Bigr)\in\overline{\Q} $$and

$$\bigl(L_{\rm alg}(F,\chi)\bigr)^\iota=C(F,\chi)\frac{L(F,\chi,0)}{\Omega^{\iota (k+2j)}}
,
$$
where $\iota=1$  if $F$ is an Eisenstein series and $\iota=2$ if $F$ is a cuspform, $C(F,\chi)$ is an explicit constant which varies if $F$ is an Eisentein series or a cuspform, $\Omega$ is a complex period, $\cO_c$ is the order 
in $\cO_K$ of conductor $c$, $\chi_j=\chi\cdot {\bf N}^j$, where ${\bf N}$ is the norm Hecke character of $K$,  and $(A_0,t_0, \omega_0)$ is a triple consisting of an elliptic curve with {\rm CM} by $\cO_c$, $t_0$ is a generator of $A_0[\fN]$ and $\omega_0$ is a generator of the invariant differentials 
on $A_0$. Here we say that $A_0$ as CM by $\cO_c$ if we have a ring isomorphism $\iota\colon K\to \mathrm{End}^0(A_0)$ such that $\iota^{-1}\bigl(\mathrm{End}(A_0)\bigr)=\cO_c$. We remark that the isomorphism classes of elliptic curves with CM by $\cO_c$ are points of a Shimura variety. Given any such $A_0$, we abusively refer to the choice of a generator $t_0$ of $A_0[\fN]$ as a $\Gamma_1(\fN)$-level structure on $A_0$, identifying $(A_0,t_0)$ to a point of the above mentioned Shimura variety of level $\Gamma_1(\fN)$.  Furthermore, $\fa\ast (A_0, t_0, \omega_0) $ denotes the natural action of $\fa\in {\rm Pic}(\cO_c) $ on the triples $(A_0, t_0, \omega_0) $.
Finally, one of the most important players in the above formula is the weight $k$ Shimura-Maass differential operator $\delta_k$. This operator is analytically defined on (real analytic) modular forms of weight $k$ for $\Gamma_1(N)$ on the complex upper half plane by the formula: $\displaystyle \delta_k(F):=\frac{1}{2\pi i}\Bigl(\frac{\partial}{\partial z}+\frac{k}{z-\overline{z}} \Bigr) (F)$ and acts on the $q$-expansion of $F$ as the differential operator $\displaystyle q\frac{d}{dq}$.

\bigskip
\noindent
{\it $p$-Adic interpolation of the values $L_{\rm alg}(F, \chi)$.}

\medskip

  As we mentioned at the beginning of this introduction if $F$ is an Eisentein series, respectively a 
cusp form and $p$ is split in $K$,
Katz  in \cite{katzEisenstein} and respectively Bertolini-Darmon-Prasana in \cite{bertolini_darmon_prasana} constructed $p$-adic $L$-functions interpolating $p$-adically the values
$L_{\rm alg}(F, \chi)$ for $\chi\in \Sigma^{(2)}_{cc}(\fN)$.

A legitimate question is then: what about if $p$ is non-split in $K$? When $F=E_{k,\epsilon}$ is an Eisenstein series and $p$ is a  prime integer inert in $K$, the $L$-values $L_{\rm alg}(E_{k,\epsilon},\chi)$ have been extensively studied by Fujiwara in \cite{fujiwara}, Bannai-Kobayashi in \cite{kobayashi_bannai} and by Bannai-Kobayashi-Yasuda in \cite{kobayashi_bannai_yasuda}. One knows, thanks to work of Katz \cite{katz_ICM}, that the values $L_{\rm alg}(E_{k,\epsilon},\chi)$ are $p$-adic integers for $\chi$ algebraic Hecke characters of infinity type $(k+j,-j)$, where $k\geq 2$ and $j\geq 0$ are integers. It was observed by Fujiwara that their $p$-adic valuations tend to infinity as $k$ or $j$ tend to infinity; this was done by carefully computing explicit lower bounds for these valuations. This makes it clear that the special $L$-values cannot be ``naively" $p$-adically interpolated 
and one would need to divide those $L$-values by naturally appearing $p$-adic periods in order to compensate for the growth of the $p$-adic valuations. 

 The main goal of the present article is to define the $p$-adic $L$-functions of Katz and of Bertolini-Darmon-Prasana in  the cases when the prime $p\ge 3$ is either inert or ramified in the quadratic imaginary field  $K$. Our construction is purely geometric and is based on the idea of an analytic continuation of certain overconvergent de Rham classes, which we will try to explain further. 
But most importantly, our geometric constructions naturally produce $p$-adic periods which divide the special classical $L$-values when they are compared with the special values of the $p$-adic $L$-functions. If $\chi$ is a classical Hecke character of $K$ of conductor $p^n$ and infinity type $(k+m,-m)$, the relevant period is denoted by $\Omega_{p,n}^{k+2m}$. We have computed the $p$-adic valuations of these $p$-adic periods in Proposition 5.9 (in the inert case), respectively Proposition 6.6 (in the ramified case), and they seem to agree 
with the bounds of the $p$-adic valuations of the special $L$-values of Fujiwara and Bannai-Kobayashi-Yasuda.

For our construction of the one variable $p$-adic $L$-function, denoted $L_p(F, -)$, \`a la Katz and \`a la Bertolini-Darmon-Prasanna, we {\em assume} that the $p$-part of the conductors of the Hecke characters appearing is  $p^n$ for $n\ge 2$. We let $\Sigma_{cc}^{(2),p^n}(k, c, \fN, \epsilon)\subset \Sigma_{cc}^{(2)}(k, c, \fN, \epsilon)$ be the subspace defined by this condition.  (Here we use our hypothesis that $p\geq 5$ and $p$ is inert. In general, one has a bound $n\geq n_0$ where $n_0$ depends on $p$. We refer the reader to \S \ref{sec:padicLinert} and \S \ref{sec:padicLramified} for a discussion in the cases that $p=3$ or that $p$ is ramified and also in the case of the two variable $p$-adic $L$-functions.)

As in the  $p$-split in $K$ situation (see \cite{bertolini_darmon_prasana}), the function $L_p(F, -)\colon \widehat{\Sigma}_{cc}^{(2),p^n}(k, c, \fN, \epsilon)\lra \C_p$  is a locally analytic function  defined on the space $\widehat{\Sigma}^{(2),p^n}(k, c, \fN, \epsilon)$, the $p$-adic completion of $\Sigma_{cc}^{(2),p^n}(k, c, \fN, \epsilon)$. We prove in  Section \ref{sec:Sigmahat} that the natural map
$w\colon\Sigma_{cc}^{(2),p^n}(k, c, \fN, \epsilon)\to \Z$ sending $\chi$ to $j$, where the infinity type of $\chi$ is $(k+j,-j)$,
is continuous and  extends to a local homeomorphism $w\colon 
\widehat{\Sigma}^{(2),p^n}(k, c, \fN, \epsilon)\to  W(\Q_p)$, where
$W(\Q_p)$ is the space of $\Q_p$-valued $p$-adic weights, i.e., $W(\Q_p)={\rm Hom}_{\rm cont}(\Z_p^\ast,\Z_p^\ast)$.  

For  $\chi\in \widehat{\Sigma}_{cc}^{(2),p^n}(k, c, \fN, \epsilon)$ we put
$\nu:=w(\chi)\in W(\Q_p)$ and we will try to explain how to make sense of the following formula:
$$
``L_p(F, \chi)=\sum_{\fa\in {\rm Pic}(\cO_c)}\chi_\nu^{-1}(\fa)\delta_k^\nu(F)\Bigl(\fa\ast (A_0, t_0, \omega_0)  \Bigr)"
$$Let $X_1(N)$ denote the modular curve parameterizing generalized elliptic curves with $\Gamma_1(N)$-level structure, seen as a rigid analytic curve over $\Q_p$ and let $\bigl({\rm H}, \nabla, {\rm Fil}^\bullet\bigr)$ denote the first relative de Rham cohomology sheaf of the universal generalized elliptic curve over 
$X_1(N)$, denoted ${\rm H}$, with its Gauss-Manin connection $\nabla$ and its Hodge filtration ${\rm Fil}^\bullet$. For every $\alpha\in W(\Q_p)$ we defined  in \cite{andreatta_iovita} a triple $\bigl(\bW_\alpha, \nabla_\alpha, {\rm Fil}^\bullet_\alpha  \bigr)$ consisting of an overconvergent sheaf of Banach modules $\bW_\alpha$ (i.e. this is a sheaf on a strict neighbourhood of the ordinary locus in $X_1(N)$), with a connection 
$\nabla_\alpha$ and a Hodge filtration, interpolating $p$-adically the family $\Bigl({\rm Sym}^m{\rm H}, \nabla_m, {\rm Fil}^\bullet_m   \Bigr)_{m\in \N}$. We also showed that if $\nu\in W(\Q_p)$, the expression
$(\nabla_k)^\nu(F^{[p]})$ is an overconvergent section of $\bW_{k+2\nu}$ which interpolates the natural iterations of $\nabla_k$ on $F^{[p]}$. Finally $F^{[p]}$, the 
$p$-depletion of $F$, is the overconvergent modular form whose $q$-expansion is $\sum_{(n,p)=1}a_nq^n$
if the $q$-expansion of $F$ is $\sum_{n=0}^\infty a_nq^n$. In \S \ref{sec:VBMS} we review and improve the discussion of \cite{andreatta_iovita}  in order to make explicit the neighbourhoods of the ordinary locus where these objects are defined. In particular, we remark that the section $(\nabla_k)^{\nu}\bigl(F^{[p]}\bigr)$ is defined in a neighbourhood  containing the points $\fa\ast (A_0, t_0, \omega_0)$ thanks to our assumption that $n\geq 2$.

Taking this for granted we need to explain how to evaluate the section $(\nabla_k)^{\nu}\bigl(F^{[p]}\bigr)$, mentioned above, at triples $\fa\ast (A_0, t_0, \omega_0)$, for $\fa\in {\rm Pic}(\cO_c)$.
We recall that $(A_0, t_0, \omega_0)$ is a triple consisting of an elliptic curve $A_0$ with $CM$ by
$\cO_c$, $t_0$ is a generator of $A_0[\fN]$ and $\omega_0$ a generator of the invariant differentials of $A_0$.  We can construct $(A_0, t_0, \omega_0)$ as follows. Write $c=dp^n$, with $d$ prime to $p$ and $n\ge 2$. Choose  a triple $(E=E_d, t, \omega)$, where $E$ is an elliptic curve with CM by $\cO_d$, $t$ is a $\Gamma_1(\fN)$-level structure on $E$ and $\omega$ is a generator of the invariant differentials of $E$. When $p$ is inert in $K$,  the elliptic curve $E$, that we view as an elliptic curve over 
an extension of $\Q_p$, has supersingular reduction and moreover it has no canonical subgroup. 
 We choose a subgroup scheme $C^{(n)}\subset E[p^n]$, generically cyclic of order $p^n$ and define
$E^{(n)}:=E/C^{(n)}$ and $E':=E/\bigl(C^{(n)}[p]\bigr)$. Then $E^{(n)}$ is an elliptic curve having a canonical subgroup 
of level $p^n$, namely $H^{(n)}:=E[p^n]/C^{(n)}$, and CM by $\cO_c$ and $E'$ is an elliptic curve with a canonical subgroup of level $p$ and CM by $\cO_{dp}$. The images of $t$ in $E^{(n)}$ and $E'$ define level $\Gamma_1(\fN)$-structures denoted $\Psi_N^{(n)}, \Psi_N'$ respectively and the pull-back of $\omega$ by the dual of the natural isogenies define invariant differentials $\omega^{(n)}$ and $\omega'$ of $E^{(n)}$ and $E'$ respectively. Therefore, the
pairs $x:=\bigl(E^{(n)}, \Psi_N^{(n)}\bigr)$ and $x':=\bigl(E', \Psi_N'\bigr)$ define points of the modular curve $X_1(N)$ and moreover for every $\fa\in {\rm Pic}(\cO_c)$ the 
pairs $x_\fa:=\fa\ast (E^{(n)}, \Psi_N^{(n)})$ and respectively $x'_\fa:=\fa\ast\big(E', \Psi_N'  \bigr)$
define points of $X_1(N)$ with the property: $(\nabla_k)^\nu(F^{[p]})$ is defined at 
$x_\fa$ and the image of the pull-back of $(\nabla_k)^\nu(F^{[p]})_{x_\fa}$ under the morphism induced by the isogeny
$E^{(n)}\to E'$, is a section of a submodule of $\bigl(\bW_{k+2\nu}\bigr)_{x'_\fa}$, for which the Hodge-filtration can be split using the CM action of $\cO_{pd}$. Its image under the splitting is a section of $\bigl(\fw^{k+2\nu}\bigr)_{x'_{\fa}}$. Using the generator
$(\fa \ast \omega')^{k+2\nu}$ of $\bigl(\fw^{k+2\nu}\bigr)_{x'_{\fa}}$ we obtain  an element of $\C_p$, which is
the looked for value: $(\nabla_k)^\nu(F^{[p]})(\fa\ast(A_0, t_0, \omega_0))$.

The construction presented above has the following property: if $\chi\in \Sigma_{cc}^{(2),p^n}(k, c, \fN, \epsilon)\subset \widehat{\Sigma}_{cc}^{(2),p^n}(k, c, \fN, \epsilon)$ is an algebraic Hecke character so that $w(\chi)=\nu=m\in \N$, choosing an isomorphism $\C_p\cong \C$,  we have
$$
(\nabla_k)^m(F^{[p]})\bigl(\fa\ast(A_0, t_0, \omega_0)\bigr)=\delta_k^m(F^{[p]})(\fa\ast (A_0, t_0, \omega_0)).
$$
Therefore, as a first approximation, we could define for $\chi\in \widehat{\Sigma}_{cc}^{(2)}(k, c, \fN, \epsilon)$ with $w(\chi)=\nu\in W(\Q_p)$:
$$
L_p(F, \chi):=\sum_{\fa\in {\rm Pic}(\cO_c)}\chi_\nu^{-1}(\fa)\bigl(\nabla_k\bigr)^\nu(F^{[p]})\bigl(\fa\ast(A_0, t_0, \omega_0)  \bigr).
$$
We need to remark that, as described, the construction depends on various choices so that in order  to remove this dependence we sum over the group $\cH(c,\fN)$, instead of ${\rm Pic}(\cO_c)$, defined by the invertible $\cO_c$-ideals co-prime to $\fN\cap \cO_c$ (see Definition \ref{def:cH}). Furthermore, as explained in Corollary \ref{cor:classicaldepletion}, the $p$-adic modular forms  $F^{[p]}$ can be realized as a classical modular form over  $X(N,p^2)$, the modular curve 
of level $\Gamma_1(N)\cap \Gamma_0(p^2)$. For this reason  we move the whole construction to  $X(N,p^2)$.  The reader is invited to see  Definition \ref{def:Lpinert} in the inert case and Definition \ref{def:Lpramified} in the ramified case for the precise formulas.

\bigskip
\noindent
{\it Interpolation properties.}

\medskip

As before, will only illustrate the interpolation properties for the one variable $p$-adic $L$-function attached to a classical eigenform $F$ of level $\Gamma_1(N)$, weight $k\ge 2$ and character $\epsilon$, in the case $p\geq 5$ is inert in $K$; a discussion of the ramified case can be found in section \S \ref{sec:padicLramified}.
Let $\chi\in \Sigma^{(2), p^n}_{cc}(k, c, \fN, \epsilon)$, i.e. $\chi$ is an algebraic Hecke character with $w(\chi)=m\ge 0$ and the $p$-part of its conductor is $p^n$ with  $n\ge 2$. We denote by $\Omega_{p,n}$ the appropriate $p$-adic period (see the remark above).  Then we have the following relationship between the values of the $p$-adic $L$ function and of the complex $L$-function at $\chi$ (see Proposition \ref{prop:comparisonclassicalinert}):
    
$$L_p(F,\chi)=\frac{ L_{\rm alg}(F,\chi)}{\Omega_{p,n}^{k+2m}}.$$

\bigskip

\noindent
{\it Special values $L_p(F, \chi)$, for $\chi\in \Sigma_{cc}^{(1)}(\fN)$.}

\medskip

We will continue to illustrate this result in the case $p\ge 5$ is inert in $K$ and $F$ is a cuspidal eigenform of even weight $k\ge 2$ and leave the reader to look at \S \ref{sec:GrossZagier} for the case $p=3$ or $p$ ramified in $K$. We also refer to Proposition \ref{prop:Kronecker} for the case $F=E_{2,\epsilon}$ of an Eisenstein series of weight $2$ and character $\epsilon$ where one gets a $p$-adic analogue of Kronecker limit formula.

We first remark that for $\chi\in \Sigma_{cc}^{(1)}(\fN)$, as the sign of the functional equation of $L(F,\chi^{-1},s)$ is $-1$, we have $L_{\rm alg}(F,\chi^{-1})=0$. The value
$L_p(F,\chi^{-1})$ does not interpolate classical $L$-values. In fact, it is  obtained in terms of the Abel-Jacobi image of a certain generalized Heegner cycle defined in \cite{bertolini_darmon_prasana} and will be called, as in \cite{bertolini_darmon_prasana}, a $p$-adic Gross-Zagier formula.

We let $(A, t_A)$ be an elliptic curve with CM by $\cO_K$ and $\Gamma_1(\fN)$-level structure, let $\omega_A$ denote a generator of the invariant differentials of $A$ and $\eta_A$ an element of ${\rm H}_{\rm dR}^1(A)$ such  that $\langle \omega_A,\eta_A\rangle =1$ via the Poincar\'e pairing. Morever, $\varphi_0\colon A\to A_0$ is a cyclic isogeny of degree $c$ so that $A_0$ has CM by $\cO_c$. For every $\fa\in \Pic(\cO_c)$ we have an isogeny $\varphi_\fa\colon A_0 \to \fa\ast A_0$ and  $ \mathrm{AJ}_p(\Delta_{\varphi_{\fa} \varphi_0})$ denotes the $p$-adic Abel-Jacobi map of the  generalised Heegner cycle denoted $\Delta_{\varphi_{\fa}\varphi_0}$ constructed in  \cite[\S 2]{bertolini_darmon_prasana}
supported in the fiber of the modular point $\varphi_\fa \circ \varphi_0\colon A \to \fa\ast A_0$.

Let $\chi\in \Sigma_{cc}^{(1)}(k, c,\fN,\epsilon)$ be a character of infinity type $(k-1-j,1+j)$ with $0\leq j \leq r:=k-2$ and such that the $p$-power of the conductor  is $p^n$ with $n\ge 2$. Then $\chi$ can be  seen as a $p$-adic character of $\widehat{\Sigma}^{(2), p^n}(k, c, \fN, \epsilon)$, i.e., we can evaluate the $p$-adic $L$-functions on it.  We have 

$$L_p(F,\chi)=\frac{c^{-j}\Omega_{p,n}^{r-2j}}{j!} \left( \sum_{\fa\in \Pic(\cO_c)}  \frac{ c^{-j}}{j!} \chi^{-1}(\fa) {\bf N}(\fa) \mathrm{AJ}_L(\Delta_{\varphi_{\fa} \varphi_0})\bigl(\omega_f\wedge \omega_A^j \wedge \eta_A^{r-j}\bigr)\right).$$
Notice that, as in the split case, the formula does not involve the derivative of the $p$-adic $L$-function but its value.

\bigskip
\noindent
We remark that Daniel Kriz has a different construction of $p$-adic $L$-functions in the cases in which $p$ is not split in $K$ by defining a $p$-adic analogue of the Shimura-Maass operator $\delta_k$ on the functions of the perfectoid tower of modular curves of level $Np^\infty$, more precisely on the functions on the supersingular locus of that perfectoid tower. See \cite{kriz}. Our work has been inspired by Kriz's report on this construction. So far no comparison between the two constructions is available but we hope to report 
on on such a result soon.

We'd also like to remark that the recent progress on the Iwasawa theory of supersingular elliptic curves in articles like 
\cite{burungale_ota_kobayashi_1}, \cite{burungale_ota_kobayashi_2}, \cite{ciperiani} gives hope that it would be soon possible to 
understand Iwasawa theoretic properties of the $p$-adic $L$-functions defined in this article.

\medskip
\noindent{\bf Acknowledgements.} We are grateful to Henri Darmon for encouraging us to think about this topic and for  
many useful discussions pertaining to this matter. We thank the anonymous referee for the very careful reading of the article and for many comments which hopefully lead to its improvement.

\section{Classical $L$-values.}
\label{sec:classical}

\bigskip

\subsection{Algebraic Hecke characters.}\label{sec:grossencharacter} 

We fix a quadratic imaginary field $K$ of discriminant $D_K$.
We start by choosing embeddings $\iota_{\infty}\colon K\hookrightarrow \C$ and $\iota_p\colon K\hookrightarrow \C_p$ and an isomorphism $\eta\colon \C_p\cong \C$ such that $\eta\circ \iota_p=\iota_\infty$.

Recall that a Hecke character of $K$ is a continuous homomorphism $\chi\colon \bA_K^\ast/K^\ast\lra \C^\ast$, where $\bA_K^\ast$ is the group of id\`eles of $K$ and $K^\ast$ is embedded diagonally in $\bA_K^\ast$. We say that $\chi$ is {\sl algebraic}  or {\it a grossencharacter of  $K$ of type $A_0$} if the restriction of $\chi$ to the infinity component $\chi_\infty\colon (K\otimes_{\Z}\R)^\ast \to \C^\ast$ of $\bA_K^\ast$ has the form $\chi_{\infty}(z\otimes 1)=\iota_\infty(z)^{n}\iota_\infty(\overline{z})^{m}$ for a pair of integers $(n,m)$, called the {\sl infinity type} of $\chi$. 

Associated to $\chi$ we have an ideal $\fF$ of $\cO_K$, the conductor, and a finite character $\nu\colon (\cO_K/\fF)^\ast\to \C^\ast$. Throughout this paper we assume that

\begin{itemize}

\item[i.]  there exists an ideal $\fN\subset \cO_K$ such that ${\rm N}_{K/\Q}(\fN)=N\Z$ and we have a ring isomorphism  $\cO_K/\fN\cong \Z/N\Z$  ({\it Heegner assumption});

\item[ii.] there exist a character $\epsilon\colon \bigl(\Z/N\Z\bigr)^\ast\to \C^\ast$ and a positive integer $c$ prime to $N$ having the following properties. Write $\epsilon_{\fN}$ for the character of $(\cO_K/\fN)^\ast$ induced by $\epsilon$ via the identification $\cO_K/\fN \cong \Z/N\Z$. Then, the conductor   $\fF$ divides $c \fN$ and denoting by $\cO_c=\Z+c \cO_K \subset \cO_K$ the order of conductor $c$ and letting $\fN_c:=\fN\cap \cO_c$, the restriction of $\nu$ to $\widehat{\cO}_c^\ast$ factors through  $\bigl(\cO_c/\fN_c\cO_c\bigr)^\ast$ and the induced character of  $\bigl(\cO_c/\fN_c\bigr)^\ast$ coincides with the character $\epsilon_{\fN_c}$, the composite of  $\epsilon_\fN$ and  the isomorphism $\bigl(\cO_c/\fN_c\cO_c\bigr)^\ast \cong \bigl(\cO_K/\fN\cO_K)^\ast$.

\end{itemize}

Under these assumptions we say that $\chi$ is of type $(c,\fN,\epsilon)$. 

\

There is a  group associated to characters of type $(c,\fN,\epsilon)$ that we now define. We denote by:

\begin{itemize}

\item[1.]  $\bA_K^{(\fN c), \ast}$ the subgroup of the id\`eles $\bA_K^{\ast}$ whose components at places dividing $\fN c$ lie in $\prod_{\cP\vert \fN c} \cO_{K_\cP}^\ast$. We write $\bA_K^{(\fN c), f, \ast}$ for the the subgroup of  $\bA_K^{(\fN c), \ast}$ of finite id\`eles;

\item[2.] $H^{(\fN c)} \subset \widehat{\cO}_K^\ast =\prod_p (\cO_K \otimes \Z_p)^\ast\subset \bA_K^{(\fN c), f, \ast} $ the compact open subgroup consisting of elements whose components in $\prod_{\cP\vert \fN} \cO_{K_\cP}^\ast$ are  congruent to $1$ modulo $\fN$ and whose components in $\prod_{\cP\vert c} \cO_{K_\cP}^\ast$ are  in $\prod_{p\vert c} (\cO_c\otimes \Z_p)^\ast$;

\item[3.] $K^{(\fN c)}:=K^\ast \cap \bA_K^{(\fN c), \ast}$.

\end{itemize}

Notice that $ \bA_K^{(\fN c), \ast}/K^{(\fN c)} \cong  \bA_K^{\ast}/K^\ast$ so that to give a Hecke character of type $(c,\fN,\epsilon)$ is equivalent to give a continuous group homomorphism $$\chi\colon \bA_K^{(\fN c), \ast}/K^{(\fN c)} \lra \C^\ast$$such that its restriction to the image of  $ \widehat{\cO}_K^\ast $ in $\bA_K^{(\fN c), \ast}/K^{(\fN c)}$ factors through  the quotient $$ \widehat{\cO}_K^\ast /H^{\fN c}\cong \bigl( \widehat{\cO}_K^\ast / \widehat{\cO}_c^\ast\bigr)  \times \bigl(\cO_K/\fN\bigr)^\ast $$and the restriction to  $\bigl(\cO_K/\fN\bigr)^\ast$ coincides with the character $\epsilon_\fN$.

\begin{definition}\label{def:cH} Define the group ${\cH}(c,\fN):=K^{(\fN c)}\backslash \bA_K^{(\fN c), f, \ast}/H^{(\fN c)}$.
\end{definition} 

As usual we have an interpretation of this group in terms of fractional ideals for the order $\cO_c$. We denote by  $I(c,\fN_c)$ the group  of  fractional invertible $\cO_c$-ideals  generated by the integral invertible ideals of $\cO_c$ prime to $\fN_c$. Let $P(c,\fN_c)$ be the subgroup of $I(c,\fN_c)$ generated by  principal integral ideals $(\alpha)$ of $\cO_c$ with $\alpha\equiv 1$ modulo  $ \fN_c$. Recall the following:

\begin{lemma}\label{lemma:cHcfN} The group ${\cH}(c,\fN)$ is finite and   ${\cH}(c,\fN) \cong I(c,\fN_c)/P(c,\fN_c) $. In particular, for $N=1$ we have ${\cH}(c,1) \cong \Pic(\cO_c)$ the group of invertible fractional $\cO_c$-ideals. For general $N$ the natural map ${\cH}(c,\fN) \to {\cH}(c,1)$ is surjective with kernel isomorphic to $(\cO_K/\fN)^\ast$.  
\end {lemma}
\begin{proof} 
The group ${\cH}(c,\fN)$ coincides with the quotient $K^\ast\backslash \bA_K^{f,\ast}/H^{(\fN c)} $ which is isomorphic to the group $K^\ast\backslash \bA_K^{\ast}/(H^{(\fN c)} \times \C^\ast)$. This is a quotient of the ray class group of modulus $c \fN$ and, in particular, it is finite. 

Let  $I'(c,\fN_c)\subset I(c,\fN_c)$   be the  subgroup  of fractional invertible  $\cO_c$-ideals generated by the integral invertible ideals of $\cO_c$ prime to $c \fN_c$. Let $P'(c,\fN_c)$ be the subgroup of $I'(c,\fN_c)$ generated by  principal integral ideals $(\alpha)$ of $\cO_c$ with $\alpha\in \prod_{p\vert c} (\cO_c\otimes \Z_p)^\ast$ and $\alpha\equiv 1$ modulo  $ \fN_c$.   Then $P'(c,\fN_c)=I'(c,\fN_c)\cap P(c,\fN_c)$ and $I'(c,\fN_c)/P'(c,\fN_c)\cong  I(c,\fN_c)/P(c,\fN_c) $.  

The map $I \mapsto I \cO_K$ defines a group isomorphism $I'(c,\fN_c)\cong I_{\cO_K}(c\fN)$   to the group of fractional $\cO_K$-ideals coprime to $c \fN$. The inverse sends an ideal $J \subset \cO_K$, coprime to  $c \fN$, to the $\cO_c$-ideal $J_c=J\cap \cO_c$ which is an invertible ideal of $\cO_c$ prime to $c\fN_c$. Such isomorphism sends $P'(c,\fN)$ to the subgroup $P_{\cO_K}(c,\fN)\subset  I_{\cO_K}(c\fN)$ generated by principal ideals $(\alpha)\subset \cO_K$  with $\alpha\in \prod_{p\vert c} (\cO_c\otimes \Z_p)^\ast$ and $\alpha\equiv 1$ modulo  $ \fN$.  Then  $I_{\cO_K}(c\fN)/P_{\cO_K}(c,\fN)$ is isomorphic to ${\cH}(c,\fN)$. The other statements are clear.

\end{proof}

Using the Lemma to give a Hecke character of type $(c,\fN,\epsilon)$ and  infinity type $(n,m)$ is  then equivalent to give a character of the group $I(c,\fN)$ such that for every $(\alpha)\in P(c,\cO_K)$, considering the class $\widetilde{\alpha}$ of $\alpha$ in $(\cO_K/\fN)^\ast $ we have $\chi\bigl((\alpha)\bigr)=\epsilon_\fN(\widetilde{\alpha})\iota_\infty(\alpha)^{-n} \iota_\infty(\overline{\alpha})^{-m}$. We will freely use these two points of view.

\subsection{Algebraic Hecke characters and their $p$-adic avatars.}\label{sec:avatar} 

Given an algebraic Hecke character $\chi\colon \bA_K^\ast/K^\ast\lra \C^\ast$  of infinity type $(n,m)$ we denote by $\alpha_p$ the character $\alpha_p\colon \bA_K^{\ast,f}\to \C_p^\ast$ on the finite id\`eles of $K$ which is trivial on all components different from $p$ and such that $\alpha_p\colon (K\otimes_{\Z}\Q_p)^\ast\to \C_p^\ast$ is the continuous character $\alpha(z\otimes 1)=\iota_p(z)^{n}\iota_p(\overline{z})^{m}$.

Every algebraic Hecke character $\chi$ as above has a {\it $p$-adic avatar}, namely    the character $$\chi \cdot \alpha_p \colon \bA_K^{\ast,f} /K^\ast\to \C_p^\ast.$$
It is continuous for the natural topology on $\bA_K^{\ast,f}$ and the $p$-adic topology on $\C^\ast_p$. 


\subsection{Classical values of $L(F, \chi, s)$.}\label{sec:classical}

We fix the following notation.  Set $A_0$ to be the elliptic curve $A_0(\C):=\C/\iota(\cO_c)$ and let $\omega_0$ be a N\'eron differential. Choose a generator $t\in \fN_c^{-1}/\cO_c$. Let $t_0$ be the corresponding  generator of $A_0[\fN]$. It defines a $\Gamma_1(\fN)$-level  structure on $A_0$ and the isomorphism class of the pair $(A_0,t_0)$ determines a point in  $X_1(N)(\C)$. For $\fa\in \Pic(\cO_c)$ coprime to $\fN\cap \cO_c$, write  $\fa\ast(A_0, t_0, \omega_0):=(A,t_\fa,\omega)$ with $A=A_0/A_0[\fa]$, $t_\fa$  the image of $t$ and $\omega$ such that its pull-back via the isogeny $\phi_\fa\colon A_0\to A$ is $\omega_0$.

\subsubsection{The case $F=E_{k,\epsilon}$ is an Eisenstein series.}
\label{sec:eisenstein}

We start by reviewing some of the results of Katz \cite{katzEisenstein}. Consider a fractional $\cO_K$-ideal $M$ of $K$ and an isomorphism $\alpha\colon (\Z/N\Z)^2\cong \frac{1}{N} M/M$.  Let $E_M$ be the elliptic curve $\C/M$ with full $\Gamma(N)$-level structure $\beta_\alpha\colon \mu_N \times \Z/N\Z \cong E[N]$ defined by $\alpha$ (we identify $\Z/N\Z\cong \mu_N$ via the choice of the primitive root of unity $\zeta_N$ given by the Weil pairing of $\alpha(1,0)$ and $\alpha(0,1)$). Let $\omega$ be a generator of the differentials of the N\`eron model of $E_M$ over a number field $L$.  We assume that $L$ contains the $N$-th roots of unity and we let $W$ be the completion of $\cO_L$ at a prime above $p$. Fix an $\cO_L$-valued function $\gamma\colon  (\Z/N\Z)^2\to \cO_L$.

For any positive integer $k\geq 1$, assuming that $\sum_j \gamma(0,j)=\sum_j  \gamma(j,0)=0$ if $k=2$, Katz defines in  \cite[Def. 3.6.5 \& Thm.~3.6.9]{katzEisenstein} the Eisenstein series $G_{k,0,\gamma}$ of weight $k$ and full level $\Gamma(N)$ over $L$. For every $\ell \geq 0$ set

$$L_{\rm alg}(\gamma,k+\ell,-\ell; M,\alpha):= \delta_k^\ell\bigl(2G_{k,0,\gamma}\bigr) (E_M,\beta_\alpha,\omega).$$It is proven in \cite[Cor.~4.1.3]{katzEisenstein} that  this is an algebraic integer and in fact lies in $L$ (Damerell's theorem). In loc. cit., one finds the Weil operator  instead of the Shimura-Maass operator $\delta_k$ in the previous formula  but it follows from \cite[(2.3.38)]{katzEisenstein} that the two operators coincide. In particular, $N^\ell \delta_k^\ell\bigl(2G_{k,0,\gamma}\bigr)=2 G_{k+2\ell,-\ell, \gamma}$ by \cite[(3.6.8)]{katzEisenstein} and one finds in  \cite[(8.6.8)]{katzEisenstein}  the explicit formula 

\begin{equation} \label{eq:explicitLvalue} L_{\rm alg}(\gamma,k+\ell,-\ell; M,\alpha):= \frac{(-1)^{k} (k+\ell-1)! N^{k} \pi^{\ell}}{a(M)^{\ell} \Omega^{k+2\ell}} \left(\sum_{0 \neq m\in M} \frac{g(m) \overline{m}^\ell}{m^{k+\ell} {\bf N}(m)^s}\right)\vert_{s=0},\end{equation}

\noindent
where  $a(M)$ is the co-volume of $M\subset \C$, ${\bf N}$ is the norm map on $K$,  $\Omega$ is a complex period associated to $\omega$ by the equality $\omega=\Omega dz$ (for $z$ the coordinate $\C$ via the uniformization $E(\C)=\C/M$) and $g$ is a function on $M/NM$ described explicitly in \cite[\S 8.7]{katzEisenstein} associated to $\gamma$. We will compute later the function $g$  in the case that $\gamma$ is associated to a character showing that, up to a Gauss sum, the associated $g$ is itself a character so that formula (\ref{eq:explicitLvalue}) will provide the link to special values of Hecke $L$-series.

\

\noindent
{\bf $L$-functions of Hecke grossencharacters.}

\medskip

Next we explain, following  \cite[\S 9.4]{katzEisenstein}, how to express the value at $0$ of the $L$-function associated to an algebraic Hecke character $\chi$ as in \S\ref{sec:grossencharacter} as a combination of  special values of real analytic Eisenstein series (of full level $\Gamma(N)$).

Denote by $\Sigma_{cc}(k,c,\fN,\epsilon)$ (respectively $\Sigma^{(2)}_{cc}(k,c,\fN,\epsilon)$) the set of Hecke characters of type $(c,\fN,\epsilon)$, according to the definition in \S\ref{sec:grossencharacter},   such that  the infinity type  is $(k+j,-j)$ with $k+j\geq 1$ (respectively with $j\geq 0$) and 

\begin{itemize}

\item[i.] $k\geq 1$ and $\epsilon$ is of parity $k$, i.e.,  that $\epsilon(-x)=(-1)^k \epsilon(x)$ (respectively and $j\geq 0$).

\item[ii.] $\epsilon$ is non-trivial if $k=2$.

\end{itemize}

Let $\fa_1,\ldots,\fa_h$ be integral invertible ideals of $\cO_K$, coprime to $c\fN$, such that $\fa_1\cap  \cO_c,\ldots,\fa_h\cap \cO_c$ are representatives of the class group of $\cO_c$. Given $\chi\in \Sigma^{(2)}_{cc}(k,c,\fN,\epsilon)$ define $L(K,\chi)$ to be the value at $0$ of the $L$-function attached to $K$ and the grossencharacter $\chi$ as in \cite[9.4.31]{katzEisenstein}: 
$$L(K,\chi)=\frac{1}{\vert \cO_c^\ast\vert}\left(\sum_{i=1}^h \frac{\chi(\fa_i)^{-1}}{{\bf N}(\fa_i)^{-s}} \sum_{0 \neq \alpha\in \fa_i\cap \cO_c} \frac{\chi\bigl((\alpha)\bigr)}{{\bf N}(\alpha)^s}\right)\vert_{s=0}.$$Here we use the conventions of loc.~cit.~and we view $\chi$ as a classical Hecke character as recalled in Section \ref{sec:grossencharacter}. Define $\gamma_{\epsilon}\colon  (\Z/N\Z)^2\to \cO_L$ by $ \gamma_{\epsilon}(m,n)=\epsilon(m)$. It follows from \cite[Thm. 3.6.9]{katzEisenstein} that  $$2G_{k,0,\gamma_{\epsilon}}=E_{k,\epsilon}=L(1-k,\epsilon) +2\sum_{n\geq 1} \sigma_{k-1,\epsilon}(n) q^n, \qquad\sigma_{k-1,\epsilon}(n) =\sum_{d\vert n} \epsilon(d) d^{k-1}$$is an Eisenstein series of weight $k$, level $\Gamma_1(N)$ and character $\epsilon$. The function $g_\epsilon\colon (\Z/N\Z)^2\to \cO_L $ associated to $\gamma_\epsilon$ is $g_\epsilon(n,m):=\frac{1}{N} \sum_{a\in \Z/N\Z} \epsilon(a) \zeta_N^{-na}$ (see \cite[(3.2.2)]{katzEisenstein}). This coincides with $$g_\epsilon(n,m)=\epsilon^{-1}(n) \frac{s(\epsilon)}{N}, \qquad s(\epsilon)=\sum_{a\in (\Z/N\Z)^\ast} \epsilon(a) \zeta_N^{-a},$$where $s(\epsilon)$ is simply the Gauss sum associated to the character $\epsilon$.  Using the notation at the beginning of \S \ref{sec:classical} we set

$$
L_{\rm alg}\bigl(E_{k,\epsilon}, \chi\bigr):=\sum_{\fa\in {\rm Pic}(\cO_c)}\chi_j^{-1}(\fa)\delta_k^j(E_{k,\epsilon})\bigl(\fa\ast (A_0,t_0,\omega_0)\bigr),
$$
where $\chi_j=\chi\cdot {\bf N}^j$. Put $a(\cO_c)$ to be the co-volume of $\iota\bigl(\cO_c\bigr)\subset \C$ and write $\Omega$ for the complex period defined by $\omega_0=\Omega dz$.  Then:

\begin{proposition} Let $\chi\in \Sigma^{(2)}_{cc}(k,c,\fN,\epsilon)$ be of infinity type $(k+j,-j)$.  We have  $$L\bigl(K,\chi\bigr)= \frac{1}{\vert \cO_c^\ast\vert}   \frac{(-1)^{k} (k+j-1)! N^{k+1} \pi^{j}}{ s(\epsilon) a(\cO_c)^{j}\Omega^{k+2j}}   L_{\rm alg}\bigl(E_{k,\epsilon}, \chi\bigr).$$

\end{proposition}
\begin{proof} This follows from (\ref{eq:explicitLvalue}) using that $\fa\ast A_0$ is the elliptic curve associated to $E_M=\C/\iota(\fa^{-1})$ with $M=\fa$, taking $\fa=\fa_1\cap \cO_c,\ldots,\fa_h\cap \cO_c$ and $\fa^{-1}$ to be the inverse as an invertible $\cO_c$-module.

\end{proof}

\subsubsection{The case $F=f$ is a cuspform.}
\label{sec:cuspform}

Let $f$ be a normalized cuspidal newform of weight $k$, of level $\Gamma_1(N)$  and nebentype $\epsilon$.
Following Definition 4.4 in \cite{bertolini_darmon_prasana} we denote by $\Sigma_{cc}(k,c,\fN,\epsilon)$ (resp.~$\Sigma_{cc}^{(2)}(k,c,\fN,\epsilon)$)  the set of algebraic Hecke characters $\chi$ of type $(c, \fN, \epsilon)$,  as in \S \ref{sec:grossencharacter}, satisfying the following:

\begin{itemize}

\item[1.] $c$ is odd;

\item[2.] the  infinity type is $(k+j, -j)$, with $k\geq 2$ and $k+j\geq 1$ (resp.~$k\geq 2$ and $j\geq 0$);

\item[3.]  they are central critical for $f$, i.e.,  we have 
$\epsilon_\chi=\epsilon$ or equivalently $\chi\vert_{\bA_\Q^\ast}=\epsilon {\bf N}^k$  where ${\bf N}$ is the norm character;

\end{itemize}

If the $q$-expansion of $f$ is $f(q)=\sum_{n=1}^\infty a_nq^n$, the $L$-function of $f$ can be written 
$$
L(f,s)=\sum_{n=1}^\infty a_n n^{-s}=\prod_{\ell}(1-\alpha_\ell \ell^{-s})^{-1}(1-\beta_\ell\ell^{-s})^{-1},
$$
where the product is over all positive prime integers $\ell$. The above equality defines the pairs $(\alpha_\ell, \beta_\ell)$, for all $\ell$.
Every algebraic Hecke character $\chi\in \Sigma^{(2)}_{cc}(\fN)$ is central critical for $f$, i.e. $s=0$ is a critical point for the $L$-function $L(f,\chi^{-1}, s)$.   With the notation at the beginning of \S \ref{sec:classical}, we define $\Omega$ to be the complex period of \cite[(5.1.16)]{bertolini_darmon_prasana}, characterized  by the property that $\omega_0=\Omega \cdot 2\pi i dz$ with $z$ the standard coordinate on $\C$ via the uniformization $A_0(\C)=\C/\iota(\cO_c)$. We have the following explicit form of Waldspurger's formula:

\begin{theorem} 
\label{theorem:classical} {\rm \cite[Thm. 5.4 \& Thm. 5.5]{bertolini_darmon_prasana}} Assume that $d_K$ is odd. Let $\chi\in \Sigma_{cc}^{(2)}(k,c,\fN,\epsilon)$ with infinity type $(k+j, -j)$ be such that the epsilon factor $\epsilon_q(f,\chi^{-1})=+1$ at every finite place. Let
$$
L_{\rm alg}(f,\chi):=\sum_{\fa\in{\rm Pic}(\cO_c)}\chi_j^{-1}(\fa)\delta_k^j(f)\bigl(\fa\ast (A_0,t_0,\omega_0)\bigr)   ,
$$ 
where $\chi_j=\chi\cdot {\bf N}^j$. Here ${\bf N}$ is the norm Hecke character and the infinite type of $\chi$ is $(k+j, -j)$ with $j\ge 0$.
Then   $L_{\rm alg}(f,\chi)$ is an algebraic number and
$$
\bigl(L_{\rm alg}(f, \chi)\bigr)^2=\frac{w(f,\chi)^{-1}C(f,\chi,c)L(f, \chi,0)}{\Omega^{2(k+2j)}},
$$
where $w(f,\chi)\in \{\pm 1\}$, $C(f, \chi, c)$ is a precisely defined constant and $\Omega$ is the complex period.

\end{theorem}

\subsection{Conclusions.}\label{sec:conclusions}

Recall that we have fixed a quadratic imaginary field $K$  and an integer $N\geq 5$, with $N$ satisfying the Heegner assumptions, i.e., there is an ideal $\fN$ of $\cO_K$ with the property that  $\cO_K/\fN\cong \Z/N\Z$. We fix a character $\epsilon\colon (\Z/N\Z)^\ast\to \overline{\Q}^\ast$. We also choose an integer $c\geq 1$  such that $(c,N)=1$ and we denote by $\cO_c$, the order in $\cO_K$ of conductor $c$. We fix an odd prime integer $p>0$ which is non-split in $K$.  

We let $\Sigma_{cc}^{(2)}(c,\fN,\epsilon)\subset \Sigma_{cc}(c,\fN,\epsilon)$ denote the spaces of Hecke characters of type $(c,\fN,\epsilon)$ as in Section \ref{sec:grossencharacter} of infinity type $(k+j,-j)$ with $k\geq 1$ and $k+j\geq 1$ (and  $j\geq 0$ for $\Sigma_{cc}^{(2)}(k,c,\fN,\epsilon)$) and    for which one of the following holds:

\begin{itemize}

\item[i.] $\epsilon$  is of parity $k$ and non trivial if $k=2$;

\item[ii.] $c$ and $d_K$ are odd,  $k\geq 2$ and $\chi\vert_{\bA_\Q^\ast}=\epsilon {\bf N}^k$. 

\end{itemize}

We write $\Sigma_{cc}^{(2)}(k,c,\fN,\epsilon)\subset \Sigma_{cc}^{(2)}(c,\fN,\epsilon)$, resp.~$\Sigma_{cc}(k,c,\fN,\epsilon)\subset \Sigma_{cc}(c,\fN,\epsilon)$ for the subspaces where the integer $k$ is fixed.

\smallskip

Consider an Eisenstein series $f$ as in Section \ref{sec:eisenstein} in the first case and a cuspform $f$ as in  Theorem \ref{theorem:classical}  in the second case. Let $k$ be the weight of $f$.  We then have an algebraic integer

$$
L_{\rm alg}(f,\chi):=\sum_{\fa\in{\rm Pic}(\cO_c)}\chi_j^{-1}(\fa)\delta_k^j(f)\bigl(\fa\ast (A_0,t_0,\omega_0)\bigr) .
$$ 

\bigskip
\noindent
Our goal in this article is to $p$-adically interpolate the family of special values  $L_{\rm alg}(f, \chi)$, by varying  the algebraic Hecke characters  $\chi\in \Sigma_{cc}^{(2)}(k,c,\fN,\epsilon)$, by a $p$-adic $L$-function. We will see that  there is an integer $n_0$, depending on $p$, for example $n_0=2$ works if $p\geq 5$, such that under the assumption that 
$$p^{n_0} \quad \hbox{{\rm divides the conductor of }} \chi$$the $p$-adic $L$-function will be defined as a locally analytic function on the (open and closed) subspace of the analytic space $\widehat{\Sigma}^{(2)}(k,c,\fN,\epsilon)$ (defined in the next section) given by the displayed condition above. We will also construct a two variable $p$-adic $L$-function by allowing $f$ to vary in a $p$-adic family of finite slope overconvergent forms.

\subsection{$p$-adic families of Hecke characters.}\label{sec:Sigmahat}

 We simply write $\Sigma^{(2)}(c,\fN,\epsilon)$ for the space of $p$-adic avatars, in the sense of Section \ref{sec:avatar} for $p$ odd, of the algebraic Hecke characters in $\Sigma^{(2)}_{cc}(c,\fN,\epsilon)$   defined in Section \ref{sec:conclusions}.  Consider the map   $$w_1\times w_2\colon \Sigma^{(2)}(c,\fN,\epsilon)\to \Z \times \Z$$given by sending a character $\chi$ of infinity type  $(k+j, -j)$ to $(k,j)$.

Let $S$ be a field extension of $K$ in $\overline{\Q}$ which contains the values of the characters  of the finite group  $\cH(c,\fN)$ defined in Definition \ref{def:cH}.  Let $\fp$ be a prime of $S$ dividing $p\cO_S$  and denote by $\cO_{S,\fp}$ the $\fp$-adic completion of $\cO_S$.  Given  $\chi\in \Sigma^{(2)}(c,\fN,\epsilon)$ we write $\chi'\colon \bA_K^{(\fN c), f, \ast}\to \cO_{S,\fp}$ for the restriction of $\chi$ to the subgroup of finite id\`eles  $\bA_K^{(\fN c), f, \ast}$ prime to $c \fN$. Denote by $\cF(\bA_K^{(\fN c), f,\ast}, \cO_{S,\fp})$ the space of continuous maps endowed with the compact open topology.  We obtain an injective map $$\iota\colon \Sigma^{(2)}(c,\fN,\epsilon)\subset \cF(\bA_K^{(\fN c), f,\ast}, \cO_{S,\fp}).$$We decompose $\Sigma^{(2)}(c,\fN,\epsilon)=\amalg_{\nu} \Sigma^{(2)}(c,\fN,\nu)$ according to their finite characters $\nu\colon \bigl(\cO_K/c\fN\bigr)^\ast\to \C_p^\ast$. The embedding $\iota $ of $ \Sigma^{(2)}(c,\fN,\nu)$ in $\cF(\bA_K^{(\fN c), f,\ast}, \cO_{S,\fp}) $ endows each  $\Sigma^{(2)}(c,\fN,\nu)$ with the induced topology and we denote by $ \widehat{\Sigma}^{(2)}(c,\fN,\nu)$ the completion of $\Sigma^{(2)}(c,\fN,\nu)$ with respect to this topology. Set

$$\widehat{\Sigma}^{(2)}(c,\fN,\epsilon):=\amalg_{\nu}  \widehat{\Sigma}^{(2)}_{cc}(c,\fN,\nu).$$

Let $q$ be  the cardinality of the residue field of $\cO_K$  and  define $$\overline{w}:=w_1\times w_2\colon \Sigma^{(2)}(c,\fN,\epsilon)\lra \Z \times \Z\hookrightarrow \bigl((\Z/6(q-1) \Z) \times \Z_p \bigr) \times  \bigl((\Z/6(q-1) \Z) \times \Z_p \bigr).$$

 \begin{lemma}\label{lemma:w} 
 The map $\overline{w}$ is continuous  and extends to a map: $$\overline{w}\colon \widehat{\Sigma}^{(2)}(c, \fN, \epsilon)\lra \bigl( (\Z/6(q-1)\Z)\times \Z_p\bigr)^2.$$
Moreover $\overline{w}$ is a local homeomorphism. In particular $\widehat{\Sigma}^{(2)}(c,\fN,\epsilon)$ inherits a unique structure of   analytic space making $\overline{w}$ a locally analytic isomorphism. 

\end{lemma} 
\begin{proof}    We take $\chi \in \Sigma^{(2)}(c,\fN,\epsilon)$ with finite character $\nu$ and  infinity type $(m,r)$. Let $t_1,\ldots,t_s\in \bA_K^{(\fN c),f, \ast} $ be representatives of elements of $\cH(c,\fN)$.  Given a positive integer $M$ we let $U(\chi,M)\subset \cF(\bA_K^{(\fN c), f,\ast}, \cO_{S,\fp})$ be the subset consisting  of functions $g$ such that for every $i=1,\ldots,s$ and every $h\in t_i \cdot H^{(\fN c)}$ we have $g(h)\equiv \chi(h)$ mod $p^M \cO_{S,\fp}$. For varying $M$ we get a fundamental system of  open neighbourhoods of $\chi$ for the compact open topology on $\cF(\bA_K^{(\fN c),f, \ast}, \cO_{S,\fp})$. If $\chi'\in \Sigma^{(2)}(c,\fN,\nu)$ has  infinity type $(m',r')$ with $m' \equiv m$ and $r'\equiv r$ modulo $6(q-1)\Z$  then $\chi'\in U(\chi,M)$ if and only if the following condition holds $$(\ast)\quad \alpha_p^{m+r-m'-r'} \overline{\alpha}_p^{r'-r}\in 1+p^M\cO_{S,\fp}\quad \forall \,\alpha_p\in 1 + \fP  \widehat{\cO}_{K,p}$$with $\fP$ the maximal ideal of $\widehat{\cO}_{K,p}$. Using the $p$-adic  logarithm and the fact that $p\neq 2$, Condition ($\ast$) is equivalent to require that $m\equiv m'$ and $r\equiv r'$ modulo $p^{M-1} \Z$. Hence $\overline{w}(\chi')\in \overline{w}(\chi) + \bigl( \{0\}  \times p^{M-1}\Z \bigr)^2$. This implies that $\overline{w}$ is continuous and  extends to a map $\overline{w}\colon \widehat{\Sigma}^{(2)}(c, \fN, \epsilon)\lra \bigl( (\Z/6(q-1)\Z)\times \Z_p\bigr)^2$.

Viceversa  consider the neighborhood $W(m,r,M)$ of $(m,r)$ in $ \bigl((\Z/6(q-1) \Z) \times \Z_p \bigr)^2$ of elements that are congruent to $(m,r)$ modulo $\bigl(6(q-1)p^{M-1}\Z_p\bigr)^2$. For $(m',r')\in W(m,r,M) $ let $\chi'\in \cF(\bA_K^{(\fN c), f,\ast}, \cO_{S,\fp})$ be defined, for $k\in K^{(\fN c)}$, $\alpha\in  H^{(\fN c)}$, $i=1,\ldots, s$  by $$\chi'(k t_i \alpha )= \chi(k t_i \alpha) \bigl(t_{i,p} \alpha_p\bigr)^{m'+r'-m-r} \bigl(\overline{t}_{i,p }\overline{\alpha}_p\bigr)^{r-r'}= \chi(t_i)  t_{i,p}^{m'+r'-m-r} \overline{t}_{i,p }^{r-r'} \alpha_p^{m'+r'} \overline{\alpha}_p^{-r'}. $$We first prove that $\chi'$ is well defined. Notice that $k t_i \alpha=h t_i \beta$ for $h$, $k\in  K^{(\fN c)}$ and $\alpha$, $\beta\in H^{(\fN c)}$ if and only if there exists $z\in  K^{(\fN c)}\cap  H^{(\fN c)}$ such that $\alpha=\beta z$ and $k=h z^{-1}$. As $K^{(\fN c)}\cap  H^{(\fN c)}\subset \cO_K^\ast=\mu_K$ (the roots of unity in $K$) we have that $\mu_K=\mu_4$ if $K=\Q(i)$, $\mu_K=\mu_6$ if $K$ is the $6$-th cyclotomic field and $\mu_K=\{\pm 1\}$ otherwise. In any case, $$\chi'(k t_i \alpha )  =\chi'(h t_i \beta) z^{m'-m+r'-r} \overline{z}^{r-r'}=\chi'(h t_i \beta)$$as by assumption $12$ divides $m'-m$ and $r'-r$ and $z^{12}=1$. 

One checks that  $\chi'$ defines a group homomorphism $\bA_K^{(\fN c), f,\ast}\to  \cO_{S,\fp}^\ast$ and that it  lies in $ U(\chi,M)$.  For $(m',r')\in \Z\times \Z$, with $m'\geq 1$ and $r'\geq 0$,  it lies in  $\Sigma^{(2)}(c,\fN,\nu)$ and  $\overline{w}(\chi')=(m',r')$.  We thus get a map
$$W(m,r,M) \lra U(\chi,M)  \cap \widehat{\Sigma}^{(2)}(c,\fN,\nu)$$ that defines the inverse of the restriction of $\overline{w}$ to $U(\chi,M)  \cap \Sigma^{(2)}(c,\fN,\nu)$ and, hence,  to $U(\chi,M)  \cap \widehat{\Sigma}^{(2)}(c,\fN,\nu)$.  Hence, $\overline{w}$ provides an homeomorphism $U(\chi,M)  \cap \widehat{\Sigma}^{(2)}(c,\fN,\nu)\cong W(m,r,M)$  as claimed.

\end{proof}

\begin{remark} As explained in the proof of Lemma \ref{lemma:w}, the choice of the quantity $6(q-1)$ is made so that the order of the group of roots of unity $\mu_K$ of the imaginary field $K$ divides $6(q-1)$. If  $p$ is inert and $p\geq 5$, then $q-1=p^2-1$ is already divisible  by $12$ and the quantity $q-1$ suffices. Similarly, if $K$ is not the $4$-th or the $6$-th cyclotomic field, then $\mu_K$ has order $2$ and since $q-1$ is divisible by $2$ we can take $q-1$ in this case as well.

\end{remark}

\bigskip
\noindent
We now fix a positive integer $k$ and denote by  $(\overline{k}, k)$ the diagonal image of $k$ in
$(\Z/6(q-1)\Z)\times \Z_p$. 

\begin{definition} We define $\widehat{\Sigma}^{(2)}(k, c, \fN, \epsilon):=\overline{w}^{-1}\bigl(\{(\overline{k},k)\}\times (\Z/6(q-1)\Z)\times \Z_p\bigr)$ and
the composite map
$$w_2\colon 
\widehat{\Sigma}^{(2)}(k, c,\fN,\epsilon)\stackrel{\overline{w}}{\lra}\{(\overline{k}, k)\}\times \bigl((\Z/6(q-1)\Z)\times \Z_p\bigr)\stackrel{\rm proj}{\lra} (\Z/6(q-1)\Z)\times \Z_p,
$$
where ${\rm proj}$ is the second projection.
\end{definition}

It follows from  Lemma \ref{lemma:w} that $\widehat{\Sigma}^{(2)}(k, c, \fN, \epsilon)$ is the closed subspace of $\widehat{\Sigma}^{(2)}(c, \fN, \epsilon)$ given by the completion of the space of $p$-adic avatars $\Sigma^{(2)}(k,c,\fN,\epsilon)$ of characters in $\Sigma_{cc}^{(2)}(k,c,\fN,\epsilon)$ (defined in \S \ref{sec:conclusions}) and that $w_2$ is a local homeomorphism. Moreover we have the following commutative and cartesian diagram:    
$$
\begin{array}{ccccccccc}
\widehat{\Sigma}^{(2)}(c,\fN,\epsilon)& \stackrel{\overline{w}}{\lra}&  \bigl((\Z/6(q-1) \Z) \times \Z_p \bigr) \times  \bigl((\Z/6(q-1) \Z) \times \Z_p \bigr)\\
\cup&&\uparrow \alpha\\
\widehat{\Sigma}^{(2)}(k, c,\fN,\epsilon)& \stackrel{w_2}{\lra}&  (\Z/6(q-1) \Z) \times \Z_p,
\end{array}
$$
where $\alpha$ is the map $\alpha(x)=\bigl((\overline{k},k), x   \bigr)$.

Now, because $p-1$ divides $q-1$ we have a canonical projection $\Z/6(q-1)\Z\to \Z/(p-1)\Z$ and a canonical locally analytic isomorphism $\Z/(p-1)\Z\times \Z_p\cong W(\Q_p)$, where let us recall $W$ is the weight space.
By composing $w_2$ and $\overline{w}$ with these maps we obtain compatible locally analytic maps:
$$w\colon \widehat{\Sigma}^{(2)}(k, c, \fN, \epsilon)\to W(\Q_p)$$and $$\widetilde{w}\colon \widehat{\Sigma}^{(2)}(c, \fN, \epsilon)\to W(\Q_p)\times W(\Q_p).$$

\begin{remark}\label{rmk:p2conductos} As $\widehat{\Sigma}^{(2)}(c,\fN,\epsilon):=\amalg_{\nu}  \widehat{\Sigma}^{(2)}_{cc}(c,\fN,\nu)$, for every integer $n>1$, the requirement that the finite character $\nu$ on  $\bigl(\cO_K\otimes \Z_p)^\ast$ is {\em not} trivial on the subgroup $1+p^{n-1} \cO_K\otimes \Z_p$, defines open and closed subspaces $\widehat{\Sigma}^{(2),p^n}(c,\fN,\epsilon)$ and $\widehat{\Sigma}^{(2),p^n}(k,c,\fN,\epsilon) $ of $\widehat{\Sigma}^{(2)}(c,\fN,\epsilon)$ and respectively of $\widehat{\Sigma}^{(2)}(k,c,\fN,\epsilon)$. They contain the classical Hecke characters with $p^n$ dividing their conductor as dense subsets.

\end{remark}

\section{Preliminaries on {\rm CM}-elliptic curves.}
\label{section:derhamCM}

In this section we consider a quadratic imaginary field $K$ and a positive integer $d$. Fix an elliptic curve   $(E,\iota)$ with full  CM by $\cO_d$, the order of conductor $d$ in $\cO_K$. The pair $(E,\iota)$ is then defined over the ring of integers $\cO_L$ of a suitable ring class field $L$ of $K$.  Let $p$ be an odd prime not dividing $d $ and let $R$ be the completion of $\cO_L$ at a prime above $p$. We review the theory of the canonical subgroup of order a power of $p$ of $E$ over $R$. If $p$ splits in $K$ then $E$ is ordinary over $R$ so that $E$ admits canonical subgroups of every level $n$, coinciding with the connected subgroup of $E[p^n]$. The cases when $p$ is either inert or ramified in $K$ are more subtle. If $E'$ is an elliptic curve over $R$ we write  ${\rm Hdg}(E')\in R$ for any element lifting the Hasse invariant of the mod $p$ reduction of $E'$.

\subsection{Canonical subgroups for CM elliptic curves. The inert case.}

Assume first that $p$ is inert in $K$. Then 

\begin{lemma}\label{Lemma:LTI}
Let $E$ be an elliptic curve over the ring of integers $R$ of a local field with {\rm CM} by $\cO_d$. Then\smallskip

a)  $E$ does not admit a canonical subgroup and  $\mathrm{v}_p\bigl( \mathrm{Hdg}(E)\bigr)\geq {\frac{p}{p+1}}$;
\smallskip

b) if $C^{(n)}\subset E[p^n]$ is a cyclic subgroup of order $p^n$ for $n\geq 1$ and we define $E^{(n)}:=E/C^{(n)}$, then $E^{(n)}$ admits a canonical subgroup ${\rm H}^{(n)}=E[p^n]/C^{(n)}$ of order $p^n$ and $\mathrm{v}_p\bigl( \mathrm{Hdg}(E^{(n)})\bigr)= {\frac{1}{p^{n-1}(p+1)}}$;

\end{lemma}

\begin{proof} We remark that $E$ cannot have a canonical subgroup $D$ of order $p$ for if it had then $D$ would be stabilized by the action of $\cO_d$ by functoriality. Since $p$ is inert, coprime to $d$, we have $\cO_d/p\cO_d\cong \F_{p^2}$ and $D$ would be an $\F_{p^2}$-vector space, which it cannot be. It then follows from the Katz--Lubin theory of canonical subgroups \cite[Thm.~3.10.7]{katzpadic} that 
$\mathrm{v}_p\bigl( \mathrm{Hdg}(E)\bigr)\geq {\frac{p}{p+1}}$. The second statement follows from loc. cit.\end{proof}

\subsection{Canonical subgroups for CM elliptic curves. The ramified case.}

Assume that $p$ is ramified in $K$ and let $\fP$ be the prime of $\cO_K$ over $p$. Let ${\rm H}:=E[\fP]$; it is a subgroup scheme of $E[p]$ of order $p$. Then we have:

\begin{lemma}
\label{lemma:ramified0}
a) $\mathrm{v}_p\bigl(\mathrm{Hdg}(E)\bigr)=1/2$ and ${\rm H}$ is the canonical subgroup of order $p$ of $E$.  \smallskip

b) For any   subgroup $C^{(n)}\subset E[p]$ of order $p^n$ with $n\geq 1$ and $C\cap {\rm H}=\{0\}$, if we denote $E^{(n)}:=E/C^{(n)}$ then $\displaystyle \mathrm{v}_p\bigl(\mathrm{Hdg}(E^{(n)})\bigr) =\frac{1}{2p^n}$ and ${\rm H}^{(n)}:=E[p^n]/C^{(n)}$ is the canonical subgroup of order $p^n$ of $E^{(n)}$.

\end{lemma}
\begin{proof}
Note that $\mathrm{v}_p\bigl(\mathrm{Hdg}(E)\bigr)$  depends only on the underlying $p$-divisible group
$E[p^\infty]$ of $E$. If $\pi$ denotes a uniformizer of $K_{\fP}$, then $[\pi]\colon  E[p^\infty] \to E[p^\infty]$ identifies its kernel with
$E[\mathfrak{P}]$ and the quotient $E[p^\infty]/E[\mathfrak{P}]$ with $E[p^\infty]$. Hence
$\mathrm{v}_p\bigl(\mathrm{Hdg}(E)\bigr)=\mathrm{v}_p\bigl(\mathrm{Hdg}(E/E[\fP])\bigr)$ and by the Katz-Lubin theory \cite[Thm.~3.10.7]{katzpadic} the only possibility is that this value is $1/2$.
Claim (b)  follows again from the Katz-Lubin theory.

\end{proof}

\subsection{A technical lemma}\label{sec:technicalsec}

Consider the elliptic curve $E'=E^{(1)}$ defined in Lemma \ref{Lemma:LTI} (in the inert case) or in Lemma \ref{lemma:ramified0} (in the ramified case).  It is a quotient of $E$ and, in particular it has CM by $\cO_{a}$ with $a=p d$. Moreover $E'$ admits a canonical subgroup ${\rm H}'$ and $E'/{\rm H}'\cong E$. We let $\lambda\colon E' \to E$ be the quotient map and we denote by $\lambda'\colon E\to E'$ the dual isogeny.

Let $C'\subset E'[p]$  denote one of the $p$ subgroups of order $p$ of $E'[p]$ distinct from ${\rm H}'$. Define
$E'':=E'/C'$ and  $\lambda''\colon E' \to E''$ be the quotient morphisms and $(\lambda'')^\vee\colon E''\to E'$ the dual isogeny. Summarizing we have the degree $p$ isogenies:

$$E\stackrel{\lambda'}{\longrightarrow} E' \stackrel{\lambda''}{\longrightarrow} E''.$$

Let ${\rm H}_{E'}$ be the de Rham cohomology ${\rm H}^1_{\rm dR}(E'/R)$. It sits in the exact sequence, defined by the Hodge filtration,  $$0 \to \omega_{E'} \to {\rm H}_{E'} \to \omega_{E'}^\vee\to 0.$$
We denote by  ${\rm H}_{E',\tau}$, ${\rm H}_{E',\overline{\tau}}$  the $R$-submodules of ${\rm H}_{E'}$ on which $\cO_a$ acts via the CM type $\tau\colon K \to L$ and its complex conjugate $\overline{\tau}$, respectively.
The isogeny $\lambda''$ defines a morphism on de Rham cohomology $(\lambda'')^\ast: {\rm H}_{E''}\to {\rm H}_{E'}$ that induces a commutative diagram:
$$
\begin{array}{rrrrrrrrrrrrrr}
0 &\lra &\omega_{E''}& \lra & {\rm H}_{E''}& \lra & \omega_{E''}^\vee &\lra & 0\\
&&(\lambda'')^\ast\big\downarrow\quad&&(\lambda'')^\ast\big\downarrow \quad &&{\rm Lie}((\lambda'')^\vee)\big\downarrow\quad \\
0& \lra & \omega_{E'} &\lra & {\rm H}_{E'}& \lra &\omega_{E'}^\vee &\lra & 0
\end{array}
$$

The last technical result of this section, to be used in \S \ref{sec:splitinginert} and in \S \ref{sec:splitingramified}, concerns the relative positions of the $R$-submodules 
${\rm H}_{E',\overline{\tau}}$ and ${\rm Lie}\bigl((\lambda'')^\vee\bigr)(\omega_{E''}^\vee)$ in $\omega_{E'}^\vee$.

\begin{lemma}\label{lemma:splittingHdR} We have

\begin{itemize}

\item[i.] ${\rm H}_{E',\tau}=\omega_{E'}$, the invariant differentials of $E'$;

\item[ii.] the inclusion  ${\rm H}_{E',\overline{\tau}} \subset \omega_{E'}^\vee$ has cokernel annihilated by  ${\rm Hdg}(E')\cdot \mathrm{diff}_{K/\Q}$, where $\mathrm{diff}_{K/\Q}$ denotes the different ideal of $K/\Q$;

\item[iii.] the image of  the inclusion ${\rm Lie}\bigl((\lambda'')^\vee\bigr)\bigl(\omega_{E''}^\vee\bigr) \subset \omega_{E'}^\vee$ is contained in $p {\rm Hdg}(E'')^{-1} \omega_{E'}^\vee $.

\end{itemize}

\end{lemma}
\begin{proof} (i) The first claim follows from the fact that $\cO_a$ acts on $\omega_{E'}$ via $\tau$ by the definition of CM type and the Hodge filtration provides an exact sequence $ 0\to \omega_{E'} \to  {\rm H}_{E'} \to \omega_{E'}^\vee \to 0$ of $\cO_a$-modules  where $\cO_a$ acts on $\omega_{E'}^\vee$ via $\overline{\tau}$. 

(ii) Write $ \widetilde{{\rm H}}_E:={\rm H}_{E,\tau}\oplus {\rm H}_{E,\overline{\tau}}\subset {\rm H}_{E}$ similarly to what has been done for $E'$. Then $\lambda$ induces morphisms

$$\begin{matrix} \widetilde{{\rm H}}_E&{ \longrightarrow} &  {\rm H}_{E} \cr
\big\downarrow & & \big\downarrow \lambda^\ast\cr
\widetilde{{\rm H}}_{E'} &{ \longrightarrow} &  {\rm H}_{E'} \cr \end{matrix}$$

The map $\lambda^\ast \colon {\rm H}_{E} \to {\rm H}_{E'}$ respects the Hodge filtration. It coincides with pull-back of the differentials $\lambda^\ast\colon \omega_E\to \omega_{E'}$. On the quotient of the Hodge filtration, it induces the map $${\rm Lie}\bigl(\lambda^\vee \bigr)\colon \omega_E^\vee\to \omega_{E'}^\vee,$$where $\lambda^\vee\colon E' \to E$ is the dual isogeny. 

To prove claim (ii)  it suffices to show that $\omega_{E'}^\vee/{\rm H}_{E',\overline{\tau}}$ is annihilated by ${\rm Hdg}_{E'}\cdot \mathrm{diff}_{K/\Q}$. This follows if we prove that\smallskip

(I) The module $\omega_{E}^\vee/{\rm H}_{E,\overline{\tau}}$ is annihilated by $\mathrm{diff}_{K/\Q}$;

\smallskip
(II) The map ${\rm Lie}\bigl(\lambda^\vee \bigr)\colon \omega_E^\vee \to \omega_{E'}^\vee$ has cokernel annihilated by ${\rm Hdg}(E')$.

\smallskip

{\it Proof of claim (I).}\enspace Notice that $E$ is defined over the $p$-adic completion $R$ of $\cO_K$.  Let $R_0=\mathbb{W}(k)\subset R$ be the ring of Witt vectors of the residue field $k$ of $R$. Since its ramification index is $\leq 2 \leq p-1$, then ${\rm H}_{E}\otimes_{\cO_K} R$ coincides with the base change via $R_0\to R$ of crystalline cohomology ${\rm H}_{\rm cris}(E_0/R_0)$ of the special fiber $E_0$ of $E$ (see \cite{berthelot_ogus}). It suffices to prove that ${\rm H}_{\rm cris}(E_0/R_0)$  is a projective $\cO_K \otimes R_0$-module of rank $1$. But this is true after inverting $p$ and $\cO_K \otimes R_0$ is a Dedekind domain for which projective modules are torsion free modules. The claim then follows.

Let us observe the following elementary fact: if $e$ and $\overline{e}$ denote the idempotents $(1,0)$ and $(0,1)$ in $K \times K$ identified with $\cO_K \otimes_\Z K $ via the isomorphism of $K$-algebras $\tau \times \overline{\tau}\colon \cO_K \otimes_\Z K \cong K\times K$, then $\mathrm{diff}_{K/\Q} e$ and $\mathrm{diff}_{K/\Q} \overline{e}$ lie in $\cO_K\otimes_\Z \cO_K$.

To conclude the proof of the lemma let $z\in {\rm H}_{E}$.  We have $\mathrm{diff}_{K/\Q} e z\in {\rm H}_{E,\tau}$ and $\mathrm{diff}_{K/\Q} \overline{e}z\in {\rm H}_{E,\overline{\tau}}$ and their sum is $\mathrm{diff}_{K/\Q} z$ by the above observation. Thus the cokernel of $\widetilde{{\rm H}}_E\subset {\rm H}_{E} $, which is the cokernel of ${\rm H}_{E,\overline{\tau}}\subset \omega_{E}^\vee$, is   annihilated by $\mathrm{diff}_{K/\Q}$ as claimed.

\smallskip

{\it Proof of claim (II).}\enspace Since  $\omega_E$ and $\omega_{E'}$  are free $R$-modules, it suffices to show that $(\lambda^\vee)^\ast\colon \omega_{E'}\to \omega_E$ has image equal to ${\rm Hdg}(E') \omega_E$. Since $\lambda\colon E'\to E$ is the quotient of $E'$ by its canonical subgroup, it coincides with Frobenius modulo $p {\rm Hdg}(E')^{-1}$. Hence $\lambda^\vee$ is Verschiebung modulo $p {\rm Hdg}(E')^{-1}$ so that the image of $(\lambda^\vee)^\ast\colon \omega_{E'}\to \omega_E$  coincides with ${\rm Hdg}(E') \omega_E$, modulo $p {\rm Hdg}(E')^{-1}$ (recall that the Hasse invariant is defined as the image of the Verschiebung map on differentials modulo $p$). Since ${\rm v}_p( {\rm Hdg}(E'))<\frac{1}{2} $ by Lemma \ref{Lemma:LTI} (in the inert case) or by Lemma \ref{lemma:ramified0} (in the ramified case), we deduce that  ${\rm v}_p( {\rm Hdg}(E')) < {\rm v}_p(p {\rm Hdg}(E')^{-1})$. The conclusion follows. 

\smallskip

(iii) By construction $\lambda''\colon E'\to E''$ is defined by dividing by a subgroup different from the canonical subgroup. In particular the dual isogeny $(\lambda'')^\vee\colon E''\to E'$ is the quotient by the canonical subgroup of $E''$. Arguing as in the proof of (II), we deduce that $(\lambda'')^\vee$ is Frobenius modulo $p {\rm Hdg}(E'')^{-1}  $ so that the induced map $((\lambda'')^\vee)^\ast\colon \omega_{E'} \to \omega_{E''}$ on differentials is $0$ modulo $p {\rm Hdg}(E'')^{-1}$.  The conclusion follows.

\end{proof}

\subsection{Adelic description of CM elliptic curves.}\label{sec:adelicdescr}

Let  $c$ and $N$ be coprime positive integers.  We assume  that there exists an ideal $\fN$ of $\cO_K$ whose norm is $N$.  
In general, the set of elliptic curves  with CM by $\cO_c$ and $\Gamma_1(\fN)$-level structure is a principal homogeneous space under the action of  the group ${\cH}(c,\fN)$ of Definition \ref{def:cH}:
 given such an elliptic curve $(E,\iota, \psi_\fN)$ and given  $\fa=(a_\cP)_\cP \in \bA_K^{(\fN c),f,\ast}$ we set $\fa \ast (E,\iota, \psi_\fN):=(E',\iota',\psi_\fN')$ to be the elliptic curve $E'$ whose Tate module $\mathbf{T}(E')$ is isomorphic to  $\fa^{-1}(\mathbf{T}(E))\subset \mathbf{T}(E) \otimes \Q$ (recall that $\mathbf{T}(E)=\prod_\ell \mathbf{T}_\ell(E)$ is a principal $\cO_c \otimes \widehat{\Z}$-module so that this makes sense). By construction also $\mathbf{T}(E')$ is a principal $\cO_c \otimes \widehat{\Z}$-module so that $E'$ has CM by $\cO_c$ and the $\fN$-torsion of $E$ and $E'$ are canonically isomorphic so that $E'$ acquires a natural $\Gamma_1(\fN)$-level structure $\psi_\fN'$. 

Using the intepretation  of ${\cH}(c,\fN)$ in terms of  invertible $\cO_c$-ideals prime to $c$ and $\fN_c$ provided by  Lemma \ref{lemma:cHcfN}, given any such ideal $\fa\subset \cO_c$ the elliptic curve $E'$ above is identified with the quotient $E':=E/E[\fa]$, $\iota'$ is the $\cO_c$-action induced by $\iota$ and the $\Gamma_1(\fN)$-level structure $\psi_\fN'$ is the composite of $\psi_\fN$ with the projection $\phi_\fa:E\to E'$.
Furthemore, if $\omega$ is an invariant differential of $E$ we set $\fa \ast (E,\iota, \psi_\fN,\omega):=(E',\iota',\psi_\fN',\omega')$ where $(E',\iota',\psi_\fN',)$ is as above and $\omega'$ is the differential on $E'$ whose pull-back to $E$ via $\phi_\fa$ is $\omega$. 

Similarly, if $D\subset E$ is a subgroup of $p$-power order with $p|c$, we define $\fa \ast (E,\iota, \psi_\fN, D):=(E',\iota',\psi_\fN', D')$ and $\fa \ast (E,\iota, \psi_\fN, D, \omega):=(E',\iota',\psi_\fN', D', \omega')$, where $(E',\iota',\psi_\fN')$ is like above and $D'\subset E'$ is the image of $D$ via $\phi_\fa$, which is a prime to $p$ isogeny, and likewise for $(E',\iota',\psi_\fN', D', \omega')$.

\section{Vector bundles with marked sections. The sheaves $\bW_k$.} \label{sec:VBMS}

We only present the theory in the cases we need, namely for the treatment of the $p$-adic $L$-functions attached to an elliptic modular eigenform twisted by Hecke characters of a quadratic, imaginary field.  
The main references for this section are \cite{andreatta_iovita} and \cite{ICM}, where more general cases are presented and  all the details are carefully spelled out.

Let $\fS$ be a formal scheme with ideal of definition $\cI$ which is invertible (i.e. locally principal, generated locally by non-zero divisors) and let $\cE$ be a locally free $\cO_\fS$-module of rank $1$ or $2$.   Let also $s\in {\rm H}^0(\fS, \cE/\cI\cE)$ be a section (we refer to it as "the marked section") such that $s\cO_{\fS}/\cI$ is a direct summand of $\cE/\cI\cE$. We call the pair $(\cE,s)$ a locally free sheaf with a marked section. We have

\begin{theorem}\label{thm:VBMSrecall} \cite[\S 2.2]{andreatta_iovita}

a) The functor attaching to a morphism of formal schemes $\rho:\fT\to \fS$ (the ideal of definition of $\fT$ is $\rho^\ast(\cI)$)
the set 
$$
\bV(\cE)\bigl(\rho:\fT\to \fS\bigr):={\rm Hom}_{\cO_{\fT}}\bigl(\rho^\ast(\cE), \cO_{\fT}   \bigr),
$$ 
is represented by the formal vector bundle $\bV(\cE):={\bf {\rm Spf}}\bigl({\rm Sym}(\cE)\bigr)$.

b) The sub-functor of $\bV(\cE)$, denoted $\bV_0(\cE,s)$, which associates to every morphism of formal schemes $\rho:\fT\lra \fS$ as above, the set
$$
\bV_0(\cE,s)\bigl(\rho:\fT\to  \fS \bigr):=\Bigl\{h\in \bV(\cE)(\rho:\fT\to\fS)\ | \ \bigl(h({\rm mod }\ \rho^\ast(\cI))\bigr)(\rho^\ast(s))=1   \Bigr\},
$$  
is represented by the open in the admissible formal blow-up of $\bV(\cE)$ at the ideal $\cJ:=(\tilde{s}-1, \cI)\subset 
\cO_{\bV(\cE)}$, where the inverse image of this ideal is generated by $\cI$. Here $\tilde{s}$ is a lift of $s$ to 
$\rH^0(\cS, \cE)$. 

\end{theorem}

\begin{remark}\label{remark:funct}   
I) {\bf Extra structure.} If $\cE$ has extra structure, namely \enspace (1)  a connection for which $s$ is horizontal and/or\enspace (2) a filtration $\cF\subset \cE$ with $\cF$ and $\cE/\cF$ locally free of rank $1$ and $s$ a generator of $\cF/\cI \cF$, then  if we denote by $\pi\colon \bV_0(\cE,s)\to \fS$ the structural morphism, the $\cO_{\fS}$-module $\pi_\ast\bigl(\cO_{\bV_0(\cE,s)}  \bigr)$ has the same type of extra structures, namely a connection in case (1) and an increasing filtration in case (2). In this second  case $\pi$ factors via $\pi'\colon \bV_0(\cF,s) \to \fS$  and the first piece of the filtration is simply $\pi_\ast'\bigl(\cO_{\bV_0(\cF,s)}  \bigr)$. See \cite[\S 2.3 \& \S 2.4]{andreatta_iovita}.

\

II) {\bf Functoriality.} Let us suppose that we have a morphism of formal schemes $\varphi\colon \fS\to \fS'$ such that the ideal of definition $\cI'$ of $\fS'$ is invertible and the ideal of definition of $\fS$ is $\varphi^\ast(\cI')$. Suppose we also have a pair $(\cE',s')$ consisting of a locally free sheaf with a marked section on $\fS'$ and let $\cE:=\varphi^\ast(\cE')$ and $s:=\varphi^\ast(s')$. Then $(\cE,s)$ is a locally free sheaf with a marked section on $\fS$. Moreover the functoriality of VBMS implies that the following natural diagram is cartesian:
$$
\begin{array}{ccccc}
\bV_0(\cE,s)&\lra&\bV_0(\cE',s')\\
u\downarrow&&v\downarrow\\
\fS&\stackrel{\varphi}{\lra}&\fS'
\end{array}
$$
In particular we have: $u_\ast\bigl(\cO_{\bV_0(\cE,s)}\bigr)\cong \varphi^\ast\bigl(v_\ast(\cO_{\bV_0(\cE',s')}  \bigr)$.

\end{remark}

\subsection{Applications to modular curves. The sheaf $\bW_k$.}\label{sec:cX}

Let $N\ge 5$ be an integer and $p\geq 3$ a prime integer such that $(N,p)=1$ and consider the 
tower of modular curves $$(\ast)\ \  X(N,p^2)\to X(N,p)\to   X_1(N)$$over the ring of integers of a finite extension of $\Q_p$, to be made more precise later. Here the modular curves classify, in order from left to right: generalized elliptic curves with $\Gamma_1(N)\cap \Gamma_0(p^2)$, respectively $\Gamma_1(N)\cap \Gamma_0(p)$, respectively $\Gamma_1(N)$-level
structures. We denote $\hat{X}(N,p^2)\to \hat{X}(N,p)\to \hat{X}_1(N)$ the sequence of formal completions of these curves along their respective special fibers and denote by $\cX(N,p^2)\to \cX(N,p)\to  \cX_1(N)$ the adic analytic generic fibers associated to the previous sequence of formal scheme. We denote $\bE$ the universal generalized elliptic curve over all these formal schemes or adic spaces.

Consider the ideal ${\rm Hdg}$, called the Hodge ideal, defined as the ideal of $\cO_{\hat{X}_1(N)}$, locally (on open affines ${\rm Spf}(R)\subset \hat{X}_1(N)$
which trivialize the sheaf $\omega_\bE$) generated by $p$ and  a local lift, ${\rm Ha}^o(E/R,  \omega)$, of the Hasse invariant ${\rm Ha}$, where $\omega$ is a basis of $\omega_E$. For every integer $r\geq 2$ we denote by
$\fX_r$ the formal open sub-scheme of the formal admissible blow-up of  $\hat{X}_1(N)$ with respect to the sheaf of ideals $(p, {\rm Hdg}^r)$, where this ideal is generated by ${\rm Hdg}^r$. Let $\mathcal{X}_r$ denote the adic generic fiber of $\fX_r$. By construction ${\rm Hdg}$ is an invertible ideal in $\fX_r$.  Recall from \cite[App. A]{halo_spectral}  that the universal generalized elliptic curve $\bE\to \fX_r$ has a canonical subgroup  ${\rm H_m}\subset \bE[p^m]$ of order $p^m$, where $m$ depends on $r$. In this article we only need: if $r\ge 2$ then $m=1$ and if $r\ge p+2$ then $m=2$.
 For $r=2$ we drop the subscript, i.e., we write $\fX:=\fX_2$, $\cX:=\cX_2$ etc.

\medskip
\noindent

Let us denote by $\pi\colon \mathcal{IG}_{m,r}:={\rm Isom}\bigl(\underline{\Z/p^m\Z}, {\rm H}_m^D\bigr)\to \mathcal{X}_r$ the $m$-th layer of the adic analytic Igusa tower over $\mathcal{X}_r$, where ${\rm H}_m^D$ denotes the Cartier dual of ${\rm H_m}$. Then $\mathcal{IG}_{m,r}$ is a finite, \`etale, Galois cover
of $\mathcal{X}_r$, with Galois group $(\Z/p^m\Z)^\ast$ and  we denote by $\fIG_{m,r}$ the formal scheme which is the normalization of
$\fX_r$ in $\mathcal{IG}_{m,r}$.  Let $r\ge 2$.

\begin{proposition}
\label{prop:cansgr}

We have

\begin{itemize}

\item[i.] the canonical subgroup ${\rm H}_1$ of the universal elliptic curve $\bE$ over $\fIG_{1,r}$ is a lifting of the kernel of Frobenius modulo $p/\mathrm{Hdg}$;

\item[ii. ] the map of invariant differentials associated to the inclusion ${\rm H}_1\subset \bE$ induces an isomorphism $ \omega_\bE/(p\mathrm{Hdg}^{-1}) \omega_\bE\cong \omega_{{\rm H}_1}$ so
that via ${\rm dlog}_{{\rm H}_1}\colon {\rm H}_1^D \to \omega_{{\rm H}_1}$ we get a section $\displaystyle s':={\rm dlog}_{{\rm H}_1}(P^{\rm univ})\in \mathrm{H}^0\bigl(\fIG_{1,r}, \omega_{\bE}/(p\mathrm{Hdg}^{-1})\omega_{\bE}\bigr)$;

\end{itemize}
\end{proposition}

\begin{proof}:  These statements are proved, for example, in \cite[Appendix A]{halo_spectral}.
 \end{proof}

\noindent
In the notations of Proposition \ref{prop:cansgr} iii), let $\displaystyle s':={\rm dlog}_{{\rm H}_1}(P^{\rm univ})\in \mathrm{H}^0\bigl(\fIG_{1,r}, \omega_{\bE}/\bigl(p\mathrm{Hdg}^{-1}\bigr)\omega_{\bE}\bigr)$ be the section defined there.

\begin{lemma}

Any local lift of the section $s':={\rm dlog}_{{\rm H}_1}(P^{\rm univ})$, as above, to a local section  $\tilde{s}$ of $\omega_{\bE}$, spans the $\cO_{\fIG_{1,r}}$-submodule
$\mathrm{Hdg}^{\frac{1}{p-1}}\omega_{\bE}\subset \omega_{\bE}.$

\end{lemma}

\begin{proof}   
The statement of the Lemma strengthens \cite{halo_spectral}, where it was stated for $r\geq p^2$. It follows in this stronger form by the explicit computation of ${\rm dlog}_{\rm H_1}\colon {\rm H_1}^D \to \omega_{\rm H_1}$ using congruence group schemes in \cite[\S 6]{andreatta_iovita_stevens}. For the convenience of the reader we recall the main ideas. For every $r\ge 2$, the result of R. Coleman in the cited article describes the canonical subgroup ${\rm H_1}$ in terms of Oort-Tate theory. As explained after Proposition 6.1 in loc.cit. the existence of $P^{\rm univ}:{\rm H_1}\to \mu_p$, which is an isomorphism generically, allows us to further describe ${\rm H_1}$ as a congruence group scheme $G_\delta$, with $\delta^{p-1}$ a generator of ${\rm Hdg}$. Finally, ${\rm dlog}(P^{\rm univ})=(P^{\rm univ})^\ast(dT/T)$, with $T$ the canonical coordinate on $\mu_p$, spans $\delta\cdot \omega_{\rm H_1}$, as proved in loc. cit. The claim follows.  

\end{proof}

\medskip
\noindent
We'll use the trivialized canonical subgroups on $\fIG_{1,r}$ in order to define locally free sheaves with marked sections on this formal scheme, to which we will apply the VBMS-machine presented in section \S \ref{sec:VBMS}. More precisely, in the notations above, we define the invertible $\cO_{\fIG_{1,r}}$-submodule
$\Omega_\bE$ of $\omega_{\bE}$ as the span of any lift
$\tilde{s}$ of $s$ such that, if we set $\underline{\delta}:=\Omega_\bE\omega_\bE^{-1}$, then $\underline{\delta}$ is an invertible $\cO_{\fIG_{1,r}}$-ideal with
$\underline{\delta}^{p-1}= \pi^\ast({\rm Hdg})$ (recall that $\pi\colon \fIG_{1,r} \to \fX_r$ is the natural projection). From $\tilde{s}$ we also get a canonical section $s$
of $\mathrm{H}^0\bigl(\fIG_{1,r},
\Omega_{\bE}/p\mathrm{Hdg}^{-\frac{p}{p-1}}\Omega_{\bE}\bigr)$ so
that $s$ defines a basis of
$\Omega_{\bE}/p\mathrm{Hdg}^{-\frac{p}{p-1}}\Omega_{\bE}$ as
$\cO_{\fIG_{1,r}}/ p\mathrm{Hdg}^{-\frac{p}{p-1}}\cO_{\fIG_{1,r}}$-module. Therefore, our first locally free sheaf with marked section 
on $\bigl(\fIG_{1,r}, \cI:=p\mathrm{Hdg}^{-\frac{p}{p-1}}\bigr)$ is the pair $\bigl(\Omega_{\bE}, s\bigr)$.

 \medskip
 \noindent
 We now define a second locally free sheaf with marked section on $\bigl(\fIG_{1,r}, \cI:=p\mathrm{Hdg}^{-\frac{p}{p-1}}\bigr)$.
We denote by ${\rm H}_\bE^\#$ the push-out in the category of coherent sheaves on $\fIG_{1,r}$
of the diagram
$$
\begin{array}{cccccccccc}
\underline{\delta}^p\omega_E&\lra&\underline{\delta}^p{\rm H}_{\bE}\\
\cap\\
\Omega_{\bE}.
\end{array}
$$ 
We then have the following commutative diagram of sheaves with exact rows:
$$
\begin{array}{ccccccccc}
0&\lra&\underline{\delta}^p\omega_\bE&\lra&\underline{\delta}^p{\rm H}_{\bE}&\lra&\underline{\delta}^p\omega_{\bE}^{-1}&\lra&0\\
&&\cap&&\downarrow&&||\\
0&\lra&\Omega_{\bE}&\lra&{\rm H}_{\bE}^\#&\lra&\underline{\delta}^p\omega_{\bE}^{-1}&\lra&0\\
&&\cap&&\cap&&\cap\\
0&\lra&\omega_{\bE}&\lra&{\rm H}_{\bE}&\lra&\omega_\bE^{-1}&\lra&0
\end{array}
$$
It follows that ${\rm H}_\bE^\#$ is a locally free $\cO_{\fIG_{1,r}}$-module of rank two and 
 $(\Omega_\bE, s)=(\underline{\delta}\omega_\bE, s)\subset ({\rm H}_\bE^\#, s)$ is a compatible inclusion of locally free sheaves with marked sections.   We also have the following commutative diagram with exact rows:
$$
\begin{array}{ccccccccccc}
0&\lra&\Omega_{\bE}&\lra&{\rm H}_{\bE}^\#&\lra&\underline{\delta}^p\omega_{\bE}^{-1}&\lra&0\\
&&||&&\cap&&\cap\\
0&\lra&\Omega_{\bE}&\lra&\underline{\delta}{\rm H}_{\bE}&\lra&\underline{\delta}\omega_\bE^{-1}&\lra&0
\end{array}
$$
in which the right square is cartesian, defining ${\rm H}_\bE^\#$ as pull-back.

We have natural actions of $\fT^{\rm ext}:=  \Z_p^\ast \bigl(1+\pi_\ast\bigl(p\mathrm{Hdg}^{-\frac{p}{p-1}} \cO_{\fIG_{1,r}}\bigr)\bigr)$
 on the morphisms of formal schemes:
 $$
  u\colon \bV_0({\rm H}_\bE^\#,s)\lra  \fX_r,  \mbox{ and on } v\colon \bV_0(\Omega_\bE,s)\lra \fX_r,
 $$ 
 with trivial action on $\fX_r$.

\begin{definition}\label{def:analweight} Given a ring $R$ which is $p$-adically complete and separated, 
we say that a homomorphism $\nu \colon\Z_p^\ast\lra R^\ast$ is an analytic weight if there exists $u\in R$ with the property that $\nu(t)= \exp (u \log t)$, for every $t\in 1+ p\Z_p$. 
\end{definition}

Assume now that $r\ge 2$ for $p\geq 5$ and $r\geq 4$ for $p=3$. Let $\nu$ be an $R$-valued analytic weight.

\begin{definition}\label{def:bW} We define $\fw^\nu:=v_\ast(\cO_{\bV_0(\Omega_\bE,s)}\widehat{\otimes}_{\Z_p} R)[\nu]$
and $\bW_\nu:=u_\ast(\cO_{\bV_0({\rm H}_\bE^\#,s)}\widehat{\otimes}_{\Z_p} R)[\nu]$ as the sub-sheaves of sections of the respective sheaves on which $\fT^{\rm ext}$ acts via $\nu$ (see \cite[\S 3.1\& \S 3.3]{andreatta_iovita}).
\end{definition}

The definition makes sense  if $r\geq p^2$ for any prime $p$ as explained in \cite[\S 3.2 \& \S 3.3]{andreatta_iovita}.

\begin{remark}\label{rmk:functVBMS} {\bf Specialization.}  In the notations of definition \ref{def:bW}, assume that $R'$ is the ring of integers of a finite extension of $\Q_p$ and that we have an algebra homomorphism $R\to R'$. Let $x$ be an $R'$-valued point of $\fX_r$ defined by an elliptic curve $E$ over $R'$. Set $\Omega_E:=x^\ast(\Omega_\bE)$,  ${\rm H}_E^\#:=x^{\ast}\bigl({\rm H}_\bE^\#\bigr)$ with induced section $s_x$ obtained from $s$ by pulling-back via $x$. Then, applying the construction above to $(\Omega_E,s_x)$ and $({\rm H}_E^\#,s_x)) $ and the $R'$-valued weight $\nu':\Z_p^\ast\lra (R')^\ast$ which is the composition 
$\displaystyle \Z_p^\ast\stackrel{\nu}{\lra}R^\ast \lra (R')^\ast$, we get $R'$-modules $\fw^{\nu'}_{R'}\subset  \bW_{\nu',R'}$ that coincide with $x^\ast(\fw^\nu)$ and $x^\ast(\bW_\nu)$ respectively by the description provided in Section \ref{section:local}.  We sometimes write  $\fw^{\nu'}_{R'}=\fw^{\nu'}_{R'}(\Omega_E,s_x)$ and $ \bW_{\nu',R'}=\bW_{\nu',R'}({\rm H}_E^\#,s_x)$ if needed.

\end{remark}

As in \cite[\S 3.2 \& \S 3.3]{andreatta_iovita} one has the following:

\begin{proposition} The sheaf $\fw^{\nu}$ is an invertible
$\cO_{\fX_r}\widehat{\otimes} R$-module and $\bW_\nu$ has a natural, increasing filtration $\bigl({\rm Fil}\bigr)_{n\ge 0}$ by $\cO_{\fX_r}\widehat{\otimes} R$-submodules such that $\fw^\nu$ is identified with ${\rm Fil}_0$. 

The Gauss-Manin connection $\nabla\colon {\rm H}_\bE\to {\rm H}_\bE\otimes \Omega^1_{\fX_r/R}\bigl({\rm log}({\rm cusps})\bigr)$ induces a connection $\nabla_{\nu}$ (with poles) on $\bW_\nu$.

\end{proposition}
\begin{proof} For $r\geq p^2$ this is proven in loc. cit. For $p\geq 3$ and $r\geq 4 $ and for $p\geq 5$ and $r\geq 2$ one argues as follows. 

Consider the maps $v_0\colon \bV_0(\Omega_\bE,s)\lra \fIG_{1,r} $ and $u_0\colon \V_0({\rm H}_\bE^\#,s)\to \fIG_{1,r}$ and define  $\fw^{\nu,0}:=v_{0,\ast}(\cO_{\bV_0(\Omega_\bE,s)}\widehat{\otimes}_{\Z_p} R)[\nu]$
and $\bW_\nu^0:=u_{0,\ast}(\cO_{\bV_0({\rm H}_\bE^\#,s)}\widehat{\otimes}_{\Z_p} R)[\nu]$ as the sheaves of functions on which $1+p\mathrm{Hdg}^{-\frac{p}{p-1}} \cO_{\fIG_1}$ acts via $\nu$. The map $\pi\colon \fIG_{1,r}\to \fX_r$  is of degree $p-1$ and is endowed with an action of $\F_p^\ast$. Then $\pi_\ast(\cO_{\fIG_{1,r}})$, $\pi_\ast\bigl(\fw^{\nu,0}\bigr)$ and  $\pi_\ast\bigl(\bW_\nu^0\bigr)$ decompose as a sum of $(p-1)$-invertible sheaves according to the residual action of $\F_p^\ast$.  One reduces to prove all statements for $\fw^{\nu,0}$ and $\bW_\nu^0$ over $\fIG_{1,r}$. To define the filtration and the connection one uses Remark \ref{remark:funct}  (see \cite[\S 3.3 \& \S 3.4]{andreatta_iovita} for the details). It is $\F_p^\ast$-equivariant by functoriality of the construction.  The description of the filtration and the fact that $\fw^{\nu,0}$ is invertible  follow from the explicit description in Section  \ref{section:local}.

\end{proof}

We now fix analytic weights $k$ and $\nu$, in the sense of Definition \ref{def:analweight},  that satisfy  $$k(t)=\chi'(t) \exp\bigl((a+u) \log(t)\bigr), \qquad \nu(t)=\varepsilon(t) \exp\bigl((b+s) \log(t)\bigr) \quad \forall t\in \Z_p^\ast$$ and  the following Assumption (cf. Assumption 4.1 in \cite{andreatta_iovita}):

\begin{itemize} 

\item $a$ and $b\in \Z$, 

\item $\chi'$ and $\varepsilon$ characters of $(\Z/p\Z)^\ast $ and $\chi'$ even, i.e., $\chi'=\chi^2$ for a character $\chi$.

\item $u\in pR$ and $s\in p^2 R$. 

\end{itemize}

We  recall one of the main results of  \cite{andreatta_iovita}, namely Theorem 4.3, about $p$-adic iterates of the Gauss-Manin connection. The operator $U$ in the theorem is the Hecke operator $U_p$ on overconvergent modular forms.

\begin{theorem}
\label{thm:evaluation}
Let $F\in {\rm H}^0(\fIG_{1,r}, \fw^k)$ such that $U(F)=0$ with $r\geq p+2$.

Then, there exists a positive integer  $b(p,r)$ depending on $p$ and $r$ 
 and there exists a section  $(\nabla_k)^\nu(F)$ of $\bW_{k+2\nu}$ over $\mathcal{IG}_{1,b(p,r)}$ whose $q$-expansion, as a nearly overconvergent form, coincides with $$(\nabla_k)^\nu(F(q)):=\sum_{j=0} \left(
\begin{array}{cc} s\\ j
\end{array} \right)\prod_{i=0}^{j-1}(u+s-1-i)
\partial^{\nu-j}\bigl(F(q)\bigr)V_{k+2\nu,j}
$$
where $\left(
\begin{array}{cc} u_s \\ j
\end{array} \right)=\frac{s \cdot (s-1) \cdots (s-j+1)}{j!}$ and if $F(q)=\sum_{n, p\not\vert n} a_{n} q^n$ then
$\partial^{\nu-j}\bigl(F(q)\bigr)=\sum_{n, p\not\vert n} \nu(n) n^{-j} a_{n} q^n$.

\end{theorem}
\begin{proof}  This is the content of \cite[Thm. 4.3]{andreatta_iovita} and we refer to loc.~cit.~for details on the $q$-expansion.
\end{proof}

\begin{remark} Notice that, in particular, any analytic character $u\colon \Z_p^\ast\to \Z_p^\ast$ satisfies the assumptions of the Theorem \ref{thm:evaluation} so that $(\nabla_k)^u(F)$ is defined for any $F$ as in the statement of the Theorem.

\end{remark}

 As a consequence of the theorem we also have the following interpolation property. Take a homomorphism $\gamma\colon R\to \Z_p$ such that the induced character $\gamma\circ \nu\colon \Z_p^\ast\to \Z_p^\ast$ is a classical positive weight $\ell$, i.e., it is given by raising elements of $\Z_p^\ast$ to the $\ell$-th power.

\begin{corollary}\label{cor:specializationintegralweights}
The specialization of  $(\nabla_k)^\nu(F)$ via $\gamma$ is $\nabla_k^\ell(F)$ (the usual $\ell$-th iterate of the Gauss-Manin connection). 
\end{corollary}
\begin{proof} See Corollary 4.7 of \cite{andreatta_iovita}. \end{proof}


\subsubsection{Local descriptions of the sheaves $\bW_\nu$.}
\label{section:local}

We have nice, local descriptions of the sheaves $\fw^{\nu,0}$
and $\bW_\nu^{(0)}$ on $\fIG_{1,r}$ provided in Section 3.2.2 of \cite{andreatta_iovita}.
Let $\nu:\Z_p^\ast\lra R^\ast$ be an $R$-valued weight, with $R$ a $p$-adically complet and separated $\Z_p$-algebra. Let $A$ be a $p$-adically complete and
separated $R$-algebra with a map $\Spf(A)\to (\fIG_{1,r})_R$ over $\Spf(R)$, and such that
$\omega_\bE\vert_{{\rm Spf}(A)}$ is free and $p\mathrm{Hdg}^{-\frac{p}{p-1}} \vert_{{\rm Spf}(A)}$, $\underline{\delta}\vert_{{\rm Spf}(A)}$ are principal ideals generated respectively by $\beta$, $\delta$. It follows that the $A$-module ${\rm H}_\bE^\#({\rm Spf}(A))$ is free of rank two and let us choose a basis of it $\{f,e\}$ such that $f(\mbox{mod }\beta A)=s={\rm dlog}(P^{\rm univ})$, where let us recall $P^{\rm univ}$ extends the universal generator of  ${\rm H}^D$ on $\mathcal{IG}_{1,r}$ to $\fIG_{1,r}$.    
 
Then an $A$-point of $\bV_0({\rm H}_\bE^\#,s)$ i.e. a point of $\V_0({\rm H}_\bE^\#,s)\bigl({\rm Spf}(A)\bigr)$ can be seen as $x:=af^\vee+be^\vee$, where $\{f^\vee, e^\vee\}$ is the dual basis of $\{f,e\}$ and $a$, $b\in A$ satisfy $(a-1)\in \beta A$. We have $u_{0,\ast}\bigl(\cO_{\bV_0({\rm H}_E^\#,s)}  \bigr)\bigl({\rm Spf}(A)\bigr)=A\langle Y,Z\rangle$ and the point $x=af^\vee+be^\vee\in \bV_0({\rm H}_E^\#,s)\bigl(  {\rm Spf}(A)\bigr)$ corresponds to the algebra homomorphism $x\colon A\langle Y,Z\rangle \lra A$ sending
$\displaystyle Y\to b, Z\to \frac{a-1}{\beta}$. We then have $$\fw^{\nu,0}\bigl({\rm Spf}(A)\bigr)=(1+\beta Z)^\nu A, \qquad 
\bW_\nu^{(0)}\bigl({\rm Spf}(A)\bigr)=(1+\beta Z)^\nu A\langle \frac{Y}{1+\beta Z}\rangle\subset A\langle Y,Z\rangle.
$$Here we use Lemma \ref{lemma:ksmallr} to define $(1+\beta Z)^\nu$. The filtration is defined by the degree in $\displaystyle \frac{Y}{1+\beta Z}$.

Now suppose that $U=\Spf(A), V=\Spf(B)\subset (\fIG_{1,r})_R$ are open affines over $\Spf(R)$, with the $R$-algebras $A,B$ as above.
Let $\delta, \beta$ be respectively generators of the ideals $\displaystyle \underline{\delta}(U), p{\rm Hdg}^{-\frac{p}{p-1}}(U)$  as above, and respectively $\delta', \beta'$ generators of the sections of the same ideals over $V$. Let also $\{f,e\}$ be an $A$-basis of ${\rm H}^\#_E(U)$ satisfying the properties above and similarly, $\{f', e'\}$ be a $B$-basis of  ${\rm H}^\#_E(V)$ satisfying similar properties. 
We denote by $\displaystyle Y,Z, T:=\frac{Y}{1+\beta Z}$ the variables attached to the basis $\{f, e\}$ as above, and $\displaystyle Y',Z', T':=\frac{Y'}{1+\beta' Z'}$ the variables attached to the basis $\{f', e'\}$ such that $\bW^{(0)}_\nu(U)=(1+\beta Z)^\nu A\langle T\rangle$
and $\bW^{(0)}_\nu(V)=(1+\beta' Z')^\nu B\langle T'\rangle$.  Let $C$ be a $p$-adically complete and separated $R$-algebra such that
$S:=\Spf(C)\subset U\cap V$ is an open affine and let $\gamma:=  \left( \begin{array}{cc} {a} & {b} \\ {c} & {d}
\end{array} \right)\in \GL_2(C)$ be the change of base matrix for the restrictions of $\{f, e\}$ and $\{f', e'\}$ to $S$. We notice that 
$a=1(\mbox{mod }\beta)$ and $c=0(\mbox{mod }\beta)$. Then the isomorphism between $\Bigl((1+\beta Z)^\nu A\langle T\rangle\Bigr)|_S$ and   $\Bigl((1+\beta' Z')^\nu B\langle T'\rangle\Bigr)|_S$ is given by sending $\displaystyle (1+\beta Z)^\nu f(T)\to (1+\beta' Z')^\nu(a+cT')^\nu f\Bigl(\frac{b+dT'}{a+cT'}\Bigr)$, where $f(T)\in C\langle T\rangle$.

\subsubsection{Definition  \`a la Katz of nearly overconvergent modular forms.}

Let $\nu: \Z_p^\ast\lra R^\ast$ be an $R$-valued weight, for $R$ a $p$-adically complete and separated $Z_p$-algebra, and $A$ a $p$-adically complete and separated $R$-algebra. Suppose $A$ has an ideal $I=\alpha A$ with $I\cap \Z_p=p^n\Z_p$ for some $n\ge 1$ and $\nu$ is $n$-analytic, i.e. its restriction to $1+p^n\Z_p$ is analytic.
Define the subgroup $G\subset \GL_2(A)$ by: 
$$
 G:=\{\left( \begin{array}{cc} {a} & {b} \\ {c} & {d} \end{array} \right)\in \GL_2(A)\quad |\ a=1(\mbox{mod }\alpha),\  c=0(\mbox{mod }\alpha)\}.
$$
We let Let $A_\nu(A):=A\langle T\rangle $ with the weight $\nu$-action by $G$, defined as follows: let $\gamma:=  \left( \begin{array}{cc} {a} & {b} \\ {c} & {d}\end{array} \right)\in G$ and $f(T)\in A_\nu\langle T\rangle$, then $\displaystyle \gamma\cdot_\nu f(T):=(a+cT)^\nu f\Bigl(
\frac{b+dT}{a+cT}\Bigr)$.

We now describe the category on which the nearly overconvergent modular forms of weight $\nu$ are defined. The objects of this category are tuples: $\bigl(E/A, \psi_N, C, s, \{f, e\}  \bigr)$ where:

$\bullet$ $E/A$ is an elliptic curve over the $p$-adically complete and separated $R$-algebra $A$, $E[p](A)$ contains a canonical subgroup $H$ with $H^D(A)\cong \Z/p\Z$. The ring $A$ is supposed to contain an ideal $I=\alpha A$ such that $I\cap \Z_p=p^n\Z_p$ and $\nu$ is $n$-analytic. 

$\bullet$ $\psi_N$ is a level $\Gamma_1(N)$-level structure on $E/A$.

$\bullet$ $C\subset E[p]$ is a $\Gamma_0(p)$-level structure on $E/A$.

$\bullet$ $s\in \omega_E/I\omega_E$ is a marked section, i.e. in this case, it is the image under ${\rm dlog}$ of a generator of $H^D(A)$.  

$\bullet$ $\{f, e\}$ is an $A$-basis of ${\rm H}_E^\#$, with $f$ an $A$-basis of $\Omega_E$ such that $f=s(\mbox{mod }I)$

\begin{definition} 
\label{def:katz}
A nearly overconvergent modular form $F$ on $\cX(N,p)$ of weight $\nu$ is a rule which assigns to every tuple $\bigl(E/A, \psi_N, C, s, \{f, e\}  \bigr)$ as above, an element $F\bigl(E/A, \psi_N, C, s, \{f, e\}  \bigr)\in A_\nu\langle T\rangle$ such that:

i) The element $F\bigl(E/A, \psi_N, C, s, \{f, e\}  \bigr)$ only depends on the isomorphism class of the tuple $\bigl(E/A, \psi_N, C, s, \{f, e\}  \bigr)$.

ii) $F$ commutes with base-change.

iii) If $\bigl(E/A, \psi_N, C, s, \{f', e'\}  \bigr)$ is the same tuple with another $A$-basis, $\{f', e'\}$,  of ${\rm H}_E^\#$, if $\gamma\in G$ is the change of basis matrix, i.e. $(f'\  e')=(f \ e)\gamma$, then 
$$
F\bigl(E/A, \psi_N, C, s, \{f', e'\}  \bigr)=\gamma\cdot_\nu F\bigl(E/A, \psi_N, C, s, \{f, e\}  \bigr).
$$ 
\end{definition}

\medskip
\noindent
Finally, using the above Definition \ref{def:katz}, one may define the action of the operator $U=U_p$ on nearly overconvergent modular forms on $\cX(N,p)$, of weight $\nu$ as follows.
Let $\bigl(E/A, \psi_N, C, s, \{f, e\}  \bigr)$ be a tuple as in Definition \ref{def:katz} such that for every $D\subset E[p]$, $D\neq C$,
the elliptic curve $E/D$ has a canonical subgroup over $A$, necessarily $E[p]/D$. Let $\pi_D:E\lra E/D:=E'$ be the projection and 
$\lambda_D: E'\lra E$ be the dual isogeny. Let $\lambda_D^\#:{\rm H}_E^\#\lra {\rm H}_{E'}^\#$ be the $A$-linear map induced by $\lambda_D^\ast$. Then $\lambda_D^\#$ induces an isomorphism $\Omega_E\cong \Omega_{E'}$.

Let $F$ be a nearly overconvergent modular form on $\cX(N,p)$, seen as in Definition \ref{def:katz}. 
Then we can evaluate $U(F)$ at the tuple $\bigl(E/A, \psi_N, C, s, \{f, e\}\bigr)$ as follows:
$$
U(F)\bigl(E/A, \psi_N, C, s, \{f, e\}  \bigr)=\sum_{D\neq C}F\bigl((E/D)/A, \overline{\psi}_N, \overline{C}, \lambda_D^\ast(s), \{\lambda_D^\#(f), \lambda_D^\#(e)\}\bigr),
$$
where $\overline{\psi}_N$ is the natural $\Gamma_1(N)$ level structure on $E/D$ induced by $\psi_N$ on $E$, and $\overline{C}=E[p]/D$.

\begin{remark} [see also Proposition 2.5.3 of \cite{ye}]
\label{rmk:bzzcal} 
Suppose $\bigl( E/\cO_K, \psi_N, C\bigr)$ is a triple with $E$ an elliptic curve over $\cO_K$, with $K$ a finite extension of $\Q_p$,
$\psi_N$ a level $\Gamma_1(N)$-structure and $C$ a $\Gamma_0(p)$-structure over $\cO_K$. Suppose $C$ is not the canonical subgroup of $E$ but for all the $p$ subgroups $D\subset E[p]$, the elliptic curve $E/D$ has a canonical subgroup, which will be 
$E[p]/D$. For example this happens if $v\bigl({\rm Ha}(E/\cO_K, \psi_N, C, \omega)\bigr)=p/(p+1)$, where $\omega$ is a generator of $\omega_E$. We suppose that $C(\cO_K)\cong \Z/p\Z$, otherwise extend scalars.

Let $P\in C^D(\cO_K)$ be a generator and let $s:={\rm dlog}(P)\in \omega_E/\alpha\omega_E$, where \\$\alpha\cO_K=p{\rm Hdg}^{-p/(p-1)}(E, \psi_N, C)$. We remark that $(s, \Omega_E=\tilde{s}\cO_K\subset \omega_E)$ is a pair consisting in a line-bundle with a marked section, so let ${\rm H}_E^\#$ be the associated submodule of ${\rm H}^1_{\rm dR}(E/\cO_K)$, as at the beginning of chapter \S 4, and let $\{f, e\}$ be a basis adapted to $(s, \Omega_E)$. Then, even if $C$ is not the canonical subgroup of $E$ ($E$ may not have a canonical subgroup at all), if $F$ is a nearly overconvergent modular form of weight $\nu:\Z_p^\ast\lra \cO_K^\ast$ which is overconvergent (i.e. defined) at $\bigl(E/D, \overline{\psi}_N, \overline{C}:=E[p]/D\bigr)$, seen as a point of $\fX(N,p)$, for every subgroup $D$ of $E[p]$ of order $p$ with $D\neq C$, then 
$U(F)$ can be evaluated at $\bigl(E/\cO_K, \psi_N, C, s, \{f,e\}\bigr)$. More precisely, we have:
$$
U(F)\bigl(E/\cO_K, \psi_N, C, s, \{f,e\}\bigr):=\sum_{D\neq C}F\bigl(E/D, \overline{\psi}_N, E[p]/C, \lambda_D^\ast(s), \{\lambda_D^\#(f), \lambda_D^\#(e)\}\bigr),
$$   
where if $D\neq C$, we denote by $\pi_D:E\lra E/D$ the canonical projection and by $\lambda_D:E/D\lra E$ the isogeny dual to $\pi_D$.
We remark that $\pi_D$ defines an isomorphism $C\cong \overline{C}:=E[p]/D$, therefore for every $D\neq C$, the tuple  
$\bigl(E/D, \overline{\psi}_N, E[p]/C, \lambda_D^\ast(s), \{\lambda_D^\#(f), \lambda_D^\#(e)\}\bigr)$ is like the ones in Definition \ref{def:katz}, so $F$ can be evaluated at all such tuples.
\end{remark}

\subsubsection{Powers of the Gauss-Manin connection.}
\label{section:powersGM}

In this section we make explicit some of the constants defined above.

We first turn to Definition \ref{def:bW}: we know the definitions of $\fw^\nu$ and $\bW_\nu$ work if $r\geq p^2$ for every $p>0$ prime integer. We will now show that the definitions make sense also if $r\geq 4$ for $p\geq 3$ and even for $r\geq 2$ for $p\geq 5$  thanks to the following:

\begin{lemma}\label{lemma:ksmallr}
Consider a $p$-adically complete and
separated $R$-algebra  $A$ with a map $\Spf(A)\to \fIG_{1,r}$. Write $h$ for
the image of  $\mathrm{Hdg}$ in $A$. Then for $x\in A$ we have
that $\nu(1+ph^{-\frac{p}{p-1}}x):=\exp\bigl(u \log(1+p h^{-\frac{p}{p-1}} x)\bigr)$ is a well defined
element of $A$, congruent to $1$ modulo $p^{\frac{1}{p-1}}$.

\end{lemma}
\begin{proof} We have
$\mathrm{v}_p\bigl(p h^{-\frac{p}{p-1}}\bigr)=1-\frac{p}{p-1}
\mathrm{v}_p\bigl(h\bigr) \geq 1 - \frac{p}{4(p-1)} \ge
 \frac{5}{4(p-1)}$ for $r\geq 4$ and $p\geq 3$. Thus
$y:=\log\bigl(1+ph^{-\frac{p}{p-1}}x\bigr)$ converges in $A$ and is divisible by $p^{\frac{5}{4(p-1)}}$ for any $x\in A$.

Similarly $\mathrm{v}_p\bigl(p h^{-\frac{p}{p-1}}\bigr) \ge
 \frac{3}{2(p-1)}$ for $r\geq 2$ and $p\geq 5$ and $y:=\log\bigl(1+ph^{-\frac{p}{p-1}}x\bigr)$ converges to an element
of $A$ divisible by $p^{\frac{3}{2(p-1)}}$ for any $x\in A$.

Recall that $\mathrm{v}_p(n!)< \frac{n}{p-1}$. Thus
$\mathrm{v}_p\bigl(y^n/n!\bigr)=n \mathrm{v}_p(y)-
\mathrm{v}_p(n!)> n
\bigl(\mathrm{v}_p(y)-\frac{1}{p-1}\bigr)\geq n
\frac{1}{4(p-1)}$. Hence $z:=\exp(uy)$ converges to an element of $A$. As $\mathrm{v}_p\bigl(y^n/n!\bigr) \geq \frac{1}{p-1}$ for $ n\geq 4$ by the previous computation and also for $1\leq n\leq 3$ by direct computation using that $p\geq 3$, we conclude that $z$ is congruent to $1$ modulo $p^{\frac{1}{p-1}} $.
\end{proof}

Next we look at Theorem \ref{thm:evaluation}.

 \begin{theorem}
\label{thm:bpr}
Let $k,\nu$ be analytic weights satisfying the Assumption before theorem \ref{thm:evaluation}  (cf. Assumption 4.1 in \cite{andreatta_iovita}), let $F\in {\rm H}^0(\fIG_{1,r}, \fw^k)$ be such that $U(F)=0$ with $r\geq p+7$.

Then, the positive integer  $b(p,r)$, whose existence is proved in Theorem \ref{thm:evaluation}, can be taken $b(p,r)=p(r-1) $ for $p\geq 5$, such that  there exists a section  $(\nabla_k)^\nu(F)$ of $\bW_{k+2\nu}$ over $\mathcal{IG}_{1,b(p,r)}$, whose $q$-expansion is described in theorem \ref{thm:evaluation}.
 \end{theorem}
 
 Before we start proving Theorem \ref{thm:bpr}, we'll make a comment regarding the usefulness of this result and then prove two results needed in its proof.
 
 \begin{remark}
 The constant $b(p,r)$ in Theorems \ref{thm:evaluation} and \ref{thm:bpr} measures the degree of overconvergence of the section
 $(\nabla_k)^\nu(F)$. In order to define the Katz-type $p$-adic $L$-function attached to a classical eigenform $f$ of level $N$ (the $F$ in both theorems is $F=f^{[p]}$), we need to evaluate this section to triples $(A, \psi, \omega)$ consisting of a CM point of the modular curve $(A, \psi)$  and an invariant differential $\omega$ of the elliptic curve $A$. Therefore the point $(A, \psi)$ has to be in the region where
 the section $(\nabla_k)^\nu(F)$ is defined, and therefore, the conductor of the $p$-adic Hecke character on which we can evaluate the $p$-adic $L$-function is related to the conductor of CM elliptic curve $A$, which is related to the 
 constant $b(p,r)$. Therefore it is important to evaluate $b(p,r)$ as precisely as possible in order to determine the domain of the $p$-adic $L$-function.
 \end{remark}

 \begin{lemma}
 \label{lemma:deltanabla}
 Let $F\in {\rm H}^0(\fIG_{1,r}, \fw^k)$ and let $r\ge p+2$. We have
 
 1) For every $N\in \N$ we have
 $$
 {\rm Hdg}^{N\frac{p+1}{p-1}}\nabla_k^N(F)\in {\rm H}^0(\fIG_{1,r}, \bW_{k+2N}).
 $$
 
2) If $U(F)=0$, for every $N\in \N$ we have
$$
{\rm Hdg}^{(p+1)pN+rN}\bigl(\nabla^{p-1}-{\rm Id} \bigr)^{pN}\in p^N{\rm H}^0\bigl( \fIG_{1,r}, \oplus_{i=0}^{(p-1)pN}\bW_{k+2i}\bigr).
$$ 
 \end{lemma}
 
 \begin{proof}
 1).  Lemma 3.20 of section \S 3.4.1 of \cite{andreatta_iovita} tells us that the sections $a,b,c,d$ defining the connection (we use the notations of that section) have the property $\displaystyle 1-a,b,c,1-d\in \frac{1}{\rm Hdg}\bW_k\otimes \xi^\ast\bigl(\Omega^1_{\fX/\Lambda}\bigr)$, where $\xi:\fIG_{1,r}'\to \fX_r$ is defined by trivializing the full $p$-torsion of the universal generalized elliptic curve over $\fX_r$, as in loc. cit..
 Moreover formula (2) at page 2027 of \cite{andreatta_iovita} implies that $\displaystyle \nabla_k(V_{k,m})\in \frac{1}{\rm Hdg}\bW_k\otimes \xi^\ast\bigl(\Omega^1_{\fX/\Lambda}\bigr)$.  In order to 
finish the proof of 1) we only need to show that the image of $\Omega_{\fIG_{1,r}/\Lambda_I}$ in $\Omega_{\fIG_{1,r}/\Lambda_I}[1/p]$ is contained in the pull-back of $\frac{1}{\rm Hdg}\Omega_{\fX/\Lambda_I}$ to $\fIG_{1,r}$.
Let $\mathfrak{U}={\rm Spf}(A)\subset \fX$ be an affine open such that $\omega_{\bE}|_{\mathfrak{U}}$ is a free $A$-module of rank one and let $h$ be a generator of the ideal ${\rm Hdg}(\mathfrak{U})$. Let us denote $\Spf(B)\subset \fX_r$ and
$\Spf(C)\subset \fIG_{1,r}$  the inverse images of $\mathfrak{U}$ in $\fX_r$ and $\fIG_{1,r}$. The relative description of these algebras is done in \cite[Lemme 3.4]{halo_spectral} as follows:
$B=A\langle Y\rangle/(h^rY-p) $ and $C:=B\langle Z \rangle/(Z^{p-1}-h)=A\langle Z,Y  \rangle/(Z^{p-1}-h, Z^{(p-1)r}Y-p)$.
Therefore $\Omega^1_{C/\Lambda_I}$ is the $C$ module generated by $dZ, dY$ and  $\Omega_{A/\Lambda_I}$, with the relations 
$\displaystyle h^r dY=-r h^{r-1} dh$ and $\displaystyle (p-1) Z^{p-2} dZ=dh$. Therefore the image of $\Omega_{C/\Lambda_I}$ in $\Omega_{C/\Lambda_I}[1/p]$ is contained in $\frac{1}{h} C\otimes_A \Omega^1_{A/\Lambda_I}$.   
 
 This implies that $\displaystyle \nabla_k(\bW_k)\subset \frac{1}{\rm Hdg}\bW_k\otimes \xi^\ast\bigl(\Omega^1_{\fX/\Lambda}\bigr)\subset \frac{1}{\mathrm{Hdg}^{\frac{p+1}{p-1}}}\bW_{k+2}$ as $\delta^2\xi^\ast\bigl(\Omega^1_{\fX/\Lambda}\bigr)\subset \fw^2$, using the Kodaira-Spencer isomorphism: $\Omega^1_{\fX/\Lambda}\cong \omega_{\bE}^2$.

 2). Let us first notice that $\bigl(\nabla_k^{p-1}-{\rm Id}\bigr)^{pN}$ is a polynomial in $\nabla_k$ of degree $(p-1)pN$, therefore if we denote by  
 $g:={\rm Hdg}^{(p+1)p N} \bigl(\nabla^{p-1}-\mathrm{Id}\bigr)^{p N}(F)$,  then, on one hand it follows from 1) above that 
 $g\in {\rm H}^0\bigl(\fIG_{1,r}, \oplus_{i=0}^{(p+1)pN}\bW_{k+2i}\bigr)$. On the other hand it follows from Theorem 4.3 of \cite{andreatta_iovita} that the restriction of $g$ to the
ordinary locus, lies in $p^{N}\mathrm{H}^0\bigl(\fIG^{\rm (\infty),ord}_{1,r},\oplus_{i=0}^{(p+1)pN} \bW_{k+2i}\bigr)$. Now using Lemma 3.4 of \cite{andreatta_iovita} we deduce that
${\rm Hdg}^{rN} g \in p^{N}\mathrm{H}^0\bigl(\fIG_{1,r},\oplus_{i=0}^{(p+1)pN} \bW_{k+2i}\bigr)$. The claim follows.

 \end{proof}
 
 \begin{lemma}\label{lemma:formula}
 Consider a positive integer $h$ and a prime $p\geq 5$. Let  $N=j_1 +\cdots + j_h$  be the sum of positive integers $j_1,\ldots,j_h$.  Then $$2+ h+ \frac{N}{p}-\sum_i v_p(j_i)-\frac{h}{p-1}>z \frac{N}{p} $$with $z:=1-\frac{1}{2p}$. 
\end{lemma}
\begin{proof}
As $N=j_1 +\cdots + j_h$ it suffices to prove this formula for $h=1$, i.e., that for every positive integer $j$
we have $2+ \frac{1}{2p^2} j >  v_p(j)+ \frac{1}{p-1}$. If $v_p(j)=0$ this is clear.  Else write $j=\gamma p^r$ with $p$ not dividing $\gamma$ and $r\geq 1$ and the inequality becomes  $$2 +
\frac{1}{2} \gamma p^{t-2} > t+ \frac{1}{p-1} .$$It suffices to prove it for $\gamma=1$. For $t=1$ or $2$ this is clear as $p\geq 3$.  For $t\geq 3$ this follows as 
$p^{t-2} > 2 (t-1)$ since $p\geq 5$. This concludes the proof.
\end{proof}
 
 We are now able to prove Theorem \ref{thm:bpr}.
 
{\it Proof of Theorem \ref{thm:bpr}} 

We recall the assumption on the analytic weights $k$ and $\nu$; they satisfy  $$k(t)=\chi'(t) \exp\bigl((a+u) \log(t)\bigr), \qquad \nu(t)=\varepsilon(t) \exp\bigl((b+s) \log(t)\bigr) \quad \forall t\in \Z_p^\ast$$ such that: i)  $a$ and $b\in \Z$, 
ii) $\chi'$ and $\varepsilon$ are characters of $(\Z/p\Z)^\ast $ and $\chi'$ is even, i.e., $\chi'=\chi^2$ for a character $\chi$.
Moreover  $u\in pR$ and $s\in p^2 R$. 

We now take $\alpha$ and $\beta \in \Z$ and write $\gamma:=\beta-\alpha+b$. Assume  that

\begin{itemize}

\item $\alpha>0$, $\gamma>0$ and $p\vert (a+2\alpha)$

\item $\beta=N p^2$ with $N$ a positive integer such that the map $ (\Z/p\Z)^\ast \to (\Z/p\Z)^\ast$, $x\mapsto x^\gamma$, coincides with the character $\chi \varepsilon\colon  (\Z/p\Z)^\ast \to (\Z/p\Z)^\ast$

\end{itemize}

The last condition implies that $x^N=\chi(x) \varepsilon(x) x^{\alpha-b}$ for every $x\in  (\Z/p\Z)^\ast$. So given $\alpha$ satisfying the first assumption, as $ (\Z/p\Z)^\ast$ is cyclic,  an $N$ with this property can be found. We also  denote by $([\ ],\langle \ \rangle)\colon \Z_p^\ast\cong \mu_{p-1}\times (1+p\Z_p)$ the canonical isomorphism.

\smallskip

We recall the definition of  $\nabla^\nu(F)$, as a nearly oveconvergent form of weight $k+2\nu$ over $\fIG_{1,p(r-1)}$. 

Let $F_0:=\nabla^\alpha\bigl(\vartheta^{\chi^{-1}-\alpha}(F)\bigr)$.  Here $\nabla^\alpha=\nabla \circ \nabla \circ \cdot \circ \nabla$ iterated $\alpha$-times. The operator $\vartheta^{\chi^{-1}-\alpha}$ is the operator defined in \cite[\S 3.8]{andreatta_iovita} associated to the finite character $\mu_{p-1}\to \mu_{p-1}$, $\zeta\mapsto \chi^{-1}(\zeta) \zeta^{-\alpha}$ using that over $\fIG_{1,r}$ we have a canonical subgroup of order $p^2$ as $r\geq p+2$.  Then $F_0$ is  a nearly overconvergent modular form of weight $k_0:=\langle u+a+2\alpha\rangle$ defined over $\fIG_{1,r}$ and such that $U(F_0)=0$ by \cite[Prop. 3.29]{andreatta_iovita}. We recall the convention used in loc.~cit.~that $\langle u+a+2\alpha\rangle(t)=\exp \bigl( (u+a+2\alpha) \log(t)\bigr)$ for any $t\in \Z_p^\ast$ 

\smallskip

Let $F_1=\nabla^{\langle s-\beta\rangle} (F_0)$. We claim that this is a nearly overconvergent form defined over $\fIG_{1,p(r-1)}$ of weight $\langle 2s-2\beta + u+a+2\alpha\rangle$. {\it Assume} this is the case and let $$\nabla^\nu(F^{[p]}):=\nabla^{\gamma} (F_1)$$where $\nabla^{\gamma}=\nabla \circ \nabla \circ \cdot \circ \nabla$ $\gamma$-times. This is a nearly overconvergent form defined over $\fIG_{1,p(r-1)}$.   Let $k'$ be the weight of $\nabla^\nu(F^{[p]})$. Write it as a weight $\gamma=\omega^\gamma \langle \gamma\rangle$ (here $\omega$ is the Teichm\"uller lift). Thus the restriction to $1+p\Z_p$ is $\langle u+a+2s+2b  \rangle$ and the restriction to $(\Z/p\Z)^\ast$ is the character $x\mapsto x^{2\gamma}$ which is $\chi' \varepsilon^2$. So $k'=k+2\nu$ as wanted. Moreover on $q$-expansions we have $$\nabla^{ \langle \gamma\rangle}(F_1)=\nabla^{\langle \gamma+s-\beta\rangle} \bigl(\nabla^{\langle\alpha\rangle}\bigl(\vartheta^{\chi^{-1}}(F)\bigr)\bigr)=\nabla^{\langle b+s\rangle} \bigl(\vartheta^{\chi^{-1}}(F)\bigr) .$$Hence,
$\nabla^\gamma(F_1)=\vartheta^{\gamma}\bigl(\nabla^{\langle b+s\rangle} \bigl(\vartheta^{\chi^{-1}}(F)\bigr)\bigr)=\vartheta^{\epsilon}\bigl(\nabla^{\langle b+s\rangle} \bigl(F^{[p]}\bigr)=\nabla^\nu\bigl(F^{[p]}\bigr)$ as wanted.

\smallskip 
To finish the proof we need to prove the claim concerning $F_1$. Notice that $w=s-\beta\in p^2 R$. We recall how  $\nabla^w(F_0)$ is constructed as a section of $\bW_{k_0+2w}$ over $\fIG_{1,p(r-1)}$ following the argument in the proof of Proposition 4.13 of \cite{andreatta_iovita}.   

From Lemma \ref{lemma:formula}, one deduces that for any positive integer $M$ we have ${\rm Hdg}^{(p+1)pM+rM}\bigl(\nabla_k^{p-1}-{\rm Id}\bigr)^{pM}=0\bigl(\mbox{mod} p^M  \bigr)$. Let $N>0$ be an integer and write $\displaystyle M:=[N/p]$, where $[ x]$ denotes the greatest integer part of the real number $x$. 
One writes $(\nabla_k)^w(F_0)$ as the limit of a sequence $(B(F_0,w)_n)_n$ of sections of $\oplus_{i=0}^{n(p-1)} \bW_{k_0+2i}$  which needs to be proven to be Cauchy. 
We recall that:
$$
B(F_0, w)_n:=\sum_{i=0}^n\frac{1}{i!}\frac{u_w^i}{(p-1)^i}\Bigl(\sum_{(j_1,j_2,...,j_i)\in H_{i,n}}\bigl(\prod_{a=1}^i\frac{(-1)^{j_a-1}}{j_a}\bigr)(\nabla^{p-1}-{\rm Id})^{j_1+...+j_i}  \Bigr)(F_0), \mbox{ for} \  n\ge 0.
$$
From  the proof of Corollary 4.11 of \cite{andreatta_iovita} and Lemma \ref{lemma:deltanabla}, for $N=j_1+...j_h$ we have: there is $\gamma$ independent of $N$ such that
$$
{\rm Hdg}^\gamma{\rm Hdg}^{((p+1)p+r)[N/p]}\bigl(\nabla^{p-1}-{\rm Id}\bigr)^{p[N/p]}(F_0)=0\bigl(\mbox{mod } p^{[N/p]} \bigr)
$$
This implies that 
$$
v_p\Bigl({\rm Hdg}^\gamma{\rm Hdg}^{((p+1)p+r)[N/p]}\bigl(B(F_0, w)_n-B(F_0, w)_{n-1}   \bigr) \Bigr) \ge 2h+\frac{N}{p}-\sum_{i=1}^hv_p(j_i)-\frac{h}{p-1}\ge z\frac{N}{p},
$$
where we used that $N/p-[N/p]<1$ and Lemma \ref{lemma:formula}. We recall that $\displaystyle z=1-\frac{1}{2p}$.
To show that the sequence $(B(F_0,w)_n)_n$ is Cauchy, when restricted to  $\fIG_{1, p(r-1)}$, we need to show that the section
$\displaystyle \frac{p^{\frac{z}{p}}}{\rm Hdg^{\frac{(p+1)p+r}{p}}}$ is nilpotent in $\cO_{\fIG_{1,p(r-1)}}$.
We remark that on $\fIG_{1,p(r-1)}$ we have $\displaystyle v_p\bigl({\rm Hdg}\bigr)\le \frac{1}{p(r-1)}$, therefore $\displaystyle v_p\bigl({\rm Hdg}^{\frac{(p+1)p+r}{p}}\bigr)\le \frac{(p+1)p+r}{p^2(r-1)}.$
It is therefore enough to show $\displaystyle \frac{z}{p}>\frac{(p+1)p+r}{p^2(r-1)}$. Putting $r=p+7+h$, where $h\ge 0$, we see that the desired inequality is equivalent to:
$(p-1/2)(p+7+h)>p^2+2p+7+h$, or $\displaystyle p\bigl(\frac{9}{2}+h\bigr)>\frac{21}{2}+\frac{3}{2}h$, which is true for all $p\ge 3$ if $h\ge 0$. Novertheless, as in the proof we used Lemma \ref{lemma:formula}, $p\ge 5$ and $r\ge p+7$ imply that $b(p,r)=p(r-1)$ works. In particular, for $p\ge 5$ and $r=p+7$ we have $b(p,r)=p(p+6)$.

\subsubsection{$p$-Adic iterations of powers of $\nabla_k$ on $\cX_1(N)_{b(p,r)}$}\label{sec:nablaucXpnp2}

Let $\pi_r\colon \mathcal{IG}_{1,r}\to  \cX_1(N)_{r}$ be the natural projection. As a consequence of Theorem \ref{thm:evaluation} we deduce that given weights $k$ and $\nu$ as in \S \ref{section:powersGM} and given $G\in {\rm H}^0(\cX_1(N)_{r}, \fw^k)$ such that $U(G)=0$, we have:

\begin{corollary}
\label{cor:xnpsquare}
There exists a unique section  $(\nabla_k)^\nu(G)$ of $\bW_{k+2\nu}$ over $\cX_1(N)_{b(p,r)}$ such that $\pi_{b(p,r)}^\ast\bigl( (\nabla_k)^\nu(G)\bigr) $ coincides with $(\nabla_k)^\nu\bigl(\pi_r^\ast(G)\bigr)$ where the latter is  defined as in  Theorem \ref{thm:evaluation}. \smallskip

If the weight $\nu$ specializes to a classical positive integral weight $\ell$ then $(\nabla_k)^u(G)$ specializes to the $\ell$-th iterate of the Gauss-Manin connection applied to $G$.
\end{corollary} 
\begin{proof} The maps $\pi_r$ and  $\pi_{b(p,r)}$ are finite \`etale  and Galois with group $A=(\Z/p\Z)^\ast$ acting on $P^{\rm univ}$. Moroever,  $\bW_{k+2\nu}=\bigl(\pi_{b(p,r),\ast}\bigl(\pi_{b(p,r)}^\ast(\bW_{k+2\nu})\bigr) \bigr)^A$ and the map $$(\nabla_k)^\nu\colon \mathrm{H}^0\bigl(\mathcal{IG}_{1,r}, \pi_{r}^\ast(\bW_k)\bigr)^{U=0}\lra \mathrm{H}^0\bigl(\mathcal{IG}_{1,b(p,r)}, \pi_{b(p,r)}^\ast\bigl(\bW_{k+2\nu}\bigr)\bigr)$$of  Theorem \ref{thm:evaluation} commutes with the action of $A$. The first claim follows. The second claim follows from the description of the $q$-expansion in Theorem \ref{thm:evaluation}.
\end{proof}

\begin{remark}
\label{remark:integral}
If in the hypothesis of Corollary \ref{cor:xnpsquare} we assume $G\in {\rm H}^0\bigl(\fX_1(N)_{r}, \fw^k\bigr)$ such that $U(G)=0$, then we obtain 
$$
\bigl(\nabla_k  \bigr)^\nu(G)\in p^{-a}\cdot {\rm H}^0\bigl(\fX_1(N)_{b(p,r)}, \bW_{k+2\nu}\bigr)\subset {\rm H}^0\bigl(\fX_1(N)_{b(p,r)}, \bW_{k+2\nu}   \bigr),
$$ where $a\in \N$ is a constant independent of $\nu$ and $G$.
\end{remark}

\medskip
\noindent
Finally we show that the $p$-adic iterates of powers of $\nabla_k$ cannot overconverge to $\cX(N,p)_{p+1}$. It has been first observed by Buzzard, Calegari that overconvergent modular forms of integer weights and infinite slope cannot overconverge too much and later by L.~Ye that this is true without the restriction on weights. 

\begin{proposition}[see also Proposition 3.1.1 in \cite{ye} and Lemma 6.13 in \cite{buzzard_calegari}]. 
\label{prop:infiniteslope}

\noindent
a) If $w\in {\rm H}^0\bigl(\cX(N,p)_{p+1}, \bW_k\bigr)^{U=0}$, then $w=0$.

b) Let $f\in {\rm H}^0(\cX(N,p), \omega_E^k)$ be a classical eigenform (so $k\in \Z$) and let $\nu$ be a $p$-adic weight. Then 
$\nabla_k^\nu(f^{[p]})\in {\rm H}^0(\cX(N,p)_b, \bW_{k+2\nu})$, with $b>p+1$.

\end{proposition}

\begin{proof}
a) Let $\bigl(E/\cO_K, \psi_N, C\bigr)$ a triple such that $v\bigl({\rm Ha}(E,\psi_N, C, \omega)\bigr)=p/(p+1)$, then this triple is as in Remark \ref{rmk:bzzcal}, for all order $p$ subgroup $C\subset E[p]$.
We have, for every $C$:
$$
(\ast)\  0=U(F)\bigl(E/\cO_K, \psi_N, C, s, \{f, e\}\bigr)=\sum_{D\neq C}F\bigl(E/D, \overline{\psi}_N, \overline{C}, \lambda_D^\ast(s), \{\lambda_D^\#(f), \lambda_D^\#(e)\}\bigr).
$$
We sum the righthand sides of these equalities over all $p+1$ subgroups $C$ of $E[p]$, divide by $p$ and obtain:  
$$
(\ast\ast)\ \sum_{D\subset E[p]}F\bigl(E/D, \overline{\psi}_N, \overline{C}, \lambda_D^\ast(s), \{\lambda_D^\#(f), \lambda_D^\#(e)\}\bigr)=0.
$$
Substructing the relation $(\ast)$ from $(\ast\ast)$ we obtain: $$F\bigl(E':=E/C, \psi_N, H:=E[p]/C, \lambda_C^\ast(s), \{\lambda_C^\#(f), \lambda_C^\#(e)\}\bigr)=0,$$ 
for all $(E', \psi_N, H)$
with valuation of its Hasse invariant=$1/(p+1)$.  By analyticity, this implies that $F=0$.

To prove b), it is enough to notice that $U\bigl(\nabla_k^\nu(f^{[p]})\bigr)=0$ by Appendix \ref{sec:nablaUV}, so we apply a).
\end{proof}

\begin{remark}
We remark that it is easier to obtain b) of the above Proposition \ref{prop:infiniteslope}, at least in the following weak form:
let $f$ be a classical eigenform of weight $k$ on $\cX(N,p)$, then $f^{[p]}$ is an overconvergent modular form on $\cX(N,p)$ of infinite slope so by Propsition 3.1.1 of \cite{ye}, $f^{[p]}\in {\rm H}^0(\cX(N,p)_r, \omega_E^k)$, with $r>p+1$. By the way $\nabla_k^\nu(f^{[p]})$ is constructed, we have $\nabla_k^\nu(f^{[p]})\in {\rm H}^0(\cX(N,p)_b, \bW_{k+2\nu})$, with $b\ge r>p+1$.

By applying Proposition \ref{prop:infiniteslope}, we have the stronger result: $\nabla_k^\nu(f^{[p]})$ cannot be analytically continued to $\cX(N,p)_{p+1}$. 

\end{remark}

\subsection{The $p$-depletion operator on classical and overconvergent modular forms.}
\label{section:depletion}

The input for  Corollary \ref{cor:xnpsquare} are overconvergent forms $G$ of weight
$k$ and level $\Gamma_1(N)$ defined over the rigid analytic fiber $\cX_{r}$ of $\fX_{r}$ and such that $U(G)=0$ (i.e. $G$ has infinite slope). 
We will now recall how to obtain a large supply of such overconvergent modular forms of infinite slope.

Consider an overconvergent form $F$ of weight $k$ and level $\Gamma_1(N)$. One has operators  $U$, $V$ that can be defined geometrically on the overconvergent modular forms; see \cite[\S 3.6\& \S 3.7]{andreatta_iovita}. On $q$-expansions if $F(q)=\sum_{n=0}^\infty a_nq^n$ then 
$$
U\bigl(F(q)\bigr):=\sum_{n=0}^\infty a_{np}q^n\mbox{ and } V\bigl(F(q)\bigr):=\sum_{n=0}^\infty a_nq^{pn}=F(q^p).
$$
It is then easy to see that $U\circ V={\rm Id}$ and we define $F^{[p]}:=F-(V\circ U)(F)=\bigl((U\circ V)-(V\circ U)\bigr)(F)$. It is an overconvergent modular form $F^{[p]} \in {\rm H}^0(\cX_1(N)_{r}\times \cU, \fw^{k^{\rm univ}})$ defined for some $r$ and we have:
$$
F^{[p]}(q)=\sum_{n=1, (n,p)=1}^\infty a_nq^n\mbox{ and } U(F^{[p]})(q)=U\bigl(F^{[p]}(q)  \bigr)=0. 
$$

Recall from the Katz-Lubin theory of the canonical subgroup that the universal generalized elliptic curve over $\cX_{p+2}$ admits a canonical subgroup of order $p^2$. Thus, $\cX_{p+2}$  can be identified with the strict neighborhood of the ordinary locus $\cX(N,p^2)_{p+2}$  of $\cX(N,p^2)$, the rigid space associated to the modular curve $X(N,p^2)$  over $\Q_p$ of level $\Gamma_1(N)\cap \Gamma_0(p^2)$.  We prove that, if $F$ is a classical modular form,  the  $p$-depletion $F^{[p]}$  can be defined on the full  $\cX(N,p^2)$, i.e. $F^{[p]}$ is a classical modular form of level $\Gamma_1(N)\cap \Gamma_0(p^2)$. 
First of all we define the operators $V^2$, respectively $V$  restricted to classical modular forms as the pull back via the morphisms $ X(N,p^2) \to X_1(N)$ and  $ X(N,p^2) \to X_1(N,p)$ over $\Q_p$,
defined by $(E/S,\psi_N,D)\mapsto (E/D,\overline{\psi}_N)$, respectively $(E/S,\psi_N,D)\mapsto (E/D[p],\overline{\psi}_N, D/D[p])$; here  $(E/S,\psi_N,D)$ is an elliptic curve over $S$ with $\Gamma_1(N)$-level structure $\psi_N$ and 
a cyclic subgroup $D$ of order $p^2$ (respectively $p$) and $\overline{\Psi}_N$ is the $\Gamma_1(N)$-level structure induced by $\psi_N$ on $E/D$, respectively on $E/D[p]$, via the quotient map $E\to E/D$, respectively $E\to E/D[p]$. 

Over $\cX_1(N)_{p+2}$ the universal  elliptic curve admits a canonical subgroups $C$ of order $p^2$ (respectively $p$)  we can identify $\cX_1(N)_{p+2}$ with an open of $\cX(N,p^2)$, resp.~$\cX(N,p)$ by taking $C$, resp.~$C[p]$ as level subgroups. Hence, the operators $V^2$ and $V$ on this space give the operators $V^2$ and respectively $V$ defined in  \cite[\S 3.7]{andreatta_iovita} which act on $q$-expansions as described above.

\begin{corollary}\label{cor:classicaldepletion} If $F$ is a classical modular form of weight $k\in \Z_{>0}$ of level $\Gamma_1(N)$, character $\epsilon$, 
then $F^{[p]}:=F\vert \bigl(1- V U
\bigr)$ is 
a classical modular form of level $\Gamma_1(N)\cap \Gamma_0(p^2)$ where $U$ is induces by the $U$-correspondence  on the pull-back of $F$ to $\cX(N,p)$ and $V$ is the operator induced by the morphism  $V\colon X(N,p^2) \to X_1(N,p)$ above. The restriction of $F^{[p]}$ to $\cX_1(N)_{p+2}$, via the identification $\cX_1(N)_{p+2}\subset  \cX(N,p^2)$, gives the $p$-depletion of $F$ described above on $q$-expansions. In particular, $U(F^{[p]})=0$. Finally, if $F$ is an eigenform for the operator $T_p$ with eigenvalues $a_p$ we also have  $F^{[p]}=F\vert \bigl(1-a_pV+\epsilon(p)p^{k-1}V^2
\bigr)$. 
\end{corollary}
\begin{proof} 
The claim follows from the discussion above and  the following simple calculation. Suppose  that  $T_p(F)=a_p(F)$. Then
$ a_pF=F\vert T_p=F\vert U +\epsilon(p)p^{k-1} F\vert V$ and therefore 
$$ F^{[p]}=F\vert (1-V U)=F\vert \bigl(1-a_pV+\epsilon(p)p^{k-1}V^2
\bigr).$$The fact that $U(F^{[p]})=0$ follows from the fact that it is true on $\cX_1(N)_{p+2}$ and the $U$ correspondences on $\cX_1(N)_{p+2}$ and $\cX(N,p^2)$  are compatible via the identification $\cX_1(N)_{p+2}\subset  \cX(N,p^2)$.

\end{proof} 

Assume more generally that $F$ is a generalized eigenform of weight $k$ for the operator $U$ of finite slope. As $U$ increases the radius of overconvergence,  we may and will assume that $F$ is defined on  $\cX_1(N)_{4}$ and even on  $\cX_1(N)_{2}$ for $p\geq 5$ thanks to Definition \ref{def:bW}. Then, its $p$-depletion ${\bf F}^{[p]}=(1-UV)({\bf F})$ will be defined over $\cX_1(N)_{r}\times U$ with $r=4p$ and $r=2p$ for $p\geq 5$.

\begin{definition}\label{definition:b(p)} Given an analytic weight $k$,  we let $$r_k(p)=\left. \begin{cases} 2p  & p\geq 5\cr  4p  &  p=3  \cr   p+2 &   k\in \Z_{>0} \cr     \end{cases}\right. $$
Let $b_k(p):=b\bigl(p,r_k(p)\bigr)$  be the integer $b(p,r)$ of Theorem \ref{thm:evaluation} with $r=r_k(p)$. We also let $n_k(p)\in\N$ be the smallest positive integer such that 
\begin{itemize} 
\item[i.] $p^{n_k(p)-1}(p+1)\ge b_k(p)$ if $p$ is inert in $K$. 

\item[ii.] $2 p^{n_k(p)}\ge b_k(p)$ if $p$ is ramified in $K$. 
\end{itemize}

In particular, we have $b_k(p)=p (r_k(p)-1)$ for $p\geq 5$. Thus, $n_k(p)=2$ for $k$ a classical weight and  $p\geq 5$ or for $p\geq 5$ ramified and $n_k(p)=3$ for $k$ non classical and $p\geq 5$.

\end{definition}

\section{The case: $p$ is inert in $K$.}

Let  $k^{\rm univ}$ be the universal weight of some rigid analytic disk $\cU$ in $W$ such that $k^{\rm univ}\vert_{\cU}$ is analytic in the sense of Definition \ref{def:analweight}. Consder a finite slope family ${\bf F}$ of nebentype $\epsilon$ and weight $k^{\rm univ}\vert_{\cU}$. As explained in the discussion before Definition \ref{definition:b(p)}  we may and will assume that  its $p$-depletion ${\bf F}^{[p]}=(1-UV)({\bf F})$ is  defined over $\cX_1(N)_{r}\times \cU$ with $r=r_{k^{\rm univ}}(p)$ as in Definition \ref{definition:b(p)}.

We assume that there is an integral weight $u\in \cU(\Q_p)$ such that the specialization ${\bf F}_u$ is an overconvergent modular form which arises from a classical eigenform $F$ of weight $u$ and nebentype $\epsilon$ that satisfies the assumptions of \S \ref{sec:conclusions}.  We suppose that  $p$ is inert in $K$.

\subsection{Evaluation at {\rm CM}-points.}\label{sec:evCMinert}

We fix:

\begin{itemize}

\item[a.] an elliptic curve $E$ over the ring of integer of a finite extension of $\Q_p$ with full CM by $\cO_d$, with $d$ prime to $pd_K$;
\item[b.] a $\Gamma_1(\fN)$-level structure $\psi_N$ on $E$; 
\item[c.] a subgroup $C^{(n)}\subset E[p^n]$ of order $p^n$ which is generically cyclic for $n\geq n_0$ with $n_0=n_{k^{\rm univ}}(p)$ as in Definition  \ref{definition:b(p)}. 

\end{itemize}

\noindent Notice that for $p\geq 5$ it suffices to take  $n\geq 3$. Consider the elliptic curve $E^{(n)}:=E/C^{(n)}$ with projections $\lambda_n\colon E\to E^{(n)}$.  Then 

\begin{itemize}

\item[i.] $E^{(n)}$ has CM by $\cO_c$ with $c=p^n d$;
\item[ii.] $\psi_N^{(n)}:=\lambda_n\circ  \psi_N$ defines $\Gamma_1(\fN_c)$-level structure on $E^{(n)}$ with $\fN_c=\cO_c\cap \fN$;
\item[iii.]  $H^{(n)}=E[p^n]/C^{(n)}$ defines canonical subgroup  of level $p^n$ of $E^{(n)}$ (see Lemma \ref{Lemma:LTI});

\end{itemize}

In particular, we define the $L$-valued point $x^{(n)}:=\bigl(E^{(n)},\psi_N^{(n)},H^{(n)}[p^2]\bigr) $ of $\cX(N,p^2)$ over some finite extension $L$ of $\Q_p$. If $\cO_L$ is the ring of integers of $L$,  thanks to Lemma \ref{Lemma:LTI}, $x^{(n)}$ extends to an $\cO_L$-valued point, $\varphi_{x^{(n)}}$, of the formal model $\fX_1(N)_{b(p)}$ of the open $\cX_1(N)_{b(p)}\subset \cX(N,p^2)$ for $b(p)$ as  in Definition  \ref{definition:b(p)}.

\

We denote $C':=C^{(n)}[p]$ and $E':=E/C'$. Factor $$\lambda_n\colon E\stackrel{\lambda'}{\lra} E' \stackrel{\lambda_n'}{\lra}  E^{(n)}.$$  Thanks to Lemma \ref{Lemma:LTI} the elliptic curve $E'$ admits a  canonical subgroup  ${\rm H}'=E[p]/C$ and together with the level $N$ structure $\psi_N'=\lambda'\circ \psi_N$ it defines an $L$-valued point $x'=\bigl(E',\psi_N')$ of  $\cX_1(N)_{p+1}$ that extends to an $\cO_L$-valued point $\varphi_{x'}$ of $\fX_1(N)_{p+1}$.

\subsection{Splitting the Hodge filtration.}\label{sec:splitinginert}

We keep the notations of the previous section. By functoriality of VBMS (see \S \ref{remark:funct}),  the isogeny $\lambda_n'\colon E'\to E^{(n)}$ induces morphisms $(\lambda_n')^\ast\colon {\rm H}_{E^{(n)}}^\sharp \to {\rm H}_{E'}^\sharp$. On the other hand $E'$ has CM by $\cO_{pd}$. Possibly enlarging $L$,   we have two embeddings $\tau$, $\overline{\tau}\colon \cO_{pd}\to \cO_L$. We define ${\rm H}_{E',\tau} \oplus {\rm H}_{E',\overline{\tau}}\subset {\rm H}_{E'}$ as in \S \ref{sec:technicalsec} and $$\widetilde{{\rm H}}_{E'}^\sharp:=H_{E',\tau}^\sharp  \oplus {\rm H}_{E',\overline{\tau}}^\sharp:=\delta_{E'}{\rm H}_{E',\tau} \oplus \delta_{E'}^p{\rm H}_{E',\overline{\tau}}  \subset {\rm H}_{E'}^\sharp.$$

\begin{lemma}\label{lemma:factorlambdai} The image of ${\rm H}_{E^{(n)}}^\sharp$ via $(\lambda_n')^\ast$ is contained in $\widetilde{{\rm H}}_{E'}^\sharp$.
\end{lemma}
\begin{proof} As $\lambda_n'$ induces an isomorphism $H'\to H^{(n)}[p]$ of canonical subgroups of level $p$, the map $(\lambda_n')^\ast$ induces an isomorphism $ \Omega_{E^{(n)}}\cong \Omega_{E'}={\rm H}_{E',\tau}^\sharp$.
Thus it suffices to prove that the map $\gamma_n$  induced on the quotients ${\rm H}_{E^{(n)}}^\sharp/ \Omega_{E^{(n)}} \to  {\rm H}_{E'}^\sharp/\Omega_{E'}$ factors through ${\rm H}_{E',\overline{\tau}}^\sharp$.

Recall that ${\rm H}_{ E^{(n)}}^\sharp/ \Omega_{ E^{(n)}}=\delta_{ E^{(n)}}^p \omega_{ E^{(n)}}^\vee$ and  ${\rm H}_{E'}^\sharp/\Omega_{E'}=\delta_{E'}^p \omega_{E'}^\vee$ and by construction the map between them is induced by the map ${\rm Lie}\bigl( (\lambda_n')^\vee\bigr)\colon \omega_{ E^{(n)}}^\vee\to \omega_{E'}^\vee$, the $\cO_L$-dual to the map of differentials $\omega_{E'}\to \omega_{ E^{(n)}}$ defined by pull-back via the dual isogeny $(\lambda_n')^\vee$. 

The map $\lambda_n'$ factors via $\lambda_2'\colon E'\to E^{(2)}$ with $E^{(2)}:=E/\bigl(C^{(n)}[p^2]\bigr)$ so that $\gamma_n$ factors through the map $\gamma_2\colon {\rm H}_{E^{(2)}}^\sharp/ \Omega_{E^{(2)}} \to  {\rm H}_{E'}^\sharp/\Omega_{E'}$  associated to $\lambda_2'$. We are then left to  prove the claim on $\gamma_n$ for $n=2$.

The map $E'\to E^{(2)}$ has kernel of degree $p$ that does not intersect the canonical subgroup. It then follows  from Lemma \ref{lemma:splittingHdR} that

\smallskip

(i) ${\rm Hdg}(E')  \omega_{E'}^\vee \subset {\rm H}_{E',\overline{\tau}} \subset \omega_{E'}^\vee$;

\smallskip

(ii) ${\rm Lie}\bigl((\lambda_2')^\vee\bigr)\bigl(\omega_{ E^{(2)}}^\vee\bigr)   \subset p {\rm Hdg}(E^{(2)})^{-1} \omega_{E'}^\vee $.

\smallskip

Thus it suffices to show that $\delta_{ E^{(2)}}^p p {\rm Hdg}( E^{(2)})^{-1} \omega_{E'}^\vee \subset {\rm Hdg}(E') \delta_{E'}^p \omega_{E'}^\vee $. This amounts to prove that ${\rm val}_p(\delta_{E^{(2)}}^p)\geq {\rm val}_p(\delta_{E'}^{2p-1})$ (we use that ${\rm val}_p({\rm Hdg}( E^{(2)}))=(p-1) {\rm val}_p(\delta_{E^{(2)}})$ and likewise for $E'$). Since ${\rm val}_p(\delta_{E'})=p{\rm val}_p(\delta_{ E^{(2)}})$, this amounts to show that $\bigl((2p-1)p-1\bigr) {\rm val}_p(\delta_{E^{(2)}}) \leq 1$. But $ {\rm val}_p(\delta_{E^{(2)}})=\frac{1}{p(p+1)(p-1)}$ by Lemma \ref{Lemma:LTI}. Hence we need to show that $2p^2-p-1 \leq p^3-p$ or equivalently $2p^2-1 \leq p^3$ which is true for any prime $p$.

\end{proof}

Take $L$ large enough so that $({\rm H'})^D(L)=\Z/p\Z$. Let $\omega'$ be an element of $\Omega_{E'/\cO_L}$ reducing to an element in the image of $s'\in {\rm dlog}\bigl(({\rm H}')^D(L)\bigr)$ modulo $  p\underline{\delta}(E')^{-p}\bigr)$ (see Section \ref{sec:cX}). 

Let us now fix $k$ and $\nu\in \cU(\Q_p)$. As explained in \S \ref{section:local}, the choice of $\omega'$ defines a  trivialization  $$v_{\omega'}\colon \fw_{E', \cO_L}^{k+2\nu}\stackrel{\sim}{\lra} \cO_L.$$ 
Let $\omega_n$ be the generator $\omega_{E^{(n)}}\otimes_{\cO_L}L$ whose pull-back via $\lambda_n'\colon E'\to E^{(n)}$ is $\omega'$. As $\lambda_n'$ defines an isomorphism ${\rm H}'\cong {\rm H}^{(n)}[p]$ of canonical subgroups, $s'$ defines a section $s_n \in {\rm dlog}\bigl(({\rm H}^{(n)}[p])^D(L)\bigr)$  and  $\omega_n$ is a generator of $\Omega_{E^{(n)}/\cO_L}$ reducing to $s_n$ modulo $  p\underline{\delta}(E^{(n)})^{-p}\bigr)$ .

It follows from Definition \ref{def:bW} that both $x'$ and $x^{(n)}$ are points of $\cX_1(N)_{p+1}$ over which   $\bW_{k+2\nu}$ is defined.  By functoriality of  the VBMS, see Remark \ref{rmk:functVBMS}, we have that $\varphi_{x'}^\ast(\bW_{k+2\nu})= \bW_{k+2\nu,\cO_L}\bigl({\rm H}_{E'}^\sharp, s \bigr) $ and  $\varphi_{x^{(n)}}^\ast(\bW_{k+2\nu})=\bW_{ k+2\nu,\cO_L}\bigl({\rm H}_{E^{(n)}}^\sharp, s_n \bigr)$.

The inclusion $\widetilde{{\rm H}}_{E'}^\sharp \subset  {\rm H}_{E'}^\sharp$  identifies  $
\bW_{k+2\nu,\cO_L}\bigl(\widetilde{{\rm H}}_{E'}^\sharp, s \bigr)$ as a submodule of $\bW_{k+2\nu,\cO_L}\bigl({\rm H}_{E'}^\sharp, s \bigr)$. As  the Hodge filtration on $\widetilde{{\rm H}}_{E'}^\sharp$ is split, also the filtration on  $\bW_{k+2\nu,\cO_L}\bigl(\widetilde{{\rm H}}_{E'}^\sharp, s \bigr)$ is canonically split. In particular, it has a canonical
$\cO_{pd}$-equivariant projection 

$$
\Psi_{E'}\colon \bW_{k+2\nu,\cO_L}\bigl(\widetilde{{\rm H}}_{E'}^\sharp, s \bigr) \longrightarrow \fw_{E',\cO_L}^{k+2\nu}.
$$

Thanks to Lemma \ref{lemma:factorlambdai} the image of $(\lambda_n')^\ast\colon \bW_{ k+2\nu,\cO_L}\bigl({\rm H}_{E^{(n)}}^\sharp, s_n \bigr) \to  \bW_{k+2\nu,\cO_L}\bigl({\rm H}_{E'}^\sharp, s \bigr)$ factors through 
$\bW_{k+2\nu,R'}\bigl(\widetilde{{\rm H}}_{E'}^\sharp, s \bigr)$  so that we get a natural projection $$\Psi_{E'}\bigl((\lambda_n')^\ast \circ \varphi_{x_i}^\ast \bigl((\nabla_k)^{\nu}(F^{[p]})\bigr)\bigr) \in  p^{-a}\fw_{E',\cO_L}^{k+2\nu}\subset \fw_{E',\cO_L}^{k+2\nu}\bigl[p^{-1}\bigr].$$We have remarked that $x^{(n)}$ is a point of $\cX_1(N)_{b(p)} $ so that $(\nabla_k)^\nu(F^{[p]}) $ can be indeed evaluated at $x^{(n)}$ by the discussion before Definition \ref{definition:b(p)}. 
We recall also that the constant $a$ was defined in Remark \ref{remark:integral}.

\begin{definition}
\label{def:evaluationCM}
We define 
$$
\delta_k^\nu(F^{[p]})(E^{(n)}/\cO_L, \psi_N, \omega_n):= v_{\omega'}\circ \Psi_{E'}\Bigl((\lambda_n')^\ast \circ \varphi^\ast_{x^{(n)}}\bigl((\nabla_k)^\nu(F^{[p]})   \bigr)  \Bigr)\in p^{-a}\cO_L\subset L.
$$

\end{definition}

\begin{remark}\label{rmk:Omegacan}
(1) Let $\Omega^{\rm can}_{E'/\cO_L} \subset  \Omega_{E/\cO_L}$ be the $\Z_p^\ast (1+p\underline{\delta}(E')^{-p}\cO_L)$-torsor  of sections arising via $(\lambda^\vee)^\ast$ from sections of $\Omega_{E'/\cO_L}$ reducing to  the image of ${\rm dlog}\bigl({\rm H}')^D(L)\bigr)\backslash\{0\}$  modulo 
$p\underline{\delta}(E')^{-p}$ and similarly for $E^{(n)}$. As the map $\lambda_n'\colon E'\to E^{(n)}$ induces an isomorphism of level subgroups of level $p$, then the map on differentials $(\lambda_n')^\ast$ gives an isomorphism $$(\lambda_n')^\ast\colon \Omega^{\rm can}_{E^{(n)}/\cO_L} \lra \Omega^{\rm can}_{E'/\cO_L}$$so that $\omega_n$ is a generator of $\Omega^{\rm can}_{E^{(n)}/\cO_L}$.

\

(2) Let $D_1$ and $D_2\subset E[p^n]$ be two  subgroups of order $p^n$, generically cyclic. Let
$\lambda_1'\colon E\to E_1'=E/D_1[p]$ and $\lambda_2'\colon E\to E_2'=E/D_2[p]$ be the two corresponding cyclic isogenies of degree $p$, as above. We claim that there exists $a\in  \widehat{\cO}_d^\ast$ such that $[a] (D_1)=D_2$ and, hence, $$\Omega^{\rm can}_{E_1'/\cO_L}= [a]^\ast \bigl(\Omega^{\rm can}_{E_2'/\cO_L}\bigr)=a^{-1} \cdot \Omega^{\rm can}_{E_2'/\cO_L}.$$Indeed,  $\bigl(\cO_d\otimes \Z_p\bigr)^\ast$ acts transitively on the set of subgroups of $E$, generically cyclic, of order $p^n$. This follows as $E$ is defined over a dvr so that there is a $1:1$ correspondence between these subgroups and the cyclic subgroups of $E[p^n]_L$, given by taking the generic fiber in one direction and taking the schematic closure in the other. Then one concludes remarking that $\bigl(\cO_d\otimes \Z_p\bigr)^\ast$acts transitively on the elements of order $p^n$ in $E[p^n](\overline{\Q}_p)\cong \cO_K/p^n \cO_K$ since $p$ does not divide $d$.

\end{remark}

\subsection{Definition of the $p$-adic $L$-functions in the inert case.}\label{sec:padicLinert}

We define both one and two-variable $p$-adic $L$-functions. 

\bigskip
\noindent
{\it Two-variable $p$-adic $L$-functions.}
\medskip

The notation is as at the beginning of this section.  We denote $\widehat{\Sigma}^{(2), p^n}(c, \fN, \epsilon)_\cU\subset \widehat{\Sigma}^{(2)}(c, \fN, \epsilon)$ the open subspace  of $p$-adic Hecke characters $\chi$ in the space $ \widehat{\Sigma}^{(2)}(c, \fN, \epsilon)$ defined in \S \ref{sec:Sigmahat} with the property that $\widetilde{w}(\chi)=(x,y) \in  \cU(\Q_p)\times \cU(\Q_p)$ and the $p$-part of the conductor is $p^n$ with $n\geq n_0$ (here  $n_0=n_{k^{\rm univ}}(p)$ is as in Definition  \ref{definition:b(p)};  we can simply take $n\geq 3$ for $p\geq 5$).

\begin{definition}\label{def:Lpinert2var} For every $\chi\in \widehat{\Sigma}^{(2),p^n}( d,\fN,\epsilon)_\cU$ with weight  $\widetilde{w}(\chi)=(u, \nu)$,  set
$$
L_p({\bf F},\chi):=\frac{1}{\vert (\cO_K/\fN)^\ast\vert} \sum_{\fa\in \cH(c,\fN)} \chi_\nu^{-1}(\fa)\delta_u^\nu\bigl(({\bf F}_u)^{[p]}\bigr)
\bigl(\fa\ast(E^{(n)}/\cO_L, \psi_N^{(n)},  \omega_n)\bigr),
$$
where the evaluation of $(\nabla_u)^\nu(({\bf F}_u)^{[p]})$ at $\fa\ast(E^{(n)}/\cO_L, \psi_N^{(n)},  \omega_n)$ is done using 
Definition \ref{def:evaluationCM} and \S \ref{sec:adelicdescr}.
\end{definition}

\begin{remark}\label{rmk:astwelldefined}
a) More precisely, what we mean by using Definition \ref{def:evaluationCM} to evaluate $(\nabla_u)^\nu(({\bf F}_u)^{[p]})$ at $\fa\ast(E^{(n)}/\cO_L, \psi_N^{(n)},  \omega_n)$ is the following.  As explained  in \S \ref{sec:adelicdescr},  for every   $\fa\in \cH(c,\fN)$  we set $$\fa\ast (E^{(n)}/\cO_L, \psi_N^{(n)},  \omega_n)=\bigl(\fa\ast (E^{(n)}/\cO_L, \psi_N^{(n)}), \iota_p(\fa_p)^{-1} \omega_{n,\fa}\bigr);$$
where we denoted $\omega_{n,\fa}$ the differential whose  pull-back  via the natural isogeny $E^{(n)}\to E^{(n)}_{\fa}:=E^{(n)}/E^{(n)}[\fa]$ is $\omega_n$. Notice that if we write $(E_\fa^{(n)},  \psi_{N,\fa}^{(n)}) $ for $\fa\ast (E^{(n)}, \psi_N^{(n)})$ then $\iota_p(\fa_p)^{-1} \omega_{\fa}\in \Omega_{E'_\fa/\cO_L}^{\rm can}$ by  Remark \ref{rmk:Omegacan}. 
Let $x_\fa^{(n)}:=(E^{(n)}_\fa, \psi_{N, \fa}^{(n)})$ seen as an $L$ point of $\cX$. Similarly set $ \fa\ast \bigl(E', \psi_N', \omega'\bigr)=:\bigl(E'_\fa, \psi_{N,\fa}', \omega'_\fa\bigr)$ we let $x'_\fa:=(E_\fa', \psi_{N, \fa}')\in \cX(L)$. We have a canonical isogeny $\lambda_{n,\fa}'\colon E_\fa'\lra E^{(n)}_\fa$ induced by $\lambda_n\colon E'\lra E^{(n)}$ and Lemma \ref{lemma:factorlambdai}
implies that 
$$
(\lambda_{n,\fa}')^\ast\bigl(\varphi_{x^{(n)}_\fa}^\ast(\nabla_u^\nu({\bf F}_u^{[p]} ) ) \bigr)\in \bW_{u+2\nu}\bigl(\tilde{{\rm H}}_{E_\fa'}^\sharp, s_\fa \bigr),
$$
which is split. Applying the splitting $\Psi_{E'_\fa}$ to $(\lambda_{n,\fa}')^\ast\bigl(\varphi_{x^{(n)}_\fa}^\ast(\nabla_u^\nu({\bf F}_u^{[p]} ) ) \bigr)$ and using the generator $(\omega_\fa')^{u+2\nu}$ we obtain an element of $L$.

b) The formula above is well posed, namely if we multiply $\fa$ by an element $r\in \widehat{\cO}_{K,p}^\ast$ which is congruent to $1$ modulo $p^n$, then one has
$$\chi_\nu^{-1}(\fa)(\nabla_u)^\nu\bigl({\bf F}_u^{[p]}\bigr) \bigl(\fa\ast (E^{(n)}/\cO_L, \psi_N^{(n)},  \omega_n)\bigr)=\chi_\nu^{-1}(r \fa)(\nabla_u)^\nu\bigl({\bf F}_u^{[p]}\bigr)
\bigl((r\fa)\ast (E^{(n)}/\cO_L, \psi_N^{(n)},  \omega_n) \bigr) .$$ Indeed $$\bigl((r\fa)\ast (E^{(n)}/\cO_L, \psi_N^{(n)},  \omega_n)\bigr)=\bigl(\fa\ast (E^{(n)}/\cO_L, \psi_N^{(n)},\iota_p(r_p)^{-1}  \omega_n)  $$ so that 
$(\nabla_u)^\nu\bigl({\bf F}_u^{[p]}\bigr)
\bigl((r\fa)\ast(E^{(n)}/\cO_L, \psi_N^{(n)},  \omega_n) \bigr) =(u+2\nu)\bigl(\iota_p(r)\bigr)(\nabla_u)^\nu\bigl({\bf F}_u^{[p]}\bigr)
\bigl(\fa\ast (E^{(n)}/\cO_L, \psi_N^{(n)},  \omega_n)  \bigr) $. The conclusion follows as $\chi_\nu(r)=(u+2\nu)\bigl(\iota_p(r)\bigr)$ by construction.

\end{remark}

We also have:

\begin{proposition}\label{prop:localanalinert}
The function $L_p({\bf F}, - )\colon \widehat{\Sigma}^{(2), p^n}(c, \fN, \epsilon)_\cU \lra \C_p$ is locally analytic, i.e. $L_p({\bf F}, \chi)$ is locally analytic in the variable 
$\chi$ considering the analytic structure on $\widehat{\Sigma}^{(2),p^n}(c, \fN, \epsilon)$ defined in Lemma \ref{lemma:w}.
\end{proposition}

\begin{proof} Let $\chi\in \widehat{\Sigma}^{(2),p^n}(c, \fN, \epsilon)_\cU$ and let $(a,b)\in \bigl((\Z/6(q-1)\Z)\times \Z_p\bigr)^2$ be its image by $(w', w)$. Given the way the toplogy of $\widehat{\Sigma}^{(2)}(c, \fN, \epsilon)_\cU$ is defined, it is enough to suppose that the components of $(a,b)$ in $\Z/6(q-1)\Z$ are fixed and so WLOG we may assume the torsion components are $0$. So $(a,b)\in \Z_p^2$
and the weights in $(u,\nu)\in W(\Q_p)^2$ associated to $\chi$ are $u(t)=\exp(a\log(t))$, $\nu(t)=\exp(b\log(t))$, for all $t\in \Z_p^\ast$. 

Let us first remark that in the expression of $L_p({\bf F}, \chi)$, for a fixed $\fa\in \cH(c, \fN)$, 
$\chi_\nu^{-1}(\fa)$ is analytic in $\chi$ and $\nu$. So we'll analyse the rest of the formula, again for a fixed $\fa$.

It is enough to work over $\mathcal{IG}_{1,b(p)}$ and we recall that we have the points $x_{\fa}^{(n)}=\fa\ast(E^{(n)}, \psi_N^{(n)})$ and $x_\fa'\ast (E', \psi_N')$ of $\cX_{b(p)}(L)$, which we regard as points of $\mathcal{IG}_{1,b(p)}(L)$ by choosing a generator of the dual of the canonical subgroup of $\bE$, i.e. a marked section. Let $T:={\rm Spa}(R, R^+)\subset \mathcal{IG}_{1,b(p)}$ be an affinoid containing $x_\fa^{(n)}, x_{\fa}'$ and such that $\omega^+_{\bE}\vert_T$ is a free $R^+$-module of rank $1$.
We choose basis $f$, $e$ of $({\rm H}_{\bE}^\sharp)\vert_T$ such that $f$ lifts the marked section $s$. As explained in \S \ref{section:local}  this gives us coordinates $Z$ and $V=\frac{Y}{1+\beta Z}$ of $\bW\vert_{T\times \cU}$ such that ${\bf F}^{[p]}\vert_{T\times \cU}=\alpha k^{univ}(1+\beta Z)$, with $\alpha\in \cO_{T\times \cU}^+(T\times \cU)$. Then 
$$
{\bf F}^{[p]}_u=\alpha_u k^{\rm univ}(u)(1+\beta Z)=\alpha_u(1+\beta Z)^a,
$$
 with $\alpha_u\in R^+$. On the other hand $\displaystyle \nabla_u^{\nu}= \exp\bigl(\frac{b}{p-1}\log(\nabla_u^{p-1}-{\rm Id})\bigr)$, therefore, using that $k^{\rm univ}(u)(t)=u(t)$ for all $t\in \Z_p^\ast$ and the formulae giving the expression of the connection in local coordinates (see \cite[Prop. 4.15]{andreatta_iovita}, we have:
$$
\nabla_u^\nu({\bf F}_u^{[p]})=\Bigl(\sum_{n=0}^\infty A_n(a,b)V^n\Bigr)(1+\beta Z)^{a+2b},
$$
where $A_n(a,b)$ are power series in $a$ and $b$ with coefficients in $R$. 

Now we choose adapted basis of ${\rm H}^\sharp_{E^{(n)}}$ and 
$\widetilde{{\rm H}}_{E'}:={\rm H}^\sharp_{E', \tau}\oplus {\rm H}^\sharp_{E', \overline{\tau}}$ which give us coordinates $V_n$, $Z_n$ of $\bW\bigl({\rm H}^\sharp_{E^{(n)}}, s^{(n)}\bigr)$ and $V'$, $Z'$ of $\bW\bigl(\widetilde{\rm H}^\sharp_{E'}, s'\bigr)$.
Evaluating ${\rm H}^\sharp_{\bE}\vert_A$ at $x_\fa^{(n)}$ means making a linear change of coordinates and the evaluation of the coeffcients in $R=\cO_T(T)$ at $x_{\fa}^{(n)}$, then further applying $\varphi_{x'_{\fa}}^\ast$ means again a linear change of variables. We obtain:
$$
\varphi_{x_{\fa}'}^\ast\Bigl(\varphi_{x_\fa^{(n)}}^\ast\bigl(\nabla_u^\nu\bigl({\bf F}_n^{[p]})\bigr)\bigr)\Bigr)=\Bigl(\sum_{n=0}^\infty B_n(a,b)(V')^n\Bigr)(1+\beta Z')^{a+2b}
$$ 
with $B_n(a,b)$ power series in $a,b$ with coeffcients in $L$.
Therefore the splitting of the Hodge filtration sends this to $\delta_u^\nu({\bf F}_u^{[p]})(\fa\ast(E^{(n)}, \psi_N^{(n)}))=B_0(a,b)(1+\beta Z')^{a+2b}$ and denoting $\gamma\in \cO_L^\ast$ the unit such that 
$1+\beta Z'=\gamma\omega_\fa'$ we have
$$
\delta_u^\nu({\bf F}_u^{[p]})(\fa\ast(E^{(n)}, \psi_N^{(n)}, \omega_n))=B_0(a,b)\gamma^{a+2b},
$$ 
which is locally analytic in $a$, $b$.

\end{proof}

\bigskip
\noindent
{\it The one variable $p$-adic $L$-function.}
\medskip

Let $k>0$ be a fixed integer and $F$ a classical eigenform of weight $k$, nebentype $\epsilon$ and level $\Gamma_1(N)$ as in section \S  \ref{sec:conclusions}. Let $\widehat{\Sigma}^{(2),p^n}(k, c,\fN,\epsilon)_\cU$ be the open subspace of Hecke characters of the space $\widehat{\Sigma}^{(2)}(k, c,\fN,\epsilon)$ defined in Remark \ref{rmk:p2conductos}  whose image via $w$ lies in $\cU(\Q_p)$. We assume $n\geq n_k(p)$ as in Definition \ref{definition:b(p)}, e.g., $n\geq 2$ for $p\geq 5$.

\begin{definition}\label{def:Lpinert} For every $\chi\in \widehat{\Sigma}^{(2),p^n}(k, c,\fN,\epsilon)_\cU$ with weight  $\nu:=w(\chi)$, we set
$$
L_p(F,\chi):=\frac{1}{\vert (\cO_K/\fN)^\ast\vert} \sum_{\fa\in \cH(c,\fN)} \chi_\nu^{-1}(\fa)\delta_k^\nu\bigl(F^{[p]}\bigr)
\bigl(\fa\ast(E^{(n)}/R, \psi_N^{(n)},  \omega_n)\bigr),
$$
where the evaluation of $(\nabla_k)^\nu(F^{[p]})$ at $\fa\ast(E^{(n)}/\cO_L, \psi_N^{(n)},  \omega_n)$ is done using 
Definition \ref{def:evaluationCM}.
\end{definition}

\begin{remark}
1) Arguing as in the proof of Proposition \ref{prop:localanalinert} it follows that also $L_p(F,\chi)$ is a locally analytic function.
\smallskip

2) Let us consider the notations of definition \ref{def:Lpinert2var} and suppose $u\in \cU(\Q_p)$ is an integer weight large enough such that ${\bf F}_u$ is classical. Let $\chi\in \widehat{\Sigma}^{(2),p^n}(u, c, \fN, \epsilon)_\cU\subset \widehat{\Sigma}^{(2)}(c, \fN, \epsilon)_\cU$, i.e. $\widetilde{w}(\chi)=\bigl(u, w(\chi)\bigr)=(u,\nu)$. Then $L_p({\bf F}, \chi)=L_p({\bf F}_u, \chi)$, i.e.,  we have $$L_p({\bf F}, -)\vert_{\widehat{\Sigma}^{(2)}(u, c, \fN, \epsilon)_\cU}=L_p({\bf F}_u, -).$$ 

3) Inside $\widehat{\Sigma}^{(2),p^n}(c,\fN,\epsilon)$ we have different one dimensional directions  to which one can restrict $L_p({\bf F}, -)$. For example, the direction of the second variable. Our choice above is motivated by the formulae in \S \ref{sec:GrossZagier} which use this specific line. 
\end{remark}

\subsection{Interpolation properties in the case $p$ is inert in $K$.}

In this section assume that $F$ is a classical modular form of weight $k$ and nebentype $\epsilon$ as in Section  \ref{sec:conclusions}.  We show that the values $L_p(F,\chi)$, for $\chi\in \Sigma^{(2)}(k, c,\fN,\epsilon)$, a classical algebraic Hecke character of $K$ of conductor divisible by $p^n$, can be related to the classical values $L_{\rm alg}(F, \chi)$.

\begin{proposition}\label{prop:comparisonclassicalinert}
For every $\chi\in \Sigma^{(2)}(k, c,\fN,\varepsilon)$ such that $p^n$ is  the $p$-part of the conductor with $n\geq n_k(p)$ (Definition \ref{definition:b(p)}) and the infinity type is $(k+m,-m)$ for $m\geq 0$ : 
$$
L_p(F,\chi)=\frac{ L_{\rm alg}(F,\chi)}{\Omega_{p,n}^{k+2m}}.
$$
Here $\Omega_{p,n}\in \cO_L $ is a non-zero element of $ p$-adic valuation $v_p\bigl(\Omega_{p,n}\bigr) = \frac{1}{p^{n-1}(p^2-1)} $.\end{proposition}
 
\begin{proof} Set $\Omega_{p,n}$ to be the $p$-adic period of $E^ {(n)}$ defined by the equality $\omega_{n}=\Omega_{p,n} \omega^{(n)}$ that expresses  $\omega_n$, which is not  a generator   of the  invariant differentials $\omega_{ E^{(n)}}$, in terms of the generator $\omega^{(n)}$.   As obseved in remark \ref{rmk:Omegacan} the section $\omega_n$ generates $\Omega^{\rm can}_{E^{(n)}/\cO_L}\subset \omega_{E^{(n)}/\cO_L}$. Then, as reviewed in \S \ref{sec:cX}, we conclude that $\Omega_{p,n}$ has valuation equal to $\frac{1}{p-1}$ times the $p$-adic valuation of ${\rm Hdg}(E^{(n)}$. The conclusion follows then from Lemma \ref{Lemma:LTI}.

We are left to prove the displayed equality. We denote by $a_p$ the $T_p$-eigenvalue of $F$, i.e., $T_p(F)=a_p F$.
Using Lemma \ref{cor:classicaldepletion}, we can interpret $F^{[p]}$ as a classical modular form of level $\Gamma_1(N)\cap \Gamma_0(p^2)$. In fact, $F^{[p]}=F\vert \bigl(1-a_pV+\epsilon(p)p^{k-1}V^2 \bigr)$. Then $(\nabla_k)^m(F^{[p]})\in {\rm H}^0\bigl(\cX(N, p^2), {\rm Sym}^{k+2m}({\rm H}_{\bE})\bigr)$, i.e. it is a classical (i.e. global) object,  the evaluation  
of $(\nabla_k)^m(F^{[p]})$ at $(\fa\ast \bigl(E^{(n)}/L, \psi_{N}^{(n)}, {\rm H}^{(n)}[p^2], \omega^{(n)})\bigr)$ can be done directly by splitting the 
Hodge filtration of ${\rm Sym}^{k+2m}({\rm H}_{\fa\ast E^{(n)}})$ using the  action of the CM field $K$ and using a non-zero differential of $E^{(n)}$.  Notice that, using the conventions of \S \ref{sec:adelicdescr} we have  $$V^j\bigl(\fa\ast (E^{(n)}/\cO_L, \psi_N^{(n)}, H^{(n)}, \omega_n)  \bigr)=\fa\ast (E^{(n-j)}/\cO_L, \psi_N^{(n-j)},  \omega_{n-j}), \quad j=1,2 .$$Here, to deal with the case $n=2$  we set $E^{(0)}=E$, $\psi_N^{(0)}=\psi_N$ and $\omega_0$ is the pull-back of $\omega_n$ via the projection $\lambda_n\colon E\to E^{(n)}$.  Hence $F^{[p]}\bigl(\fa\ast ((E^{(n)},\psi_N^{(n)},H^{(n)},\omega_n)\bigr)$ is $$F\bigl(\fa\ast ( E^{(n)},\psi_N^{(n)},\omega_n)\bigr)-a_p F\bigl(\fa\ast  (E^{(n-1)}/\cO_L, \psi_N^{(n-1)},  \omega_{n-1})  \bigr)+\epsilon(p)p^{k-1}F\bigl(\fa\ast  (E^{(n-2)}/\cO_L, \psi_N^{(n-2)},  \omega_{n-2})  \bigr) .$$It follows from Definition \ref{def:Lpinert} that  $$
L_p(F,\chi):=\sum_{\fa\in \cH(c,\fN)} \chi_m^{-1}(\fa)\delta_k^m\bigl(F^{[p]}\bigr)
\bigl(\fa\ast(E^{(n)}/\cO_L, \psi_N^{(n)},  \omega_n)\bigr)
$$
On the other hand, as recalled in Section \ref{sec:conclusions}, we have
$$
L_{\rm alg}(F,\chi):=\sum_{\fa\in \Pic(\cO_c)} \chi_m^{-1}(\fa) \delta_k^m(F)\bigl(\fa\ast(E^{(n)}/\cO_L, \psi_N^{(n)},  \omega^{(n)})\bigr),
$$ 
As $\vert (\cO_K/\fN)^\ast\vert \cdot \vert \Pic(\cO_c)\vert=\vert  \cH(c,\fN) \vert$ by Lemma \ref{lemma:cHcfN}, dividing $L_{\rm alg}(F,\chi)$ by $\Omega_{p,n}^{k+2m}$ we have that 
$$\frac{ L_{\rm alg}(F,\chi)}{\Omega_{p,n}^{k+2m}}=\frac{1}{\vert (\cO_K/\fN)^\ast\vert} \sum_{\fa\in\cH(c,\fN)} \chi_m^{-1}(\fa) \delta_k^m(F)\bigl(\fa\ast(E^{(n)}/\cO_L, \psi_N^{(n)},  \omega_n)\bigr).$$In order to conclude it suffices to show that $$\sum_{\fa\in \cH(c,\fN)} \chi_m^{-1}(\fa) \delta_k^m(F)\bigl(\fa\ast(E^{(m)}/\cO_L, \psi_N^{(m)},  \omega_m\bigr)=0$$for every $m\leq n-1$.

Given $\fb\in \cH(p^{n-1}d,\fN)$ and a class $\fa_0\in \cH(p^nd,\fN)$ mapping to $\fb$ all other classes mapping to $\fb$ are of the type $r \fa_0$ with $$r\in U_n:= \bigl(1+ p^{n-1} \widehat{\cO_d}\bigr) / \bigl(1+ p^{n} \widehat{\cO_d}+ p^n \Z_p\bigr)\cong \Z/p\Z.$$As explained in  Remark \ref{rmk:astwelldefined} $\delta_k^m(F)\bigl(\fa\ast(E^{(m)}/\cO_L, \psi_N^{(m)},  \omega_m\bigr)$, for $m<n$, depends only on $\fb$. As  $\chi_m(r \fa_0)=\chi(r) \chi_m( \fa_0)$ it suffices to show that $\sum_{r\in U_n } \chi^{-1}(r) =0$. This follows from the fact that  $p^n$ is the maximal power of $p$ dividing the conductor of $\chi$ so that  $\chi(U_n)=\mu_p$ and the sum of the $p$th roots of unity is $0$.

\end{proof}

\section{The case: $p$ is ramified in $K$. }\label{sec:ramifiedcase}

Consider, as in the inert case, a finite slope family ${\bf F}$, of nebentype $\epsilon$ and weight $k^{\rm univ}\vert_{\cU}$ where $k^{\rm univ}$ is the universal weight of some rigid analytic disk $\cU$ in $W$ such that $k^{\rm univ}\vert_{\cU}$ is analytic in the sense of Definition \ref{def:analweight}. 
As in the discussion before Definition \ref{definition:b(p)}  we  assume that  its $p$-depletion ${\bf F}^{[p]}=(1-UV)({\bf F})$  is  defined over  $\cX_1(N)_{r}\times \cU$ with $r=r_{k^{\rm univ}}(p)$ as in Definition \ref{definition:b(p)}. 

We assume that there is an integral weight $u\in\cU(\Q_p)$ such that the specialization ${\bf F}_u$ is an overconvergent modular form which arises from a classical eigenform $F$ of weight $u$ and nebentype $\epsilon$ that satisfies the assumptions of \S \ref{sec:conclusions}.  
We suppose that  $p$ is ramified in $K$ and denote by $\fP$ the unique prime ideal of $\cO_K$ over $p$ and by $K_{\fP}$ the totally ramified extension of $\Q_p$ of degree $2$ given by the $\fP$-adic completion of $K$.

\subsection{Evaluation at {\rm CM}-points.}\label{sec:evCMramified}

Consider a pair  $(E,\psi_N)$ consisting  of an elliptic curve with {\rm CM} by $\cO_d \subset \cO_K$,  an order of conductor $d$ prime to $p$,
 and a level $\Gamma_1(\fN)$-structure  defined over the ring of integers $\cO_L$ of a finite extension $L$ of $\Q_p$.  According to Lemma \ref{lemma:ramified0} the subgroup $E[\fP]={\rm H}$ is the canonical subgroup of $E/\cO_L$ of order $p$.   We choose $n\geq n_{k^{\rm univ}}(p)$ as in Definition \ref{definition:b(p)}.  In particular, we may take $n\geq  2$ for $p\geq 5$. Fix,  possibly by enlarging the field $L$:

\begin{itemize}

\item[i.] a   subgroup $C^{(n)}\subset E[p^n]$, generically cyclic of order $p^n$  such that $C^{(n)}[p]\cap E[\fP]=\{0\}$;
\item[ii.]  $\lambda_n\colon E \to E^{(n)}:=E/C^{(n)}$  the projection and $\psi_N^{(n)}:=\lambda_n\circ \psi_N$ the induced $\Gamma_1(\fN)$-level structure;
\item[iii.]  $H^{(n)}=E[p^n]/C^{(n)}$ which is the canonical subgroup of level $p^{n}$ of of $E^{(n)}$ by  Lemma \ref{lemma:ramified0}. 

\end{itemize}

Notice that $E^{(n)}$ has CM by $\cO_c$ with $c=p^nd$ and $\mathrm{val}_p\bigl(\mathrm{Hdg}(E^{(n)}))\bigr) =\frac{1}{2p^n}$ by  Lemma \ref{lemma:ramified0} so that $x^{(n)}=\bigl(E^{(n)}, \psi_N^{(n)}\bigr)$ defines an $\cO_L$-valued point  $\varphi_{x^{(n)}}\colon {\rm Spf}(\cO_L)\to \fX_{2p^n}$.   Therefore  if $k$, $\nu\in \cU(\Q_p)$ we can use Theorem \ref{thm:evaluation}  to define:

$$\varphi_{x^{(n)}}^\ast\bigl((\nabla_k)^\nu({\bf F}_k^{[p]})\bigr)\in p^{-a}\varphi_{x^{(n)}}^\ast\bigl(\bW_{k+2\nu}\bigr)\bigl({\rm Spf}(\cO_L)\bigr),$$for some $a\in \N$  introduced in Remark \S \ref{remark:integral}.

\subsection{Splitting the Hodge filtration.}\label{sec:splitingramified}

With the notations of the previous section  set $C:=C^{(n)}[p]$ and $\lambda'\colon E \to E':=E/C$. We notice that $E'$ has {\rm CM} by $\cO_{pd}$. The subgroup $E[p]/C={\rm H}'$ is the canonical subgroup of $E'$. Set $\Psi_N':=\lambda'\circ \psi_N'$. Then $x':=\bigl(E',\Psi_N'\bigr)$ defines an $\cO_L$-valued point $\varphi_{x'}\colon {\rm Spf}(\cO_L)\to \fX_{2p}$  thanks to  Lemma \ref{lemma:ramified0} and we can factor $\lambda_n$ as $$\lambda_n\colon E \stackrel{\lambda'}{\lra} E' \stackrel{\lambda_n'}{\lra} E^{(n)}.$$Morever, it follows from Definition \ref{def:bW}  that $\bW_{k+2\nu}$ is defined on $ \fX_{2p}$  so that we can take its fiber at $x'$ and $x^{(n)}$. As explained in Remark \ref{remark:funct} the isogeny $\lambda_n'$ induces a morphism $(\lambda_n')^\ast\colon {\rm H}_{E^{(n)}}^\sharp \to {\rm H}_{E'}^\sharp$. We define $H_{E',\tau} \oplus H_{E',\overline{\tau}}\subset {\rm H}_{E'}$ as in \S \ref{sec:technicalsec} and $$\widetilde{{\rm H}}_{E'}^\sharp:=H_{E',\tau}^\sharp  \oplus H_{E',\overline{\tau}}^\sharp:=\delta_{E'}{\rm H}_{E',\tau} \oplus \delta_{E'}^p{\rm H}_{E',\overline{\tau}}  \subset H_{E'}^\sharp.$$We then have the following analogue of Lemma \ref{lemma:factorlambdai}:

\begin{lemma}\label{lemma:factorlambdairamified} The image of ${\rm H}_{E^{(n)}}^\sharp$ via $(\lambda_n')^\ast$ is contained in $\widetilde{{\rm H}}_{E'}^\sharp$.
\end{lemma}

\begin{proof} The proof proceeds as the proof of Lemma \ref{lemma:factorlambdai} reducing to the case $n=2$. In this case Lemma \ref{lemma:splittingHdR} states that

\smallskip

(i) ${\rm Hdg}(E') {\rm diff}_{K/\Q}  \omega_{E'}^\vee \subset {\rm H}_{E',\overline{\tau}} \subset \omega_{E'}^\vee$;

\smallskip

(ii) $\bigl(((\lambda_2')^\vee)^\ast\bigr)^\vee\bigl(\omega_{E^{(2)}}^\vee\bigr) \subset p {\rm Hdg}(E^{(2)})^{-1} \omega_{E'}^\vee $.

\smallskip

As in loc.~cit.~it suffices to show that $\delta_{E^{(2)}}^p p {\rm Hdg}(E^{(2)})^{-1} \omega_{E'}^\vee \subset {\rm Hdg}(E') {\rm diff}_{K/\Q} \delta_{E'}^p \omega_{E'}^\vee $. Since ${\rm diff}_{K/\Q}=\fP$, then ${\rm v}_p({\rm diff}_{K/\Q})=1/2$ and we need to prove that  $1/2 + {\rm v}_p(\delta_{E^{(2)}} )\geq {\rm v}_p(\delta_{E'}^{2p-1})$ or equivalently that $\bigl((2p-1)p-1\bigr) {\rm val}_p(\delta_{E^{(2)}}) \leq 1/2$. But $ {\rm val}_p(\delta_{E^{(2)}})=\frac{1}{2p^2(p-1)}$ by Lemma \ref{lemma:ramified0}. Hence we need to show that $2p^2-p-1 \leq p^3-p^2$ or equivalently $3p^2-p-1 \leq p^3$ and this is true for any prime $p$.

\end{proof}

Possibly after enlarging $L$ we assume that ${\rm H}'(L)=\Z/p\Z$. Let $\omega'$ be a generator of $\Omega_{E'/\cO_L}\subset \omega_{E/\cO_L}$ reducing to ${\rm dlog}({\rm H'}(L))\bigl({\rm mod }\ p\underline{\delta}(E')^{-p}\bigr)$ (see Section \ref{sec:cX}).   $$v_{\omega'}\colon \fw_{E', \cO_L}^{k+2\nu}\stackrel{\sim}{\lra} \cO_L.$$We denote by   $\omega_n$ the differential on $\omega_{E^{(n)}/\cO_L}$ whose pull-back via $\lambda_n'$ is $\omega'$ and $s_n\in {\rm dlog}({\rm H}^{n}[p](L))$ the section defined by $s$ using the isomorphism of canonical subgroups ${\rm H}'\cong {\rm H}^{(n)}[p].$provided by $\lambda_n'$. Then, $\omega_n$  is a generator of $\Omega_{E^{(n)}/\cO_L}\subset \omega_{E^{(n)}/\cO_L}$ reducing to $s_n$ modulo $ p\underline{\delta}(E^{(n)})^{-p}\bigr)$ .  

Thanks to Lemma \ref{lemma:factorlambdairamified} we deduce that the morphism $\bW_{k+2\nu,\cO_L}({\rm H}_{E^{(n)}}^\sharp,s_n)\to  \bW_{k+2\nu,\cO_L}( {\rm H}_{E'}^\sharp,s)$, induced by the map $(\lambda_n')^\ast\colon {\rm H}_{E^{(n)}}^\sharp\to  {\rm H}_{E'}^\sharp$ using the functoriality of Remark \ref{remark:funct}, factors through  $\bW_{k+2\nu}\bigl(\widetilde{{\rm H}}_{E'}^\sharp, s \bigr)$. For the latter we have a canonical splitting $\Psi_{E'}$ of the Hodge filtration $\fw_{E'}^{k+2\nu}\subset \bW_{k+2\nu}\bigl(\widetilde{{\rm H}}_{E'}^\sharp, s \bigr)$. In particular, we get 
$$
\Psi_{E'}\bigl((\lambda_n')^\ast \circ \varphi_{x^{(n)}}^\ast \bigl((\nabla_k)^{\nu}(F^{[p]})\bigr)\bigr) \in  p^{-a}\fw_{E'}^{k+2\nu}\subset \fw_{E'}^{k+2\nu}\bigl[p^{-1}\bigr],
$$
where we use that $\varphi_{x^{(n)}}^\ast(\bW_{k+2\nu})=\bW_{k+2\nu,\cO_L}({\rm H}_{E^{(n)}}^\sharp,s_n)$ (see Remark \ref{rmk:functVBMS}). Notice that $x^{(n)}$  defines a point of $\cX_1(N)_{p(p+1)}$ so that $(\nabla_k)^{\nu}(F^{[p]})$ ca be evaluated at $x^{(n)}$ by the discussion before Definition \ref{definition:b(p)}. Then,

\begin{definition}
\label{def:evaluationCMramified}
We define 
$$
\delta_k^\nu({\bf F}^{[p]})(E^{(n)}/\cO_L, \psi_N^{(n)},  \omega_n):=v_{\omega'}\circ \Psi_{E'}\Bigl((\lambda_n')^\ast \circ \varphi_{x^{(n)}}^\ast\bigl((\nabla_k)^\nu(F^{[p]})   \bigr)  \Bigr)\in p^{-a}\cO_L\subset \cO_L[p^{-1}]=L
$$

\end{definition}

\subsection{Definition of the $p$-adic $L$-function in the ramified case.}\label{sec:padicLramified}

Consider $\chi\in \widehat{\Sigma}^{(2), p^n}(c,\fN,\epsilon)$, a character of conductor divisible by $p^n$  and with weights $w'(\chi)=u\in \cU(\Q_p)$ and  $w(\chi)=\nu\in \cU(\Q_p)$ satisfying the Assumption  in \S \ref{section:powersGM}. Recall that we assume that $n\geq n_{k^{\rm univ}}(p)$ as in Definition  \ref{definition:b(p)}. We define the $p$-adic 
$L$-function $L_p({\bf F}, \chi)$, where ${\bf F}$ is an eigenfamily of weight $k^{\rm univ}$ over the affinoid disk $\cU\subset W$, as at the beginning of Section \ref{sec:ramifiedcase}, by the same formula as in the inert case:

\begin{definition}
$$
L_p({\bf F}, \chi):= \frac{1}{\vert (\cO_K/\fN)^\ast \vert}\sum_{\fa\in  \cH(c,\fN) } \chi_{\nu}^{-1}(\fa)\delta_u^\nu\bigl(({\bf F}_u)^{[p]}\bigr)
\bigl(\fa\ast (E^{(n)}/\cO_L,\psi_N^{(n)},  \omega_n)   \bigr),
$$
where  the evaluation of $({\bf F}_u)^{[p]}$ at $\fa\ast (E^{(n)}/\cO_L,\psi_N^{(n)},  \omega_n)$ is defined via Definition \ref{def:evaluationCMramified} using \S \ref{sec:adelicdescr}.
\end{definition}

We have the analogue of Proposition \ref{prop:localanalinert}. 

\begin{proposition}\label{prop:localanalramified}
The function $L_p({\bf F}, - )\colon \widehat{\Sigma}^{(2), p^n}(c, \fN, \epsilon)_\cU \lra \C_p$ is locally analytic.
\end{proposition}

We refer to loc.~cit.~for the proof. If now $k$ is a fixed positive integer and $F$ is a classical eigenform of weight $k$, level $N$ and nebetype $\epsilon$ and $\chi\in \widehat{\Sigma}^{(2),p^n}(k, c, \fN, \epsilon)$,  with $w(\chi)=\nu\in \cU(\Q_p)$, we define

\begin{definition}\label{def:Lpramified}
$$
L_p(F, \chi):= \frac{1}{\vert (\cO_K/\fN)^\ast \vert}\sum_{\fa\in  \cH(c,\fN) } \chi_{\nu}^{-1}(\fa)\delta_k^\nu\bigl( F^{[p]}\bigr)
\bigl(\fa\ast (E^{(n)}/\cO_L,\psi_N^{(n)},  \omega_n)   \bigr),
$$
where  the evaluation of $F^{[p]}$ at $\fa\ast (E^{(n)}/\cO_L,\psi_N^{(n)},  \omega_n)$ is done via Definition \ref{def:evaluationCMramified}, using \S \ref{sec:adelicdescr}.
\end{definition}

As in the inert case we have, also $L_p(F, \chi)$ is a locally analytic function. Furthermore,  we have $L_p({\bf F}, - )\vert_{\widehat{\Sigma}^{(2), p^n}(k, c, \fN, \epsilon)}=L_p({\bf F}_k, -)$, where the equality is as functions on $\widehat{\Sigma}^{(2)}(k, c, \fN, \epsilon)$.

\subsection{Interpolation properties in the case $p$ is ramified in $K$.}

Assume that $F$ is a classical eigenform of weight $k$, level $\Gamma_1(N)$ and nebentype $\epsilon$ and that $\chi\in \Sigma_{cc}^{(2)}(c,\fN,\epsilon)$ is an algebraic Hecke character with weight $w(\chi)=j$ (i.e. the infinity type of $\chi$ is $(k+j,-j)$ with $j\in \N$) as in Section \ref{sec:conclusions}. We assume that $\chi$ is of conductor with $p$-part $p^n$ with $n\geq n_{k}(p)$, the integer of Definition \ref{definition:b(p)}.  Then:

\begin{proposition}\label{prop:comparisonclassicalramified}  We have
$$L_p\bigl(F,\chi\bigr)=\frac{ L_{\rm alg}(F,\chi)}{\Omega_{p,n}^{k+2j}}
$$wiith $\Omega_{p,n}\in \cO_L$ a $p$-adic period of valuation $ v_p\bigl(\Omega_{p,n} \bigr) =\frac{1}{2p^{n}(p-1)} $.

\end{proposition}

\begin{proof} The proof proceeds as the proof of Proposition \ref{prop:comparisonclassicalinert}. We only elaborate on the $p$-adic valuation of $\Omega_{p,n}$. It is the non-zero element of $\cO_L$ such that, given a generator $\omega^{(n)}$ of the differential $\omega_{E^{(n)}/\cO_L}$, we have $\omega_n=\Omega_{p,n} \omega^{(n)}$. By the discussion of \S \ref{sec:cX} we have $(p-1)v_p\bigl(\Omega_{p,n} \bigr)=v_p({\rm Hdg}(E^{(n)})$.  It follows from  Lemma \ref{lemma:ramified0} that  ${\rm Hdg}(E^{(n)})$ has $p$-adic valuation $1/ (2p^{n})$. The conclusion follows.

\end{proof}

\section{Special values $L_p(F,\chi)$, for $\chi\in \Sigma^{(1)}_{cc}(\fN)$.}

In this section we give formulae for the special values of our $p$-adic $L$-functions at algebraic Hecke characters in $\Sigma^{(1)}_{cc}(\fN)$. In the case $p$ is split in $K$, these formulae have been called ``$p$-adic Gross-Zagier" formulae in \cite{bertolini_darmon_prasana}, if $F$ is a cuspform and they have been called ``Kronecker limit formulae" in \cite{katzEisenstein}, if $F$ is an Eisenstein series. We prove such formulae in the cases in which $p$ is not split in $K$.

\subsection{The $p$-adic Gross-Zagier formulae.}\label{sec:GrossZagier}

We work under the assumptions of \S \ref{sec:conclusions} where $F=f$ is a normalized new cuspform for $\Gamma_1(N)$ of integer weight $k\ge 2$ and character $\epsilon$.  For $n\in \N$ consider the subset $\Sigma_{cc}^{p^n}(k,c, \fN, \epsilon)$  of the space of algebraic  Hecke characters $\Sigma_{cc}(k,c, \fN, \epsilon)$ defined in \S \ref{sec:conclusions} where $c=dp^n$ 
and $(d,p)=1$ and $p^n$ is  the $p$-part of the conductor.  Then $\Sigma_{cc}^{p^n}(k,c, \fN, \epsilon)=\Sigma_{cc}^{(1),p^n}(k,c, \fN, \epsilon)\cup \Sigma_{cc}^{(2),p^n}(k,c, \fN, \epsilon)$ where 
\smallskip

$\bullet$ $\Sigma_{cc}^{(1),p^n}(k,c, \fN, \epsilon)$ is the subset of characters of $\Sigma_{cc}(c, \fN, \epsilon)$ having infinity type 
$(k-1-j, 1+j)$, with $0\le j\le k-2$

\smallskip
$\bullet$ $\Sigma_{cc}^{(2),p^n}(k,c, \fN, \epsilon)$ is the subset of characters of $\Sigma_{cc}(c, \fN, \epsilon)$ having infinity type 
$(k+j, -j)$ for $j\ge 0$.

\smallskip

As explained in \cite[\S 5.3]{bertolini_darmon_prasana} the characters in $\Sigma_{cc}^{(1),p^n}(k,c, \fN, \epsilon)$ can be realized in the completion $\widehat{\Sigma}^{(2),p^n}(k,c, \fN, \epsilon)$ of the space $\Sigma_{cc}^{(2),p^n}(k,c, \fN, \epsilon)$ defined in Remark \ref{rmk:p2conductos},  i.e., as  $p$-adic limits of characters  in the space $\Sigma_{cc}^{(2),p^n}(k,c, \fN, \epsilon)$.   The goal of the present section is to prove the following:

\begin{theorem}\label{thm:AJ} Assume  that $n\geq n_k(p)$ with $n_k(p)$ as in Definition \ref{definition:b(p)} (e.g., $n_k(p)=2$ for $p\geq 5$). Let $\chi\in \Sigma_{cc}^{(1),p^n}(k,c, \fN, \epsilon)$ be a character of infinity type $(k-1-j,1+j)$ with $0\leq j \leq r:=k-2$.   Then, the value $L_p(f,\chi)$ of the $1$-variable $p$-adic $L$-function $L_p(f,\_)$ of Definitions \ref{def:Lpinert} and \ref{def:Lpramified} at $\chi$ viewed as an element of $\widehat{\Sigma}^{(2),p^n}(k,c, \fN, \epsilon)$ is

$$L_p(f,\chi)= 
\frac{c^{-j}\Omega_{p,n}^{r-2j}}{j!}\sum_{\fa\in \Pic(\cO_c)}  \chi^{-1}(\fa) {\bf N}(\fa) \mathrm{AJ}_L(\Delta_{\varphi_{\fa} \varphi_0})\bigl(\omega_f\wedge \omega_A^j \wedge \eta_A^{r-j}\bigr).$$
Here and elsewhere $\omega_f$ is $f$ seen as a section of ${\rm H}^1_{\rm dR}\bigl(X_1(N)^{\rm an}, {\rm Sym}^k({\rm H}_{\bE})\bigr)$.  

\end{theorem}
 
We explain the notation.   The pair $(A,t_A)$ denotes an elliptic curve with CM by $\cO_K$ and $\Gamma_1(\fN)$-level structure, $\omega_A$ is a generator of the invariant differentials of $A$ and $\eta_A$ is an element of  ${\rm H}_{\rm dR}^1(A)$ such  that $\langle \omega_A,\eta_A\rangle =1$ via the Poincar\'e pairing. Everything is defined over the finite extension $L$ of $\Q_p$. Morever, $\varphi_0\colon A\to A_0$ is a cyclic isogeny of degree $c=dp^n$ so that $A_0$ has CM by $\cO_c$. 
In the sections  \S \ref{sec:padicLinert} and \S \ref{sec:padicLramified}, we used the notation $E^{(n)}$ instead of $A_0$, but we prefer in this section to follow the notations of \cite{bertolini_darmon_prasana} in everything connected to 
the generalized Heegner cycles.  

For every $\fa\in \cH(c,\fN) $ we have an isogeny $\varphi_\fa\colon A_0 \to \fa\ast A_0$ and  $ \mathrm{AJ}_F(\Delta_{\varphi_{\fa} \varphi_0})$ is the $p$-adic Abel-Jacobi map of a generalised Heegner cycle constructed in  \cite[\S 2]{bertolini_darmon_prasana}
 over the modular point $\varphi_\fa \circ \varphi_0\colon A \to \fa\ast A_0$.

\

The rest of the section is devoted to the proof of  Theorem \ref{thm:AJ}.  First of all, following \cite{bertolini_darmon_prasana}, we explain how one can compute the Abel-Jacobi image of the generalized Heegner cycle in terms of a Coleman primitive of our modular form $f$.  Let us denote by $\cX$ the open analytic subspace of the modular curve $X_1(N)$  defined in \S\ref{sec:cX}. It  is an open rigid subspace of $X_1(N)^{\rm an}$ and consists of the points $(E,t)$ such that $E$ has a canonical subgroup of order $p$.   Let us denote by $\Phi$ the Frobenius on ${\rm H}^1_{\rm dR}\bigl(\cX, {\rm Sym}^r{\rm H}_E\bigr)$ and by $P(X)\in \Q_p[X]$ a polynomial such that: $P(\Phi)([\omega_f])=0$ and $P(\Phi)$ defines an automorphism of
${\rm H}^0_{\rm dR}(\cX, {\rm Sym}^r({\rm H}_E)^{\rm loc})$; 
we recall that $[\omega_f]$ denotes the cohomology class of $f$ in  ${\rm H}^1_{\rm dR}\bigl(\cX, {\rm Sym}^r{\rm H}_E\bigr)$, and  ${\rm Sym}^r{\rm H}_E^{\rm loc}$ denotes the sheaf of locally analytic sections (i.e. analytic on every residue class of $X_1(N)^{\rm an}$) of  the sheaf ${\rm Sym}^r{\rm H}_E$. Recall from \cite[\S 11]{ColemanPrimitive} that we have a Coleman primitive $G$ of $f$:
 this is a section $G$ over $X_1(N)^{\rm an}$ of ${\rm Sym}^{r}{\rm H}_E^{\rm loc} $ such that $\nabla(G)=f$ and such that $P(\Phi)(G\vert_{\cX})$ is an analytic section of ${\rm Sym}^r{\rm H}_E$ on some overconvergent neighbourhood $\cX'\subset \cX$ of the ordinary locus in $X_1(N)^{\rm an}$. It then  follows that $G$ is unique up to a horizontal section of ${\rm Sym}^r{\rm H}_E$ on $\cX'$, i.e., it is unique if $r>0$ and unique up to a constant if $r=0$. 

 This is related to $\mathrm{AJ}_F(\Delta_{\varphi_{\fa} \varphi_0})(\omega_f\wedge \omega_A^j\eta_A^{r-j})$ as follows. The isogeny $\varphi_0\colon A\to A_0$  of degree $c$ and the  $\Gamma_1(\fN)$-level structure $ t_A$ on $A$ induces a $\Gamma_1(\fN)$-level structure $t_0$ on $A_0$. The element $\omega_A$ is a generator of the invariant differentials on $A$ such that $\langle \omega_A,\eta_A\rangle =1$ via the Poincar\'e pairing. Thus we get an invariant differential $\omega_0$ on $A_0$ such that $\varphi_0^\ast(\omega_0)=\omega_A$. The Coleman primitive $G$, being defined over $X_1(N)^{\rm an}$,   can be evaluated at the points $\fa\ast (A_0, t_0)$ and, using the CM action, we can decompose ${\rm Sym}^{r}({\rm H}_{\rm dR}^1(\fa \ast A_0))=\oplus_{i=0}^{r} {\rm Sym}^{r}({\rm H}_{\rm dR}^1(\fa \ast A_0))_{\tau^{r-i} \overline{\tau}^i}  $ into eigenspaces for the action of $K$. Then we decompose $G\bigl(\fa\ast (A_0, t_0,\omega_0)\bigr)=\sum_{i=1}^{r} (-1)^iG_i\bigl(\fa\ast (A_0, t_0, \omega_0)\bigr)\omega_\fa^{r-i}\eta_\fa^i$, where
$\omega_\fa, \eta_\fa$ is a basis of ${\rm H}_{\rm dR}^1(\fa\ast A_0)$ adapted to the $K$-decomposition such that $\omega_\fa$ corresponds to $\omega_0$.

We recall, following \cite{bertolini_darmon_prasana} section \S 3.3 that the $p$-adic Abel-Jacobi image of the generalized Heegner cycle ${\rm AJ}_{p,L}(\Delta_\varphi)$ in ${\rm H}^1_f(L, V)$ is seen as an extension class in the category of $p$-adic Galois representations of $G_L$:
$$
0\lra V\lra W\lra \Q_p\lra 0,
$$ 
 where $L$ is a finite extension of $\Q_p$ over which the generalized Heegner cycle $\Delta_\varphi$ is defined and $V,W$ are crystalline $G_L$-representations such that $V$ is the restriction of a global Galois representation of Deligne weight $-1$, to $G_L$.
Via the $p$-adic comparison isomorphism this extension class becomes an extension class of admissible filtered, Frobenius modules in ${\rm Ext}^1_{L_0, \rm crys}(L_0, H)$, where $H={\rm D}_{\rm cris}(V)$ and $L_0=D_{\rm cris}(\Q_p)$. Here $L_0$ is the maximal unramified extension of $\Q_p$ in $L$, and it is seen as a filtered, Frobenius module with 
Frobenius, the geometric Frobenius $\sigma\in {\rm Gal}(L_0/\Q_p)$ and filtration ${\rm Fil}^i(L_0)=L$ for $i\le 0$ and ${\rm Fil}^i(L_0)=0$ for $i>0$.

The following lemma is Proposition 3.5 in \cite{bertolini_darmon_prasana} in the case $L=L_0$ (in their case 
$(c, p)=1$).

\begin{lemma}
\label{lemma:ext}
There is a canonical and functorial isomorphism:
$$
\alpha:{\rm Ext}^1_{L_0,\rm crys}(L_0, H)\cong H_L/{\rm Fil}^0(H), \mbox{ where } H_L:=H\otimes_{L_0}L.
$$

\end{lemma}

\begin{proof} For an object $D$ in the category of filtered, Frobenius modules over $L_0$ we denote by
${\rm Fil}(D)^\bullet$ its filtration by $L$-vector subspaces of $D_L:=D\otimes_{L_0}L$, and by $\phi_D:D\lra D$ its $\sigma$-linear Frobenius, which is bijective.

The map $\alpha$ is defined as follows: let 
$x\in {\rm Ext}^1_{L_0,\rm ffm}(L_0, H)$ denote the class of the extension
$$
(\ast)\ 0\lra H\stackrel{\beta}{\lra} E\stackrel{\gamma}{\lra} L_0\lra 0.
$$ 

We recall from \cite{zink} that if $D$ is a finite dimensional $L_0$-vector space with a $\sigma$-linear automorphism $\phi_D:D\lra D$, then $D$ has a canonical and functorial slope decomposition
$$
D=\oplus_{\lambda\in \Q} D_\lambda,
$$
where if $\displaystyle \lambda=\frac{r}{s}\in \Q$, with $r,s\in \Z$, $s>0$, then $D_\lambda$ is the largest $L_0$-vector subspace of $D$ containing an $\cO_{L_0}$-lattice $M$, such that $\phi_D^s(M)=p^rM$.

It follows immediately that we have a canonical exact sequence of the slope $\lambda=0$-subspaces of the sequence $(\ast)$:
$$
0\lra H_0\lra E_0\stackrel{\gamma_0}{\lra} (L_0)_0=L_0\lra 0,
$$
and the assumption on the weight of $V$ implies that $H_0=0$. Therefore we have $\gamma_0:E_0\cong (L_0)_0=L_0$ is an isomorphism and this implies that, as Frobenius modules, we have a canonical isomorphism $(\beta +\gamma_0^{-1}):H\oplus L_0\cong E$.
Let us denote, as in the proof of Proposition 3.5 of \cite{bertolini_darmon_prasana} by $\eta^{\rm frob}\in E_0$ the unique element such that $\gamma(\eta^{\rm frob})=1\in L_0$ and $\phi_E(\eta^{\rm frob})=\eta^{\rm frob}$.

We also have an exact sequence of $L$-vector spaces 
$$
0\lra {\rm Fil}^0(H)\lra {\rm Fil}^0(E)\stackrel{\gamma}{\lra} {\rm Fil}^0(L_0)=L\lra 0
$$
so we denote by $\eta^{\rm hol}\in {\rm Fil}^0(E)$ an element such that $\gamma(\eta^{\rm hol})=1$.
Let $\eta\in H_L$ be the unique element such that $(\beta\otimes 1_L)(\eta)=\eta^{\rm hol}-\eta^{\rm frob}\otimes 1$.

 We define $\alpha(x)=\eta\bigl( \mbox{ mod}{\rm Fil}^0(H)\bigr)\in H_L/{\rm Fil}^0(H)$. We observe that $\alpha(x)$ is well-defined and determines a morphism $\alpha:{\rm Ext}^1_{L_0, \rm ffm}(L_0, H)\lra H_L/{\rm Fil}^0(H)$.
Let us show that $\alpha$ is an isomorphism of groups. 

If $\alpha(x)=0$, we have that ${\rm Fil}^0(E)$ is generated by $\eta^{\rm frob}\otimes 1$ and this shows that the exact sequence $(\ast)$ is split as a sequence of filtered, Frobenius modules, i.e. $x=0$. This shows injectivity.

Let now $y\in H_L/{\rm Fil}^0(H)$ be a class. We define $E_y:=H\oplus uL_0$ as vector spaces with Frobenius
$\phi_{E_y}$ defined by: $\phi_E(h+au):=\phi_H(h)+\sigma(a)u$, for $h\in H$ and $a\in L_0$.
We define the filtration on $(E_y)_L=H_L\otimes (u\otimes 1)L$ by: ${\rm Fil}^i(E_y)={\rm Fil}^i(H)$ for $i>0$ and
${\rm Fil}^i(E_y)={\rm Fil}^i(H)+\bigl((u\otimes 1)+\tilde{y}\bigr)L$, where $\tilde{y}\in H_L$ is an element lifting $y$, for $i\le 0$.
Observe that ${\rm Fil}^\bullet(E_y)$ is well-defined, the filtered, Frobenius module $(E_y, {\rm Fil}^\bullet(E_y), \phi_{E_y})$ is admissible and we have a natural exact sequence of filtered, Frobenius modules
$$
0\lra H\lra E_y\lra L_0\lra 0.
$$
Moreover the image under $\alpha$ of the class of this extension is $y$.
\end{proof}

Next we'd like to calculate the Abel-Jacobi image of the cycles $\Delta_{\varphi_0}$, so we apply lemma \ref{lemma:ext} for  
$H_L=\epsilon{\rm H}^1_{\rm dR}\bigl(X_1(N)^{\rm an}, \mathcal{L}_{r,r}, \nabla_r\bigr)(r+1)$, where
$\epsilon$ is an idempotent defining a Chow motive (see \cite{bertolini_darmon_prasana}) and  $\mathcal{L}_{r,r}$ is the $\cO_{X_1(N)^{\rm an}}$-module $\mathcal{L}_{r,r}:={\rm Sym}^r({\rm H}_{\bE})\otimes {\rm Sym}^r\bigl({\rm H}^1_{\rm dR}(A)\bigr)$. The Poincar\'e pairing
$$
\langle \ ,\ \rangle: {\rm H}^1_{\rm dR}\bigl(X_1(N)^{\rm an}, \mathcal{L}_{r,r}, \nabla_r\bigr)(r+1)\times {\rm H}_{\rm dR}^1
\bigl(X_1(N)^{\rm an}, \mathcal{L}_{r,r}, \nabla_r\bigr)(r)\lra L,
$$
gives an isomorphism 
$$
{\rm H}^1_{\rm dR}\bigl(X_1(N)^{\rm an}, \mathcal{L}_{r,r}, \nabla_r\bigr)(r+1)/{\rm Fil}^0\cong 
{\rm Fil}^1\Bigl({\rm H}^1_{\rm dR}\bigl(X_1(N)^{\rm an}, \mathcal{L}_{r,r}, \nabla_r\bigr)(r)   \Bigr)^\vee,
$$
therefore we have (see \cite{bertolini_darmon_prasana}, \S 3.4)
${\rm AJ}_L(\Delta_{\varphi_0})\in {\rm H}^0(X_1(N)^{\rm an}, \omega_{\bE}^{r+2})^\vee\otimes {\rm Sym}^r\bigl({\rm H}^1_{\rm dR}(A)^\vee\bigr)$. The evaluation of ${\rm AJ}_L(\Delta_{\varphi_0})$ at the family of elements  $\bigl(\omega_f\otimes (\omega_A)^j(\eta_A)^{r-j}\bigr)_{j=0, r}$ is given by:

\begin{lemma}{\rm \cite[Lemma 3.22]{bertolini_darmon_prasana}.}
We have $\mathrm{AJ}_L(\Delta_{\varphi_0})(\omega_f\otimes (\omega_A)^j(\eta_A)^{r-j})=  c^jG_j(A_0, t_0,\omega_0)$.
\end{lemma}

 Evaluating the $p$-dic $L$-function $L_p(f,\_)$ of Definitions \ref{def:Lpinert} (if $p$ is inert) and \ref{def:Lpramified} (if $p$ is ramified) at $\chi$ and using the conventions after Theorem \ref{thm:AJ} for the notation $(A_0,t_0,\omega_0)$, and of \S \ref{sec:adelicdescr} for $\fa\ast (A_0, t_0, \omega_0)$, we obtain: 
$$L_p(f,\chi)  =\frac{1}{\vert (\cO_K/\fN)^\ast \vert} \sum_{\fa\in \cH(c,\fN)} \chi_{-1-j}^{-1}(\fa)\delta_k^{-1-j}\bigl(f^{[p]}\bigr)
\bigl(\fa\ast (A_0, t_0, \omega_0)\bigr).$$
Thus, in order to prove Theorem \ref{thm:AJ} we are left to show the formula:

\begin{equation}\label{eq: LpG0} L_p(f,\chi)=
 \frac{\Omega_{p,n}^{r-2j}}{j!} \sum_{\fa\in \Pic(\cO_c)} \chi^{-1}_{-1-j}(\fa) G_j\bigl(\fa\ast (A_0, t_0, \omega_0)\bigr).\end{equation}

We'll first prove some properties of the Coleman primitives. Let $f$ be our classical cuspidal eigenform of weight $k=r+2\ge 2$ and level $\Gamma_1(N)$. Let $P(X)$ be a polynomial with the property (1) $P(\Phi)([\omega_f])=0$ and (2) $P(\Phi)$ defines an automorphism of $\cH^0:={\rm H}^0_{\rm dR}\bigl(\cX', {\rm Sym}^r({\rm H}_{\cE})^{\rm loc}\bigr)$ over a strict neighbourhood $\cX'$ of the ordinary locus where we have a lift $\Phi$ of Frobenius. Then

\begin{lemma}
\label{lemma:Cprimitive}
Let $G$ and $G'$ be two Coleman primitives of $f$ such that $P(\Phi)(G)$ and $P(\Phi)(G')$ are analytic over a strict neighbourhood $\cX'$ of the ordinary locus. Then we have:

a) If $G(q)=G'(q)$ then $G=G'$ (where $G(q)$ and $G'(q)$ denote the $q$-expansions of $G$ and $G'$
respectively).

b) If $\phi\neq T\subset \cX'$ is an admissible open and $G\vert_T=G'\vert_T$ then $G=G'$. 
\end{lemma}

\begin{proof}  We set $F:=G-G'$, then $\nabla(F)=\nabla(G)-\nabla(G')=f\vert_{\cX'}-f\vert_{\cX'}=0$, therefore 
$F\in \cH^0$. We also have: $P(\Phi)(F)=P(\Phi)(G)-P(\Phi)(G')\in {\rm H}^0\bigl(\cX', {\rm Sym}^r({\rm H}_{\bE})   \bigr)$, i.e., $P(\Phi)(F)$ is an analytic section of $ {\rm Sym}^r({\rm H}_{\bE})$.

a) Let now $F(q)=G(q)-G'(q)$. Then $P(\Phi)(F)(q)=P(\Phi)(G)(g)-P(\Phi)(G')(q)=0$ and as $P(\Phi)(F)$ is an analytic section of ${\rm Sym}^r({\rm H}_{\bE})$, we have $P(\Phi)(F)=0$. But by property 2) above 
$P(\Phi)$ is an automorphism of $\cH^0$, therefore $F=0$, i.e., $G=G'$.

b) Similarly we have $P(\Phi)(F)|_{P(\Phi)^{-1}(T)}=0$ and as $\phi\neq P(\Phi)^{-1}(T)$ is an open of $\cX'$ we have
$P(\Phi)(F)=0$ by analytic continuation which implies, as above, $F=0$, i.e., $G=G'$. 
\end{proof}

We define  $\displaystyle G^{[p]}:=G\vert(1-VU)=G\vert\bigl(1-T_p \circ  V+\frac{1}{p}[p]  V^2\bigr)$, viewed as a locally analytic section of ${\rm Sym}^r({\rm H}_{\bE})$ on $\cX_{p+2}$ or a locally analytic section of ${\rm Sym}^r({\rm H}_{\bE})$ on $\cX(N,p^2)$.
Let  $(G^{[p]})_r$ be the locally analytic global section of $\omega_E^{-r} $ over  $\cX(N,p^2)$  given by 

$$(G^{[p]})_r:=(-1)^r\times \mbox{ the image of } G^{[p]}  \mbox{ via the quotient map } {\rm Sym}^{r}({\rm H}_E)\to \omega_E^{-r}.$$

\begin{remark} 1) The sign in the above definition  is there so that the restriction of $(G^{[p]})_r$ to the ordinary locus agrees with the section
denoted $G^\flat_r$ in \cite{bertolini_darmon_prasana}, proposition 3.24.

2) In \cite{bertolini_darmon_prasana} the objects $G\vert (UV-VU)$ and $f|(UV-VU)$ are denoted by $G^\flat$ and respectively $f^\flat$ but 
we prefer to follow the later notations of \cite{darmon_rotger} of $G^{[p]}$ and respectively $f^{[p]}$ for the same objects. 
\end{remark}

Recall also from \S\ref{sec:nablaucXpnp2} that we have a sheaf $\bW_{-r}$  over the  rigid open subspace $\cX_{p+2}$ of the rigid analytic modular curve $X_1(N)^{\rm an}$ and also over $\cX(N,p^2)_{p+2}$ of the rigid analytic modular curve $\cX(N,p^2)^{\rm an}$, with a connection $\nabla_{-r}^j\colon \bW_{-r} \to \bW_{-r+2j}$ . Moreover we have an inclusion  $\omega_E^{-r}=\Fil_{0} \bW_{-r} \subset \bW_{-r}$. Furthermore, as $-1-j\in \Z$ can be seen as a weight in $W(\Q_p)$, then $\nabla_{r+2}^{-1-j}(f^{[p]})$ was defined as a section of $\bW_{r-2j}$ over $\cX_{p+2}$ and also over $\cX(N,p^2)_{p+2}$. Moreover 
$\nabla_{-r}^{r-j}\bigl((G^{[p]})_r\bigr)$ is a section over $\cX(N, p^2)_{p+2}$ of $\bW_{r-2j}$, therefore the statement of the following Proposition makes sense. Let $\cX_{>p+2}$ denote the wide open in the 
rigid analytic modular curve $X_1(N)^{\rm an}$, $\displaystyle \cX_{>p+2}:=\lim_{\rightarrow, r}\cX_r$, where the inductive limit is over all $r\in \N$ with $r>p+2$ and by $\cX(N, p^2)_{>p+2}$ the inverse image of $\cX_{>p+2}$ in $X(N, p^2)^{\rm an}$.  

\begin{proposition}\label{prop:nabla-1-j} (1) The section $G^{[p]}$ is analytic on $\cX_{>p+2}$, so in particular on $\cX(N,p^2)_{>p+2}$ and $\nabla_r(G^{[p]})=f^{[p]} $;

(2)  $(G^{[p]})_r $ is an analytic section of $\omega_E^{-r} $ on $\cX(N,p^2)_{>p+2}$ and we have
$\nabla_{-r}^{r-j} (G^{[p]})_r=    r ! \nabla_{r+2}^{-1-j}\bigl(f^{[p]}\bigr)$ as sections of $\bW_{r-2j}$ over $\cX(N,p^2)_{>p+2}$.

\end{proposition} 

\begin{proof} Using the definition of $\nabla_{r+2}^{-1-j}$  it suffices to prove both statements  on $\cX_{>p+2}$.

(1) Write $f^{[p]}=f\vert(1-VU) 
=f\vert(1-a_pV+\epsilon(p)p^{k-1}V^2)$, where the last polynomial of degree $2$ in $V$ can be written as 
a polynomial $P$ for a Frobenius lift $\Phi$ on $\cH^1:={\rm H}^1_{\rm dR}\bigl(\cX_{>p+2}, {\rm Sym}^r({\rm H}_{\bE})  \bigr)$.
Moreover we have 
$P(\Phi)([\omega_f])=[f^{[p]}]=0$ as $\cH^1$ is finite dimensional and in a finite dimensional vector space
if $UV=1$ then $VU=1$ as well. Also $P(\Phi)=(1-VU)$ is an isomorphism on $\cH^0$ by the calculations of lemma 11.1 of \cite{ColemanPrimitive}, i.e., the operator $(1-VU)=P(\Phi)$ is one of the polynomials in Frobenius which can be used to define the Coleman primitives. Therefore,  $G^{[p]}$ is an analytic section 
of ${\rm Sym}^r({\rm H}_{\bE})$ on a strict neighbourhood $\cX'\subset \cX_{>p+2}$ of the ordinary locus of $X_1(N)^{\rm an}$.
Secondly, both $\nabla_r(G^{[p]})$ and $f^{[p]}$ are analytic overconvergent sections of ${\rm Sym}^r{\rm H}_\bE$ which by Proposition \ref{prop:nablapdep} agree on the admissible open $Y^{\rm ord}$, the ordinary locus of $\cX_1(N)$ minus the residue classes of the cusps. Therefore they are equal on $\cX'$.  
 And thirdly, as the cohomology class $[f^{[p]}]=0$ in  $\cH^1:={\rm H}^1_{\rm dR}\bigl(\cX_{>p+2}, {\rm Sym}^r({\rm H}_{\bE})  \bigr)$ and as $\cX_{>p+2}$ is a wide open analytic space, i.e. 
a Stein space so that $\displaystyle \cH^1=\frac{{\rm H}^0\bigl(\cX_{>p+2}, {\rm Sym}^{r+2}{\rm H}_\bE  \bigr)}{\nabla\Bigl({\rm H}^0\bigl(\cX_{>p+2}, {\rm Sym}^r{\rm H}_\bE   \bigr)   \bigr)\Bigr)}$, there is a section $G'\in {\rm H}^0\bigl(\cX_{>p+2}, {\rm Sym}^r({\rm H}_{\bE})  \bigr)$, 
unique up to horizontal section of the sheaf ${\rm Sym}^r{\rm H}_{\bE} $ such that $f^{[p]}=\nabla_r(G')$. 
Choose $G'$ such that $G'\vert_{\cX'}=G^{[p]}$. It follows that $G^{[p]}$ can be analytically extended to $\cX_{>p+2}$
(by $G'$).

(2) It follows from (1) that $(G^{[p]})_r$ is an overconvergent modular form  defined on $\cX(N, p^2)_{>p+2}$ as it is an image of $G^{[p]}$. We first check the  second statement for $j=0$, on $q$-expansions. We write $G^{[p]}(q)=\sum_{j=0}^r (-1)^jg_j^{[p]}(q) V_{r,j}$ according to the canonical basis of ${\rm Sym}^r({\rm H}_E)$ at the cusp.  In this case \cite[Thm. 4.3]{andreatta_iovita} gives
the expression of the $q$-expansion of $$\nabla_{r+2}^{-1-r}\bigl(f^{[p]}\bigr):=\partial^{-1-r}\bigl(f^{[p]}(q)\bigr)V_{-r,0}+\sum_{j\ge 1}  \left(\begin{matrix} -1-r  \cr j \end{matrix}\right) \prod_{i=0}^{j-1} (-i)   \partial^{-r-1-j}\bigl(f^{[p]}(q)\bigr) V_{-r,j}=  
$$
$$
=\partial^{-r-1}\bigl(f^{[p]}(q)\bigr) V_{-r,0}.
$$ 
On the other hand $ r! \partial^{-r-1}\bigl(f^{[p]}(q)\bigr) = g^{[p]}_r(q)$  (see \cite[Prop. 3.24]{bertolini_darmon_prasana}  or the proof of Proposition \ref{prop:deta-1-j}) and the claim follows as $G^{[p]}_r(q)= g^{[p]}_r(q) V_{-r,0}$ by definition.  This implies that $r!\nabla^{-1}(f^{[p]})=\nabla^rG_r^{[p]}$. Therefore for every $0<j\le r$ we have
$r!\nabla_{r+2}^{-1-j}(f^{[p]})=\nabla_r^{-j}\bigl(\nabla_{r+2}^{-1}(f^{[p]})\bigr)=\nabla_r^{-j}\bigl(\nabla_{-r}^r((G^{[p]})_r)  \bigr)
=\nabla_{-r}^{r-j}(G_r^{[p]}).$
\end{proof}

\begin{remark}
As observed in proposition \ref{prop:nabla-1-j} the  element $G^{[p]}$ is an analytic section of ${\rm Sym}^r({\rm H}_{\bE})$ on $\cX_{>p+2}$. Its pullback to $X(N,p^2)^{\rm an}$ extends to a global, locally analytic section of the same sheaf which is analytic on the inverse image of $\cX_{>p+2}$. This follows as $f^{[p]}$ is a classical modular form on $X(N,p^2)^{\rm an}$ and $G$ is a global, locally analytic section on $X_1(N)^{\rm an}$. 
\end{remark}

We fix an $\fa\in \cH(c, \fN)$.  We recall that $G^{[p]}\in {\rm H}^0\bigl(\cX_{>p+2}, {\rm Sym}^r({\rm H}_E)\bigr)$ and as $\nabla_{r+2}^{-1}(f^{[p]})=G^{[p]}$, then $G^{[p]}$ can be evaluated at $x_\fa:=(\fa\ast (A_0,t_0,D))\in \cX(N,p^2)$.
We decompose $G^{[p]}(x_\fa)$ in   ${\rm Sym}^r{\rm H}^1_{\rm dR}(\fa\ast A_0)$ according to the $K$-action as follows (we remind the reader that in this section ${\rm Sym}^r{\rm H}^1_{\rm dR}(\fa\ast A_0) $ is seen as a finite dimensional vector space over a finite extension of $\Q_p$ with an action of $K$, so the decomposition is a full decomposition)
$$
G^{[p]}(x_\fa)=\sum_{j=0}^{r} (-1)^j(G^{[p]})_j, \mbox{ where } (G^{[p]})_j\in \Bigl({\rm Sym}^r\bigl({\rm H}^1_{\rm dR}(\fa\ast A_0)\bigr)\Bigr)_{\tau^{r-j}, \overline{\tau}^{j}}.
$$ 
We denote by $\omega_\fa$ the pull-back of $\omega_0$ via the isogeny dual to the isogeny $A_0\to A_0/A_0[\fa]$,  and by $\eta_\fa$ an element of ${\rm H}^1_{\rm dR}(\fa\ast A_0)$ such that the pair $\omega_\fa, \eta_\fa$ is a basis compatible with the $K$ decomposition. 
Moreover for $0\le j\le r$ we denote by $G^{[p]}_j\bigl(\fa\ast (A_0,t_0,\omega_0)\bigr)\in L$ the element such that
$(G^{[p]})_j= G^{[p]}_j\bigl(\fa\ast(A_0,t_0,\omega_0)\bigr)\omega_\fa^j\eta_\fa^{r-j}.$

\begin{proposition}\label{prop:deta-1-j} We have $$\delta_k^{-1-j}\bigl(f^{[p]}\bigr)\bigl( \fa\ast (A_0, t_0, D, \omega)\bigr)=\frac{\Omega_{p,n}^{r-2j}}{j!}G^{[p]}_j\bigl(\fa\ast (A_0,t_0,D,\omega_0)\bigr).$$

\end{proposition}

\begin{proof} We denote $\bigl(B, t_B, D_B, \omega_B\bigr):=\fa\ast\bigl(A_0, t_0, D, \omega_0\bigr)$. Thanks to Proposition \ref{prop:nabla-1-j} we have $$\delta_k^{-1-j}\bigl(f^{[p]}\bigr)(B, t_B, D_B,\omega_B)=\frac{1}{r!} \bigl(\nabla^{r-j}_r (G^{[p]})_{r}\bigr)_0 (B, t_B, D_B, \omega_B),$$
where $\bigl(\nabla^{r-j}_r (G^{[p]})_{r}\bigr)_0\bigl(B, t_B, D_B, \omega_B\bigr)$ is the $0$-th component of  $\bigl(\nabla^{r-j}_{-r} (G^{[p]})_{r}\bigr) \bigl(B, t_B, D_B, \omega_B\bigr)$ for the action of $K$, i.e., the coefficient of $\omega_B^{r-2j}$ in $(\bW_{r-2j})_{y_\fa}$.

Denote by $\cD'$ the residue class in the modular curve $X_1(N)$ at the point $y_\fa:=(B,t_B)\in \cX$. As $y_\fa$ is a smooth point of the special fiber of $X_1(N)$, $\cD'$ is isomorphic to the $p$-adic (wide) open unit disk centered at $y_\fa$. Let $\phi\colon \cD\to \cD'$ be the inverse image of $\cD'$ in $X(N,p^2)_{p+2}$. It is a finite map and $\cD$ contains  the  $\cO_L$-valued point $x_\fa$ defined by $(B, t_B,D_B)$. 

We first work over $\cD'$, namely we denote ${\rm H}_{\cD'}:={\rm H}_{\bE}\vert_{\cD'}$. We have
${\rm H}_{\cD'}\cong {\rm H}_{y_\fa}\otimes \cO_{\cD'}$ so that ${\rm H}_{y_\fa}={\rm H}^1_{\rm dR}(B)$
is isomorphic to the space of horizontal sections for the Gauss-Manin connection $\nabla$ on ${\rm H}_{\cD'}$.
Let us fix a basis $\omega$, $\eta$ respecting the $K$-decomposition of ${\rm H}_{y_\fa}={\rm H}_{B}$ and pairing to $1$ via the paring on ${\rm H}_{B}$ induced by the principal polarization on $B$. The Hodge filtration of ${\rm H}_{\cD'}$ is generated by an element  $\omega':=\omega\otimes 1+b(\eta\otimes 1)$, for some global section $b\in \cO_{\cD'}$ such that $b(y_\fa)=0$. Then $\omega'$ and $\eta':=\eta\otimes 1$ provide a basis for of ${\rm H}_{\cD'}$. 
The Kodaira-Spencer isomorphism ${\rm KS}$ identifies $\omega^{\otimes 2}_{\bE}\vert_{\cD'}\cong \Omega^1_{\cD'}$. Since $\nabla\bigl(\omega' \bigr)=  (\eta\otimes 1) \otimes db$ then  ${\rm KS}$ is defined by $db$ with respect to the basis $\omega'$ and $\eta'$ so that $db$ is a basis element of $\Omega^1_{\cD'}$. Let $\partial$ be the derivation on $\cO_{\cD'}$ dual to $db$.
The Gauss-Manin connection $\nabla\colon {\rm H}_{\cD'} \to {\rm H}_{\cD'}\otimes \Omega^1_{\cD'}$ is defined by $\nabla((\omega')=\eta'\otimes db$ and $\nabla(\eta')=0$. 

We notice that  the filtration $\Fil_\bullet \bW_{-r}[1/p]\vert_{\cD'}$  admits $(\omega')^{-r-j} (\eta')^j $, for $j\in \N$, as a basis as a $\cO_{\cD'}$-module.  In fact, consider $\pi\colon \bV_0({\rm H}_{\cD'}^\sharp,s)\to \cD'$ the analytic generic fiber of the vector bundle with a marked sections of Theorem \ref{thm:VBMSrecall}.  If we invert $p$, then ${\rm H}^\sharp_{\cD'}[1/p]={\rm H}_{\cD'} $. Thus, $\pi$ factors thorugh $\pi'\colon \bV'({\rm H}_{\cD'})\to \cD'$ that classifies sections of the dual of ${\rm H}_{\cD'}$ that are invertible on $\omega'$. Consider  $\bW_{-r}^{\rm alg}:=\pi_\ast'(\cO_{\bV'({\rm H}_{\cD'})})[-r]$, the functions on which $\mathbb{G}_m$, acting by scalar multiplication on ${\rm H}_{\cD'}$, acts by the algebraic character: $\mathbb{G}_m\to \mathbb{G}_m$, $g\mapsto g^{-r}$.   It admits $(\omega')^{-r-j} (\eta')^j $, for $j\in \N$,  as an $\cO_{\cD'}$-basis. It is endowed with a filtration such that $\Fil_n \bW_{-r}^{\rm alg}$ is spanned by $(\omega')^{-r-j} (\eta')^j $ for $0\leq j\leq n$. The map  $\bV_0({\rm H}_{\cD'}^\sharp,s)\to  \bV'({\rm H}_{\cD'})$  identifies  $\Fil_n \bW_{-r}^{\rm alg} \cong \Fil_n \bW_{-r}[1/p]\vert_{\cD'}$. 

Then we can compute $\Bigl(\nabla^{r-j}_{-r}(G^{[p]})_r\Bigr)_0$ using the connection on $\bW_{-r}^{\rm alg}$ as follows: For every $n\ge 0$ and $a$ a section of $\cO_{\cD'}$, we have (using the explicit description of the connection in section \S 3.4.1 of \cite{andreatta_iovita})
$$
\nabla_{-r}\bigl(a(\omega')^{-r-n}(\eta')^n\bigr):=\partial(a)(\omega')^{-r-n+2}(\eta')^n+(-r-n)a(\omega')^{-r-n+1}(\eta')^{n+1}.
$$
It follows that we can write $(G^{[p]})_r=g_r(\omega')^{-r}$ in $\bW_{-r}|_{\cD'}$, with $g_r$ a section of $\cO_{\cD'}$ in a neighborhood of $y_{\fa}$, and we have:
$$
\nabla_{-r}^{r-j}\bigl((G^{[p]})_r  \bigr)=\nabla_{-r}^{r-j}\bigl(g_r(\omega')^{-r}\bigr)=
\partial^{r-j}(g_r)(\omega')^{r-2j}+M,
$$
where $M$ contains terms of the form $a_i(\omega')^{r-2j-i}(\eta')^i$ with $i>0$ and $a_i$ sections of $\cO_{\cD'}$.
It follows that $\nabla_{-r}^{r-j}\bigl((G^{[p]})_r  \bigr)_0=\partial^{r-j}(g_r)$.

 Now in order to calculate this quantity we look at ${\rm Sym}^r\bigl({\rm H}_\bE|_{\cD'}\bigr)$ and its connection $\nabla_r$.
We write  ${\rm Sym}^r\bigl({\rm H}_{\bE}\vert_{\cD'}\bigr)=\oplus_{j=0}^r  \cO_{\cD'} (\omega')^j(\eta')^{r-j}$. If we specialize at $y_\fa$ this decomposition induces the decomposition of ${\rm Sym}^r {\rm H}_{y_\fa}$ into eigenspaces for the $K$-action.

Write  $G^{[p]}\vert_{\cD'}=\sum_{j=0}^r (-1)^{j}g_{j} (\omega')^{r-j}(\eta')^{j}$ and $f^{[p]}\vert_{\cD'}=h (\omega')^{k}$. We observe that $\eta'({\rm mod }\ \omega'\cO_{\cD'})=(\omega')^{-1}$ and therefore
the coeffcient denoted $g_r$ in this expression of $G^{[p]}$ and the one appearing in the previous expression of $(G^{[p]})_r$ are the same. 
Possibly after shrinking $\cD'$ we may assume that the $g_{j}$'s,  $G^{[p]}$ and $h$ are defined  and analytic on  $\cD'$ (we recall that $G^{[p]}$ and $f^{[p]}$ are only an overconvergent section of ${\rm Sym}^r({\rm H}_{\bE})$, respectively an overconvergent modular form on $\cX$). We then have $\nabla(G^{[p]}\vert_{{\cD}'})=f^{[p]}\vert_{\cD'}$ implies that $\partial(g_{0})=h$ and $\partial(g_{j})-(r-j+1) g_{j-1}=0$ for $1\leq j\leq r$. 

In conclusion, for all $0\le j\le r$, and on the annulus where all the sections are defined (in particular at $y_\fa$)  we have $\Bigl(\nabla^{r-j}_{-r}(G^{[p]})_r\Bigr)_0=\partial^{r-j}(g_{r})   =(r-j)!g_{j} $. Evaluating at $y_\fa$ we conclude that
$$\bigl(\nabla^{r-j}_{-r} (G^{[p]})_r\bigr)_0 (B, t_B,\omega_B)=\bigl(\nabla^{r-j}_{-r} (G^{[p]})_r\bigr)_0(y_\fa)=(r-j)! g_{j}(y_\fa)=(r-j)! (G^{[p]})_{j}(B, t_B, \omega_B).$$
Now we pull-back these equalities to $\cD$ via the tamely ramified map $\phi\colon \cD\to \cD'$ and we evaluate at $x_\fa$. We obtain the claimed equality 
$$\delta_k^{-1-j}\bigl(f^{[p]}\bigr)(B, t_B, D_B,\omega_B)=\frac{1}{j!} G^{[p]}_{j}(B, t_B, D_B, \omega_B).$$

\end{proof}

In particular, we have $\displaystyle L_p(f,\chi)= \frac{\Omega_{p,n}^{r-2j}}{j! \vert (\cO_K/\fN)^\ast \vert}\sum_{\fa\in \cH(c,\fN) } \chi_{-1}^{-1-j}(\fa)G^{[p]}_j
\bigl(\fa\ast (A_0, t_0, D, \omega_0)$.
In order to prove (\ref{eq: LpG0}) and conclude the proof of Theorem \ref{thm:AJ}, it remains to show:
$$
\frac{1}{\vert (\cO_K/\fN)^\ast \vert} \sum_{\fa\in \cH(c,\fN) } \chi_{-1}^{-1-j}(\fa)G^{[p]}_j
\bigl(\fa\ast (A_0, t_0, D, \omega_0)=\sum_{\fa\in \Pic(\cO_c)} \chi_{-1}^{-1-j}(\fa)G_j
\bigl(\fa\ast (A_0, t_0, \omega_0).
$$

The calculations in the Appendix \ref{sec:nablaUV} imply: for $i\in \{1,2\}$ we have
$f|V^i=\bigl(\nabla(G)\bigr)|V^i=p^{-i}\nabla(G|V^i)$.  It follows that we have:
$$
\nabla\Bigl(G|\bigl(1-\frac{1}{p}a_pV+\epsilon(p)p^{k-3}V^2)\Bigr)=f|(1-a_pV+\epsilon(p)p^{k-1}V^2
=f^{[p]}.
$$
Therefore on $\cX_{p+2}$ we have: $G^{[p]}=\nabla_k^{-1}(f^{[p]})=G|\bigl(1-\frac{1}{p}a_pV+\epsilon(p)p^{k-3}V^2)$ and this identity is compatible with the CM decomposition at each point $x_\fa$, for $\fa\in \cH(c,\fN) $. Moreover $G$ is a global locally analytic section on $X_1(N)^{\rm an}$, in particular it is defined and it is analytic on all the supersingular residue classes.

From now on the proof follows as in the proof of Proposition \ref{prop:comparisonclassicalinert} in the inert case and 
of Proposition \ref{prop:comparisonclassicalramified} in the ramified case.

\subsection{The Kronecker limit formula.}

We next prove the analogue of a particular case of the  Kronecker limit formula \cite[\S 10]{katzEisenstein} (more specifically the account in \cite[Thm.~1.3]{bertolini_etc.})  for $p$ non-split in $K$, proceeding as in \S \ref{sec:GrossZagier}. We take $k=2$ and $\epsilon$ an even, non trivial character.

Let $\chi$ be an algebraic Hecke character in the space $\Sigma_{cc}(k,c,\fN,\epsilon)$ defined in \S \ref{sec:eisenstein} of infinity type $(1,1)$. We assume that  $c=dp^n$  and $(d,p)=1$ and $p^n$ is the $p$-part of the conductor of $\chi$.
Let us remark that $\chi$ can be viewed as an element of the space  $\widehat{\Sigma}^{(2),p^n}(k,c,\fN,\epsilon)$ of Remark \ref{rmk:p2conductos} so that we can evaluate the $p$-adic $L$-function  $L_p(E_{2,\epsilon}, \_)$ of Definitions \ref{def:Lpinert} (if $p$ is inert) and \ref{def:Lpramified} (if $p$ is ramified) at $\chi$. 

Let $u_\epsilon$ be a modular unit, namely an element $u_\epsilon \in {\rm H}^0\bigl(Y_1(N),\cO_{Y_1(N)}^\ast\bigr)$ on the open modular curve $Y_1(N)$, such that $\displaystyle {\rm dlog} u_\epsilon=\frac{d u_\epsilon}{u_\epsilon}=E_{2,\epsilon}$. Then,

\begin{proposition}\label{prop:Kronecker}   Assume that $n\geq n_k(p)$ as in Definition \ref{definition:b(p)}.   Then, $$
L_p(E_{2,\epsilon},\chi)=\sum_{\fa\in {\rm Pic}(\cO_c)} \frac{{\bf N}(\fa)}{\chi(\fa) } \log_p (u_{\epsilon})\bigl(\fa\ast  (A_0,t_0,\omega_0)\bigr)  .$$
\end{proposition}

\subsubsection{On the $p$-adic logarithm.}

Let $X$ be a wide open disk or annulus in $\PP^1_{\C_p}$ and $g\in \cO_{X}(X)^\ast$.
We wish to study the Coleman integral of the differential form $\displaystyle \omega_g:={\rm dlog}(g)=
\frac{dg}{g}$ on $X$. Let us denote by $t$ a parameter of $X$. We work over $\C_p$ and denote by $\F$ its residue field.
With the notations above we have the following

\begin{lemma}
\label{lemma:diskannulus}
a) If $X$ is a disk and $g\in \cO_X(X)^\ast$ then there is $a\in \C_p^\ast$, $h\in \cO_X(X)$ with the property $|h(x)|_X<1$ for all $x\in X$
such that $g=a(1+h).$

b) If $X$ is an annulus with parameter $t$ and $g\in \cO_X(X)^\ast$, then there are: $a\in \C_p^\ast$, a section $h\in \cO_X(X)$ with the property $|h(x)|_X<1$ for all $x\in X$ and $n\in \Z$ such that $g=at^n(1+h)$ and ${\rm Res}_t(\omega_g)=n$. 
\end{lemma}

\begin{proof}
This lemma is probably well known but we'll sketch the proof of a) for the convenience of the reader and leave she/he to think about b). As the power series expansion of 
$g$ does not change if we restrict to a smaller disk, it is enough to prove the lemma for all affinoid disks contained in $X$, i.e. it is enough to prove it for $X=\{x\in \PP^1_{\C_p}\ |\  |x|\le 1\}$. 
On $\cO_X$ then we have the norm $|g  |_X:=\sup_{x\in X}|g(x)|$, which satisfies the maximum modulus principle, i.e. the norm of $g$ can be calculated on the annulus $Y:=\{x\in X\  \ \ |x|=1\}\subset X$. As $\overline{X}={\rm Spec}(\F[t])$ is irreducible the norm $|\ |_X$ is multiplicative.
Let $c\in \C_p^\ast$ be such that $|cg|_X=1$, then $|(cg)^{-1}|_X=1$ i.e. $\overline{cg}\in \bigl(   \F[t]\bigr)^\ast=\F^\ast$. Let $b\in \cO_{\C_p}^\ast$ be such that  $\overline{bcg}=1$. 
Then clearly if we set $a=bc$ and $h=a^{-1}g-1$ we have $g=a(1+h)$ with $h$ satisfying the desired property. 
\end{proof}

Let us denote by ${\rm log}_p\colon \C_p^\ast\to \C_p$ the locally analytic homomorphism uniquely determined by the properties:  \enspace i)~${\rm log}_p(p)=0$ and \enspace ii) If $x\in \C_p^\ast$ is such that $|x|<1$ then $\displaystyle {\rm log}_p(x)=\sum_{n=1}^\infty (-1)^{n-1}\frac{(x-1)^n}{n}$. \enspace Then $\displaystyle d\bigl({\rm log}_p(z)\bigr)=\frac{dz}{z}.$

Let us now remark that if $X$ is a wide open disk and $g\in \cO_X(X)^\ast$, using lemma \ref{lemma:diskannulus},  the function $\displaystyle G:={\rm log}_p(a)+\sum_{n=1}^\infty(-1)^{n-1}\frac{h^n}{n}\in \cO_X(X)$ and it satisfies $\displaystyle dG=\omega_g:=\frac{dg}{g}$. We'll use the notation $G:={\rm log}_p(g)$.

If $X$ is an annulus with parameter $t$ let us denote by $T:={\rm log}_p(t)$ a new variable such that
$\displaystyle dT={\rm dlog}(t)=\frac{dt}{t}$. Let $\cO_{\rm log}(X):=\cO_X(X)[T]$. Let $g\in \cO_X(X)^\ast$
be such that ${\rm Res}_t(\omega_g)=n$. Then the function $G\in \cO_{\rm log}(X)$ defined by:
$$
G:={\rm log}_p(a)+nT+\sum_{n=1}^\infty(-1)^{n-1}\frac{h^n}{n}
$$ 
satisfies $\displaystyle dG=\omega_g:=\frac{dg}{g}$. We'll use the notation $G:={\rm log}_p(g)\in \cO_{\rm log}(X)$.

We remark that we have the following rigidity property: if $V\subset X$ is a non-void admissible open subspace and $G$, $G'\in \cO_{\rm log}(X)$
are such that $G\vert_V=G'\vert_V$ that $G=G'$.

\subsubsection{The proof of Proposition \ref{prop:Kronecker} }
We now come back to our modular unit $u_\epsilon$.  Let 
$G:={\rm log}_p(u_\epsilon)$ denote the locally analytic function on  $\mathcal{Y}:=\bigl(\widehat{Y}_1(N)\bigr)^{\rm an}$ which is ${\rm log}_p(u_{\epsilon})\vert_X\in \cO_X(X)$ for every 
residue class $X$ of a point in $\mathcal{Y}$. Here $\mathcal{Y}$ is the rigid generic fiber of the formal scheme $\widehat{Y}_1(N)$, i.e. the
rigid space which is the complement in $X_1(N)^{\rm an}$ of the residue classes of all the cusps.   Then $G$ is a Coleman primitive of $\displaystyle \frac{du_\epsilon}{u_\epsilon}=E_{2,\epsilon}$ on $\mathcal{Y}$, uniquely defined
 and which satisfies the rigidity principle stated at the end of the previous section.  
 Proposition \ref{prop:nabla-1-j} and Proposition \ref{prop:deta-1-j}   with $r=j=0$ imply that 
$$
L_p(E_{2,\epsilon},\chi)=\frac{1}{\vert (\cO_K,\fN)^\ast\vert} \sum_{\fa\in \cH(c, \fN)} \frac{{\bf N}(\fa)}{\chi(\fa) } \delta_2^{-1}\bigl(E_{2,\epsilon}^{[p]}\bigr)\bigl(\fa\ast  (A_0,t_0,\omega)\bigr) =
$$ 
$$
=\frac{1}{\vert (\cO_K,\fN)^\ast\vert} \sum_{\fa\in \cH(c, \fN)}  \frac{{\bf N}(\fa)}{\chi(\fa) } G^{[p]}\bigl(\fa\ast  (A_0,t_0,\omega)\bigr).
$$
The arguments at the end of the previous section imply that the latter value coincides with $$ \sum_{\fa\in {\rm Pic}(\cO_c)} \frac{{\bf N}(\fa)}{\chi(\fa) } \log_p (u_{\epsilon})\bigl(\fa\ast  (A_0,t_0,\omega_0)\bigr) ,$$ as let us recall that the point $\fa\ast  (A_0,t_0)$ being supersingular is not in the residue class of any cusp. The claim follows.

\section{Appendix: $\nabla$, $U$ and $V$.}\label{sec:nablaUV}

Let $\widehat{Y}^{\rm ord}$ be the formal open subscheme of $\widehat{X}_1(N)$ corresponding the ordinary locus (with the cusps removed). Let $\bE\to \widehat{Y}^{\rm ord}$ be the universal elliptic curve. Denote by $\varphi\colon \bE \to \bE'$  the quotient by the canonical subgroup of order $p$ and by $\varphi^\vee\colon \bE'\to \bE$ the dual isogeny. We have a unique morphism  $\Phi\colon \widehat{Y}^{\rm ord}\to  \widehat{Y}^{\rm ord}$ such that the pull-back of $\bE$ is $\bE'$. It is a finite and flat mophism of degree $p$. Let $r\ge 0$ be an integer and denote by $\cF_r$ either the sheaf ${\rm Sym}^r{\rm H}_\bE$ or the sheaf
${\rm Sym}^r{\rm H}_\bE^{\rm loc}$ of locally analytic sections of the first sheaf as defined in section \S \ref{sec:GrossZagier}.

\

{\it The $V$ operator:}\enspace The operator $V\colon\cF_r\to \cF_r$  is defined by the following rational map  on ${\rm H}_E$ $$V(\gamma):=\bigl((\varphi^\vee)^\ast)^{-1}\bigl(\Phi^\ast(\gamma)\bigr)=\frac{\varphi^\ast}{p}\bigl(\Phi^\ast(\gamma)\bigr);$$here $\gamma\in  {\rm H}_\bE$ and $\Phi^\ast(\gamma)$ is viewed as an element of ${\rm H}_{\bE'}\cong \Phi^\ast({\rm H}_\bE)$. The map $((\varphi^\vee)^\ast\colon {\rm H}_\bE \to {\rm H}_{\bE'}$ is the pull-back via $\varphi^\vee$ and similarly $\varphi^\ast\colon  {\rm H}_{\bE'} \to {\rm H}_\bE$ is induced by $\varphi$. 

\

{\it The $U$ operator:}\enspace The operator $U\colon \Phi_\ast \bigl( \cF_r\bigr) \to \cF_r$ is defined by  the rational map on ${\rm H}_\bE$ given by the composite $$U:=\frac{1}{p} {\rm Tr}_\Phi \circ (\varphi^\vee)^\ast  $$where $(\varphi^\vee)^\ast\colon \Phi_\ast\bigl({\rm H}_\bE\bigr) \to \Phi_\ast\bigl(\Phi^\ast {\rm H}_\bE\bigr)$  is defined by pull-back via $\varphi^\vee$ and $ {\rm Tr}_\Phi\colon
 \Phi_\ast\bigl(\Phi^\ast {\rm H}_\bE\bigr) \to {\rm H}_\bE$ is the trace map (as coherent sheaves) via the finite and flat map  $\Phi$.  Notice that $U \circ V= {\rm Id}$. 

\begin{proposition}
\label{prop:nablapdep} Let $r\in \N$ and consider the connection $\nabla\colon \cF_r \to \cF_{r+2}$. We have the formulae $$ \nabla \circ V=p V  \circ \nabla  \qquad\mbox{ and}\qquad   \nabla  \circ U=\frac{1}{p} U   \circ \nabla.$$In particular if $\nabla(G)=f$ then  $\nabla(G^{[p]})=f^{[p]} $ where $(\  -\ )^{[p]}$ stands for the $p$-depletion operator  $(1- V \circ U)$.

\end{proposition}

\begin{proof} The  compatibility of $\nabla$ with the $p$-depletion clearly follows from the commutation formula for $\nabla$ with $V$ and $U$ respectively. 

The isogeny $\varphi\colon \bE \to \bE'$ composed with the projection $\pi\colon \bE'=\bE\times_{\widehat{Y}^{\rm ord}}^\Phi \widehat{Y}^{\rm ord} \to \bE$ induces the map $\eta := {\varphi^\ast} \circ \Phi^\ast\colon  H_\bE \to H_\bE$. Let $d \Phi\colon \Omega^1_{ \widehat{Y}^{\rm ord}} \to  \Omega^1_{ \widehat{Y}^{\rm ord}}$ be the map induced by pull-back via $\Phi$.  By functoriality we get a commutative diagram 

$$\begin{matrix} {\rm H}_{\bE} & \stackrel{\nabla}{\to} &  {\rm H}_{\bE}  \otimes \Omega^1_{ \widehat{Y}^{\rm ord}}\cr 
\eta  \downarrow & & \downarrow \eta \otimes d \Phi \cr 
{\rm H}_{\bE} & \stackrel{\nabla}{\to} &  {\rm H}_{\bE}  \otimes \Omega^1_{ \widehat{Y}^{\rm ord}}. \cr 
\end{matrix}$$

Recall that the Kodaira-Spencer isomorphism $\omega_\bE^2\cong  \Omega^1_{ \widehat{Y}^{\rm ord}}$ is defined by restricting $\nabla$ to $\omega_\bE$ and projecting onto $\omega_\bE^\vee \otimes \Omega^1_{ \widehat{Y}^{\rm ord}}$. The commutative diagram above implies that  the map $d\Phi\colon \Omega^1_{ \widehat{Y}^{\rm ord}} \to \Omega^1_{ \widehat{Y}^{\rm ord}}$ induces the map  $$pV=p \bigl((\varphi^\vee)^\ast\bigr)^{-2}\circ \Phi^\ast= \bigl(\varphi^\ast \otimes  (\varphi^\vee)^\ast)^{-1} \bigr)\circ \Phi^\ast\colon \omega_\bE^2 \to  \omega_\bE^2.$$Indeed $\eta$ is $\varphi^\ast$ on $\omega_\bE$ and $(\varphi^\vee)^\ast$ on $\omega_\bE^\vee$. 
Since by construction $\eta=p V$, we conclude that the following diagram commutes:

$$\begin{matrix} {\rm H}_{\bE} & \stackrel{\nabla}{\to} &  {\rm H}_{\bE}  \otimes \omega_\bE^2 \cr 
V  \downarrow & & \downarrow V  \otimes (pV) \cr 
{\rm H}_{\bE} & \stackrel{\nabla}{\to} &  {\rm H}_{\bE}  \otimes\omega_\bE^2 .\cr 
\end{matrix}$$
Passing to $\cF_r$ we get the statement on the $V$-operator. 

\

Next we study the $U$ operator.  The map $(\varphi^\vee)^\ast$, induced by $\varphi^\vee\colon \bE'\to \bE$, is compat
ible with the Gauss-Manin connection by functoriality. On the other hand $\Phi^\ast\bigl(\Omega^1_{ \widehat{Y}^{\rm ord}}\bigr)=p\Omega^1_{ \widehat{Y}^{\rm ord}}$ (as can be seen using Serre-Tate coordinates) so that  $\frac{1}{p}\Phi_\ast \circ \Phi^\ast\bigl(\Omega^1_{ \widehat{Y}^{\rm ord}}\bigr)=\Phi_\ast \bigl(\Omega^1_{ \widehat{Y}^{\rm ord}}\bigr)$. We then  have the commutative diagram

$$\begin{matrix} \Phi_\ast ({\rm H}_{\bE}) & \stackrel{\nabla}{\to} &  \Phi_\ast (({\rm H}_{\bE})  \otimes  \Phi_\ast\bigl(\Omega^1_{ \widehat{Y}^{\rm ord}}\bigr)[p^{-1}] \cr 
(\varphi^\vee)^\ast  \downarrow & & \downarrow (\varphi^\vee)^\ast \otimes  {\rm Id} \cr 
\Phi_\ast ({\rm H}_{\bE'}) & \stackrel{\nabla}{\to} &  \Phi_\ast({\rm H}_{\bE'})  \otimes \Bigl(\Phi_\ast \circ \Phi^\ast\bigl(\Omega^1_{ \widehat{Y}^{\rm ord}}\bigr)\Bigr)[p^{-1}]  \cr 
{\rm Tr}_\Phi  \downarrow & & \downarrow {\rm Tr}_\Phi   \cr 
{\rm H}_{\bE} & \stackrel{\nabla}{\to} &  {\rm H}_{\bE}  \otimes\Omega^1_{ \widehat{Y}^{\rm ord}}[p^{-1}].\cr
\end{matrix}$$

(here we use $ \Phi_\ast({\rm H}_{\bE'}) = \Phi_\ast\circ \Phi^\ast ({\rm H}_{\bE})$)  Using this commutative diagram and the Kodiara-Spencer isomorphism, we deduce that the map $$ \Phi_\ast\bigl(\Omega^1_{ \widehat{Y}^{\rm ord}}\bigr) [p^{-1}] =\Phi_\ast \circ \Phi^\ast\bigl(\Omega^1_{ \hat{Y}^{\rm ord}}\bigr)[p^{-1}] \stackrel{{\rm Tr}_\Phi}{\to} \Omega^1_{ \widehat{Y}^{\rm ord}}[p^{-1}]$$coincides with the rational map $$\rho:={\rm Tr}_\Phi \circ \bigl((\varphi^\vee)^\ast \otimes (\varphi^\ast)^{-1}\bigr)\colon \Phi_\ast\bigl(\omega_\bE^2\bigr) \to \Phi_\ast\circ \Phi^\ast \bigl(\omega_\bE^2\bigr) \to  \omega_\bE^2.$$As $(\varphi^\ast)^{-1}= p^{-1}  (\varphi^\vee)^\ast$ then $\rho= \frac{1}{p} U$.  We conclude that the following diagram commutes:

$$\begin{matrix} \Phi_\ast ({\rm H}_{\bE})  & \stackrel{\nabla}{\to} &  \Phi_\ast ({\rm H}_{\bE}) \otimes  \Phi_\ast (\omega_\bE^2) \cr 
U  \downarrow & & \downarrow U  \otimes (p^{-1} U) \cr 
{\rm H}_{\bE} & \stackrel{\nabla}{\to} &  {\rm H}_{\bE}  \otimes\omega_\bE^2 .\cr 
\end{matrix}$$
Passing to $\cF_r $ we get the claim for the $U$ operator.

\end{proof}

\section{Appendix: the case $p$ split}

We explain how to recover the construction of the $p$-adic $L$-functions of \cite{bertolini_darmon_prasana} in the easier case that $p$-splits in $K$ following the approach outlined above. We recall the formula. For $\chi\in \widehat{\Sigma}^{(2)}$
with $\nu:=w(\chi)\in W(\Q_p)$, we have in \cite[Thm 5.9]{bertolini_darmon_prasana}:

\begin{equation}\label{eq:BDPLp}
L_p(F,\chi^{-1})=\sum_{\fa\in {\rm Pic}(\cO_{c})}
 \chi_\nu^{-1}(\fa)\vartheta_k^\nu\bigl(F^{[p]}\bigr)
\bigl(\fa\ast (A_0, t_0, \Omega_{\rm can})\bigr),  
\end{equation} Here $\vartheta$ is Serre's theta operator that on $q$-expansions sends $g(q)\mapsto q\cdot \frac{d g(q)}{dq}$,  $A_0$ is an ordinary elliptic curve with full CM by the order $\cO_c$ and $t_0$ is a $\Gamma_1(\fN)$-level structure on $A_0$.  Its formal group is isomorphic to $\widehat{\mathbb{G}}_m$ and $\Omega_{\rm can}$ is the invariant differential defined, via such an isomorphism, by the standard differential on $\mathbb{G}_m$. In this case notice that the $p$-adic completion of $\cO_K$ is isomorphic to $\Z_p\times \Z_p$. We assume as in loc.~cit.~that $c$ is prime to $p$.

\
 
We relate this to our approach, hoping that it will help to better understand the general situation.  We first introduce the sheaves $\bW_{k}$ and the interpolation of the connection  $(\nabla_k)^{\nu}$.

Fix a ring $R$,  $p$-adically complete and separated, $p$-torsion free. Fix weights $k$, $\nu \colon \Z_p^\ast \to R^\ast$ such that there exist elements $u_k$ and $u_\nu \in p R$ and finite characters $\epsilon_k$ and $\epsilon_\nu$ of $\Z_p^\ast$  such that for all $t\in \Z_p^\ast$,  we have $$k(t)=\epsilon_k(t) {\rm exp}\bigl(u_k\cdot {\rm log}(t)\bigr),\qquad \nu(t)=\epsilon_\nu(t){\rm exp}\bigl(u_\nu\cdot {\rm log}(t)\bigr).$$Let $\fX^{\rm ord}$ be the open formal scheme defining the ordinary locus in the completion $\widehat{X}_1(N)$ of $X_1(N)$ along the special fiber. Let $\fIG_\infty\to \fX^{\rm ord}$ be the  Igusa tower, classifying isomorphisms $\widehat{\bE}\cong \widehat{\mathbb{G}}_m$ between the formal group of the genearlised elliptic curve $\bE$ over $\fX^{\rm ord}$ and the formal torus $\widehat{\mathbb{G}}_m$. It is a Galois cover with group $\Z_p^\ast={\rm Aut}(\widehat{\mathbb{G}}_m)$: given an isomorphism $\varphi\colon \widehat{\bE}\cong \widehat{\mathbb{G}}_m$ we let $\alpha\in \Z_p^\ast$ act by sending $\varphi\mapsto \alpha^{-1} \varphi$. Over $\fIG_\infty$ the universal isomorphism $\varphi^{\rm uniuv} \colon \widehat{\bE}\cong \widehat{\mathbb{G}}_m$ defines a  canonical generator $\omega_{\rm can}$ of the invariant differentials  of $\bE$ relative to $\fX^{\rm ord}$ as the pull-back via $\varphi^{\rm univ}$ of the canonical invariant differential of $\mathbb{G}_m$.

Let $\omega_\bE$, resp. ${\rm H}_{\bE}$ be  the relative differentials, resp. the  logarithmic de Rham cohomology of $\bE/\fX^{\rm ord}$. We ifirst introduce the  relevant objects using the formalism of VBMS (vector bundles with marked sections) discussed in Section \ref{sec:VBMS}. Define $$\bV_0\bigl({\rm H}_{\bE},\omega_{\rm can}\bigr)\longrightarrow \bV_0\bigl(\omega_{\bE},\omega_{\rm can}\bigr)  \longrightarrow \fIG_\infty$$to be the $p$-adic formal schemes spaces classifying sections of the dual ${\rm H}_{\bE}^\vee$ , resp. $\omega_\bE^\vee$ that are $1$ on $\omega_{\rm can}$.   The first map is induced by the inclusion $\omega_\bE\to {\rm H}_\bE $ (the Hodge filtration). The second map is an isomorphism $\bV_0\bigl(\omega_{\bE},\omega_{\rm can}\bigr)  \cong \fIG_\infty$. These spaces are endowed with compatible  actions of $\Z_p^\ast$, considering the Galois action on $\fIG_\infty$ and the scalar multiplication on $\omega_\bE$ and on ${\rm H}_\bE$; in particular, such action preserves $\omega_{\rm can}$ so that the isomorphism  $\bV_0\bigl(\omega_{\bE},\omega_{\rm can}\bigr)  \cong \fIG_\infty$ is $\Z_p^\ast$-equivariant as claimed.  We define $$\omega^{k,o}\subset {\rm H}^0\bigl(\bV_0\bigl(\omega_{\bE},\omega_{\rm can}\bigr), \cO_{\bV_0\bigl(\omega_{\bE},\omega_{\rm can}\bigr)}\widehat{\otimes} R\bigr) ={\rm H}^0\bigl(\fIG_\infty, \cO_{\fIG_\infty}\widehat{\otimes} R\bigr) ,$$as the subspace on which $\Z_p^\ast$ acts via the character $k$. Let $\omega^k:=\omega^{k,o}[p^{-1}]$.   Similarly let $$\bW_k^o\subset {\rm H}^0\bigl(\bV_0\bigl({\rm H}_{\bE},\omega_{\rm can}\bigr), \cO_{\bV_0\bigl({\rm H}_{\bE},\omega_{\rm can}\bigr)}\widehat{\otimes} R\bigr) $$be the subspace of functions such that $\Z_p^\ast$ acts via the character $k$. Set $\bW_k:=\bW_k^o[p^{-1}]$. It is is a $\Q_p$-Banch space with unit ball $\bW_k^o$. The morphism $\bV_0\bigl({\rm H}_{\bE},\omega_{\rm can}\bigr)\to \bV_0\bigl(\omega_{\bE},\omega_{\rm can}\bigr)$ induces inclusions $$\omega^{k,o}\subset \bW_k^o, \qquad \omega^k\subset \bW_k.$$The formalism of VBMS allows to define $\bW_k$ in the more general settings of $p$ inert or ramified.

 Work of Katz \cite{katzpadic} implies that ${\rm H}_{\bE}$ splits canonically as the direct sum ${\rm H}_{\bE}=\omega^{1,o}\oplus \omega^{-1,o}$ (the unit root splitting). The direct summand corresponding to $\omega^{-1,o}$ is the part of ${\rm H}_{\bE}$, identified with the log crystalline cohomology of the special fiber $\bE$,  on which Frobenius is an isomorphism. It is proven in \cite[\S 3.5]{andreatta_iovita},  Formula (3), using the unit root splitting, that \begin{equation}\label{eq:directsumcompelted}\bW_k^o \cong \widehat{\bigoplus}_{i\in \N} \omega^{k-2i,o},\end{equation} where $\widehat{\oplus}$ stands for the $p$-adic completion of the infinite direct sum.

\begin{remark} For $k$ a positive integer  $\omega^{k,o}$ is identified with the $k$-th power of the invariant differentials of $\bE/\fX^{\rm ord}$,  the $k$-th symmetric power ${\rm Sym}^k({\rm H}_{\bE})$ of ${\rm H}_{\bE}$ splits canonically as ${\rm Sym}^k({\rm H}_{\bE})=\oplus_{i=0}^k \omega^{k-2i,o}$ and it is identified with a subsbace of $\bW_k^o$. This should justify the introduction of $\bW_k$ as the correct substitute for $\sym^k$ when $k$ is not a positive integer.
\end{remark} 

 Motivated by \cite[Rmk 3.39]{andreatta_iovita} and by \cite[Rmk 2.4.2 \& \S 3.5.2]{UNO} we define,  for $g\in  \bigl(\omega^{k,o}\bigr)^{U_p=0} $,  $$\nabla^\nu\bigl(g
\bigr):=\sum_{j=0}^\infty \left(\begin{array}{cc} u_\nu \\ j
\end{array} \right)\prod_{i=0}^{j-1}(u_k+u_\nu-1-i) \vartheta^{\nu-j}\bigl(g\bigr).
$$For $\nu=n\in \N$ it follows from  the $q$-expansion principle and the computation of powers of the Gauss-Manin connection on $q$-expansions in \cite[Lemma 3.38]{andreatta_iovita}, that  the displayed  formula agrees with the $n$-th iteration of the Gauss-Manin connection.  We remark that the assumption that $g\in  \bigl(\omega^{k,o}\bigr)^{U_p=0} $ is equivalent to ask that the $q$-expansion is of the form $g(q)=\sum_{p\not\vert n} a_n q^n$. Then $\vartheta^{\nu-j}\bigl(g\bigr)\in \omega^{k+2\nu -2 j,o}$ is well defined and has $q$-expansion 
 $\vartheta^{\nu-j}\bigl(g\bigr)(q)= \sum_{p\not\vert n} \bigl(\nu(n) n^{-j}\bigr) a_n q^n$. Our assumption on the weights implies, using the description (\ref{eq:directsumcompelted}),  that the infinite sum in the definition of   $\nabla^\nu\bigl(g \bigr)$ converges to an element $$\nabla^\nu\bigl(g \bigr) \in \widehat{\bigoplus}_{i\in \N} \omega^{k-2i,o}[1/p]=\bW_{k+2\nu}.$$

\begin{remark} In \cite{bertolini_darmon_prasana} only the composite $\vartheta^{\nu}\bigl(g\bigr)$ of $\nabla^\nu\bigl(g\bigr)$ to ${\rm H}^0\bigl(\fX^{\rm ord}, \omega^{k+2\nu}\bigr)$, defined by the splitting of the inclusion $\omega^{k+2\nu}\subset \bW_{k+2\nu}$, is used. One can then relax the assumption on the weights. This assumption is used only   to ensure that the product 
$\left(\begin{array}{cc} u_\nu \\ j
\end{array} \right)\prod_{i=0}^{j-1}(u_k+u_\nu-1-i)$ converges $p$-adically to $0$ as $j$ goes to $\infty$ impl
ying that $\nabla^\nu\bigl(g
\bigr)$ is well defined.  For $p$ inert or ramified the unit root splitting is {\em not} available outside the ordinary locus and we have no choice but to work with $\nabla^\nu\bigl(g
\bigr)$.
\end{remark}

We now specialize at the point $x_0$ of $\fX^{\rm ord}$ defined by the ordinary elliptic curve with $\Gamma_1(N)$-level structure $(A_0,t_0)$. Let $L$ be the field of definition of $x_0$ and $\cO_L$ its ring of integers. It is an extension of the CM field $K$. We view $x_0$ as an $\cO_L$-point of $\fX^{\rm ord}$. The $p$-adic completion $K_p$ of $K$ splits as the product of two copies of $\Q_p$ corresponding to the completion of $K$ with respect to the two primes $\mathfrak{P}$ and $\overline{\mathfrak{P}}$ over $p$. Given $\alpha\in K$ we write $\alpha\in K_{\mathfrak{P}}$ and $\overline{\alpha}\in K_{\overline{\mathfrak{P}}}$ for the two projections.  The field $K$ acts on ${\rm H}_{A_0}[1/p]$ so that, via the unit root splitting ${\rm H}_{A_0}=\omega_{A_0} \oplus \omega_{A_0}^{-1}$, $K$ acts on $\omega_{A_0}[1/p]$ via the linear action of $K_{\mathfrak{P}}$ and on $\omega_{A_0}^{-1}[1/p]$ via the linear action of $K_{\overline{\mathfrak{P}}}$. 

Let $\omega^{k+2\nu,o}\vert_{x_0}:=x_0^\ast( \omega^{k+2\nu,o})$ and $  \bW_{k+2\nu}^o\vert_{x_0}:=x_0^\ast \bigl( \bW_{k+2\nu}^o\bigr)$ and similarly for $\omega^{k+2\nu}\vert_{x_0}$ and $\bW_{k+2\nu}\vert_{x_0}$. Then $\bW_{k+2\nu}\vert_{x_0}=\bW_{k+2\nu}^o\vert_{x_0}[1/p]$. There is an action of  $1+p \cO_c$ on $\bW_{k+2\nu}\vert_{x_0}$   obtained from the $K$-acion on ${\rm H}_{A_0}[1/p]$ letting $K$ act via endomorphisms of $A_0$: indeed $\bW_{k+2\nu}^o\vert_{x_0}$ can be defined directly as $k+2\nu$-invariant functions on $\bV_0\bigl({\rm H}_{\bE},\omega_{\rm can}\bigr)\vert_{x_0}=\bV_0\bigl({\rm H}_{A_0},x_0^\ast(\omega_{\rm can})\bigr)$ and  the latter formal $\cO_L$- scheme is functorial with respect to endomorphisms of $A_0[p^\infty]$.

Via the unit root splitting  $\bW_{k+2\nu}^o\vert_{x_0}$ coincides with the $p$-adic completion of $\oplus_{i=0}^\infty \omega^{k+2\nu-2i,o}\vert_{x_0}$. The key observation is that  we can recover $\omega^{k+2\nu-2i}\vert_{x_0} \subset  \bW_{k+2\nu}\vert_{x_0}$ without using the unit root splitting but using only the CM action; indeed it is  the subspace on which $\alpha\in 1+p \cO_c$ acts via $(k+2\nu-i)(\alpha) \cdot \overline{\alpha}^{i}$.  In particular, we get a splitting $\Psi_{x_0}\colon  \bW_{k+2\nu}\vert_{x_0} \to \omega^{k+2\nu}\vert_{x_0} $ of the inclusion $\omega^{k+2\nu}\vert_{x_0} \subset  \bW_{k+2\nu}\vert_{x_0}$ and $$\vartheta^{\nu}\bigl(g\bigr)\bigl (A_0, t_0,\Omega_{\rm can}\bigr)=\Psi_{x_0}\left(\nabla^\nu\bigl(g
\bigr)\vert_{x_0} \right) (\Omega_{\rm can}).$$We then recover Definition (\ref{eq:BDPLp}) in the form: $$L_p(F,\chi^{-1})=\sum_{\fa\in {\rm Pic}(\cO_{c})}
 \chi_\nu^{-1}(\fa)  \Psi_{\fa\ast x_0}\left(\nabla^\nu\bigl(F^{[p]}
\bigr)\vert_{\fa\ast x_0} \right) \bigl(\fa\ast\Omega_{\rm can}\bigr).$$This is the formula that we generalize in the case $p$ inert or ramified. Notice that at the expense of working with $\nabla^\nu$, that imposes restrictions on the permissible weights $k$ and $\nu$, we avoided the use of the unit root splitting, substituted by the splitting induced by the CM action of $1+p \cO_c$ on the specialization of $\bW_{k+2\nu}$ at the CM points $\fa\ast x_0$. Such action is defined  using the functoriality of the formalism of VBMS relatively to endomorphisms of $A_0$. In particular, it is still available in the inert or ramified case and constitutes one of the ingredients in our proof.


\begin{thebibliography}{99}


\bibitem[AI]{andreatta_iovita}  F.~Andreatta,  A.~Iovita:
\emph{  Triple product $p$-adic L-functions associated to finite
slope $p$-adic families of modular forms}, Duke Math.~J.~{\bf 170},  1989--2083 (2021).

\bibitem[AIPHS]{halo_spectral} F.~Andreatta, A.~Iovita, V.~Pilloni: \emph{  Le Halo Spectral}, Ann.~Sci.~ENS {\bf 51},  603--656, (2018).

\bibitem[ICM18]{ICM} F.~Andreatta, A.~Iovita, V.~Pilloni: \emph{$p$-Adic variation of automorphic sheaves.} ICM 2018. Rio de Janeiro, World Scientific Publishing Co. Pte. Ltd., Hackensack, 249--276 (2018).

\bibitem[AIS]{andreatta_iovita_stevens}  F.~Andreatta,  A.~Iovita, G.~Stevens: \emph{Overconvergent modular sheaves and modular forms for ${\rm GL}_{2/F} $}, Israel J.~of Mathematics {\bf 201}, 299--359 (2014).

\bibitem[BK]{kobayashi_bannai} K.~Bannai, S.~Kobayashi: \emph{Algebraic theta functions and the $p$-adic interpolation of Eisenstein-Kronecker numbers}. Duke Math.~J.~{\bf 153},  229--295 (2010). 

\bibitem[BKY]{kobayashi_bannai_yasuda} K.~Bannai, S.~Kobayashi, S.~Yasuda: \emph{ The radius of convergence of the p-adic sigma function.} Math.~Z. {\bf 286},  751--781 (2017).

\bibitem[BCDDPR]{bertolini_etc.} M.~Bertolini, F.~Castella, H.~Darmon, S.~Dasgupta, K.~Prasanna, V.~Rotger:
{\it $p$-adic $L$-functions and Euler systems: a tale in two trilogies}. 
in ``Automorphic forms and Galois representations", London Mathematical Society Lecture Notes Series {\bf 414}, 52--101 (2014). 

\bibitem[BDP]{bertolini_darmon_prasana} M.~Bertolini, H.~Darmon, K.~Prasanna: \emph{Generalized Heegner cycles and $p$-adic Rankin $L$-series}, Duke Math.~J.~{\bf 162}, 1033--1148  (2013). 

\bibitem[BO]{berthelot_ogus} P.~Berthelot, A.~Ogus: \emph{Notes on crystalline cohomology}, Princeton University Press, Princeton, N.J., (1978).

\bibitem[BKO1]{burungale_ota_kobayashi_1} A.~Burungale, S.~Kobayashi, K.~Ota: \emph{Rubin's conjecture on local units in the anticyclotomic tower at inert primes}, Ann. of Math., (2) 194, no. 3, 943-966, (2021).

\bibitem[BKO2]{burungale_ota_kobayashi_2}  A.~Burungale, S.~Kobayashi, K.~Ota: \emph{p-Adic L-functions and rational points on CM elliptic curves at inert primes}, preprint.

\bibitem[BC]{buzzard_calegari} K.~Buzzard, F.~Calegari: \emph{The $2$-adic eigencurve is proper.} Doc. Math. (Extra Vol.), 211--232, (2006).

\bibitem[Ci]{ciperiani} M.~Ciperiani: \emph{Supersingular elliptic curves over $\Z_p$-extensions.} preprint. 


\bibitem[Co]{ColemanPrimitive} R.~Coleman: {\it A $p$-adic Shimura isomorphism and $p$-adic periods of modular forms}, in ``$p$-adic monodromy and the Birch and Swinnerton-Dyer conjecture" (Boston, MA, 1991), 21–51, Contemp. Math. {\bf 165}, Amer. Math. Soc., Providence, RI, 1994. 

\bibitem[DR1]{darmon_rotger} H.~Darmon, V.~Rotger: \emph{Diagonal cycles and Euler systems I: a $p$-adic Gross-Zagier formula},
Ann.~Sci.~ENS {\bf 47}, 779--832 (2014).

\bibitem[Fu]{fujiwara} Y.~Fujiwara: \emph{On divisibilities of special values of real analytic Eisenstein series}, J. Fac. Sci. Univ. Tokyo, Sect. 1A Math. {\bf 35}, 393--410 (1988).


\bibitem[K1]{katzpadic} N.~Katz: \emph{$p$-adic properties of modular schemes and modular forms}, In ``Modular functions of one variable III" LNM \textbf{350},  69--190, Springer, Berlin, 1973.

\bibitem[K2]{katzEisenstein} N.~Katz: \emph{$p$-adic Interpolation of Real Analytic Eisenstein Series}, Annals of Math.~{\bf 104}, 459--571 (1976). 

\bibitem[K3]{katz_invent} N.~Katz: \emph{$p$-Adic $L$-functions for CM-fields}, Invent. Math. {\bf 49}, 199--297 (1978).

\bibitem[K4]{katz_ICM} N.~Katz: \emph{$p$-Adic $L$-functions, Serre-tate local moduli, and ratios of soultions of differential equations}, Proceeding of the ICM 1978,  Academia Scientiarum Fenica, Helsinki, 365--371 (1980). 


\bibitem[Kr]{kriz} D.~Kriz: \emph{Supersingular $p$-adic $L$-functions, Maass-Shimura Operators and Waldspurger Formulas}, Ann.~of Math.~Stud.,~{\bf 212}
Princeton University Press, Princeton, NJ, xiii+258 pp. (2021). 

\bibitem[Ur]{UNO} E.~Urban: \emph{  Nearly overconvergent modular forms}, Iwasawa theory 2012, 401-441, Contrib. Math. Comput. Sci. {\bf 7}, Springer, Heidelberg, (2014).

\bibitem[Z]{zink} T.~Zink: \emph{ Cartiertheorie commutativer formaler Gruppen}, Teubner Texte zar Mathematik, {\bf 68}, Leipzig: Teubner (1984).  

\bibitem[Ye]{ye}: L.~Ye, \emph{A modular proof of the properness of the of Coleman-Mazur eigencurve.}, ArXiv, (2020)

\end{thebibliography}
\end{document}